\documentclass[final]{siamltex}

\usepackage{graphicx}
\usepackage{psfrag}
\usepackage{float}
\usepackage{amssymb,amsmath}
\usepackage{t1enc}
\usepackage[latin1]{inputenc}
\usepackage{amsfonts}
\usepackage{stmaryrd}

\usepackage{epstopdf}
\usepackage{graphics}
\usepackage{epsf}
\usepackage{xfrac}
\usepackage{amscd}
\usepackage{color} 
\usepackage{wrapfig}
\usepackage[mathscr]{euscript}
\usepackage[english]{babel}
\usepackage{epsfig}
\usepackage{dsfont}
\usepackage[center]{caption}

\everymath{\displaystyle}

\newcommand{\bD}{\pmb{D}}
\newcommand{\bC}{\pmb{C}}
\newcommand{\bR}{\pmb{R}}
\newcommand{\bA}{\pmb{A}}
\newcommand{\bB}{\pmb{B}}
\newcommand{\bF}{\pmb{F}}
\newcommand{\bG}{\pmb{G}}

\newcommand{\br}{\pmb{r}}
\newcommand{\bv}{\pmb{v}}
\newcommand{\bw}{\pmb{w}}
\newcommand{\bn}{\pmb{n}}
\newcommand{\bW}{\pmb{W}}
\newcommand{\iS}{\mathcal{S}}
\newcommand{\iT}{\mathcal{T}}

\newcommand{\iF}{\mathcal{F}}
\newcommand{\iN}{\mathcal{N}}

\newcommand{\iD}{\mathcal{D}}

\newcommand{\iR}{\mathcal{R}}
\newcommand{\iH}{\mathcal{H}}
\newcommand{\iB}{\mathcal{B}}
\newcommand{\Div}{\text{div}}
\newcommand{\half}{\frac{1}{2}}  % Jean's variables
\newcommand{\prm}{\alpha}
\newcommand{\cil}{c_i(\lambda)}
\newcommand{\ril}{\br_i(\lambda)}
\newcommand{\cie}{c_i(\eta)}
\newcommand{\rie}{\br_i(\eta)}
\newcommand{\ciz}{c_i (\zeta)}
\newcommand{\riz}{\br_i(\zeta)}
\newcommand{\cix}{c_i(\xi)}
\newcommand{\rix}{\br_i(\xi)}
\newcommand{\tSigma}{\widetilde{\Sigma}}
\newcommand{\ratioD}{\mathfrak{D}}

\newtheorem{remark}[theorem]{Remark}

\definecolor{freeblue}{rgb}{0.25,0.41,0.88}
\definecolor{darkorange}{rgb}{1.0, 0.55, 0.0}
\definecolor{mediumred-violet}{rgb}{0.73, 0.2, 0.52}
\definecolor{americanrose}{rgb}{1.0,0.01,0.24}
\definecolor{deepmagenta}{rgb}{0.8, 0.0, 0.8}
\definecolor{palatinateblue}{rgb}{0.15, 0.23, 0.89}
\definecolor{ruby}{rgb}{0.88, 0.07, 0.37}
\definecolor{shamrockgreen}{rgb}{0.0, 0.62, 0.38}

\newcommand{\CJ}[1]{{\color{black}{#1}}}
\newcommand{\CJa}[1]{{\color{black}{#1}}}
\newcommand{\JR}[1]{{\color{black}{#1}}}
\newcommand{\Phuong}[1]{{\color{black}{#1}}}
\newcommand{\MK}[1]{{\color{black}{#1}}}

\newcommand{\All}[1]{{\color{black}{#1}}}
\newcommand{\JRo}[1]{{\color{black}{#1}}}
\title{Space-Time Domain Decomposition Methods for Diffusion Problems in Mixed Formulations} 

\author{Thi-Thao-Phuong Hoang \footnotemark[1]\ \footnotemark[4]   
\and J\'er\^ome Jaffr\'e\footnotemark[1]
\and  Caroline Japhet\footnotemark[3]\ \footnotemark[1]
\and  Michel~Kern\footnotemark[1]\ 
\and  Jean E. Roberts\footnotemark[1]}

\usepackage{RR}
\usepackage{hyperref}
\RRNo{8271}
\RRdate{March 2013}
\RRauthor{
Thi Thao Phuong Hoang%
  \thanks[sfn]{INRIA Paris-Rocquencourt, project-team Pomdapi, 78153 Le Chesnay Cedex, France Emails: \texttt{Phuong.Hoang\_Thi\_Thao@inria.fr}, \hspace{0.1cm}\texttt{Jerome.Jaffre@inria.fr},~\texttt{Caroline.Japhet@inria.fr}, \texttt{Michel.Kern@inria.fr}, \texttt{Jean.Roberts@inria.fr}}% 
  \thanks{Partially supported by ANDRA, the French agency of nuclear waste management}
  \and
J\'er\^ome Jaffr\'e 
\thanksref{sfn}
\and Caroline Japhet \thanksref{sfn}%
\thanks{Universit\'e Paris 13, UMR 7539, LAGA,
  99 Avenue J-B Cl\'ement, 93430 Villetaneuse, France Email: \texttt{japhet@math.univ-paris13.fr}} 
\and Michel~Kern %
\thanksref{sfn}
\and Jean E. Roberts %
\thanksref{sfn} }
\authorhead{Hoang \& Jaffr\'e \& Japhet \& Kern \& Roberts}
\RRtitle{M\'ethode de D\'ecomposition de Domaine Espace-Temps pour les Probl\`emes de Diffusion en Formulation Mixte \vspace{0.8cm}}
\RRetitle{Space-Time Domain Decomposition Methods for Diffusion Problems in Mixed Formulations}
\RRnote{This work was partially supported by the GNR MoMaS (PACEN/CNRS, ANDRA, BRGM, CEA, EDF, IRSN)}
\RRresume{
Ce papier traite de m\'ethodes de d\'ecomposition de domaine sans recouvrement, globales en temps, 
appliqu\'ees \`a un probl\`eme de diffusion en formulation mixte. Deux approches sont consid\'er\'ees: l'une
bas\'ee sur l'op\'erateur de Steklov-Poincar\'e, et l'autre bas\'ee sur une m\'ethode de relaxation d'onde optimis\'ee (OSWR),
utilisant des conditions de transmission de type Robin. Pour chaque m\'ethode, un probl\`eme
d'interface est \'ecrit sous forme mixte, les inconnues \'etant sur les interfaces espace-temps entre les sous-domaines,
et diff\'erentes grilles en temps sont utilis\'ees, adapt\'ees aux diff\'erentes \'echelles temporelles dans les
sous-domaines. Des d\'emonstrations \`a la fois du caract\`ere bien pos\'e des probl\`emes locaux dans les sous-domaines
(intervenant dans chaque m\'ethode) et de la convergence de l'algorithm OSWR sont donn\'ees en formulation mixte.
Des r\'esultats num\'eriques pour des probl\`emes 2D avec de fortes h\'et\'erog\'en\'eit\'es sont pr\'esent\'es pour
illustrer les performances des deux m\'ethodes.
}

\RRabstract{ 
This paper is concerned with global-in-time, nonoverlapping domain decomposition methods for the mixed formulation
of the diffusion problem. Two approaches are considered: one uses the time-dependent Steklov-Poincar\'e operator
and the other uses Optimized Schwarz Waveform Relaxation (OSWR) based on Robin transmission conditions. For
each method, a mixed formulation of an interface problem on the space-time interfaces between subdomains is derived,
and different time grids are employed to adapt to different time scales in the subdomains. Demonstrations of the
well-posedness of the subdomain problems involved in each method and a convergence proof
of the OSWR algorithm are given for the mixed formulation. Numerical results for 2D problems with strong
heterogeneities are presented to illustrate the performance of the two methods.}
\RRmotcle{formulation mixte, d\'ecomposition de domaine espace-temps, probl\`eme de diffusion,
  op\'erateur de Steklov-Poincar\'e d\'ependant du temps, m\'ethode de relaxation d'onde optimis\'ee,
  maillages non-conformes en temps}
\RRkeyword{mixed formulations, space-time domain decomposition, diffusion problem,
  time-dependent Steklov-Poincar\'e operator, optimized Schwarz waveform relaxation, nonconforming time grids}
\RRprojets{POMDAPI}
\URRocq % pour ceux qui sont au centre de la France
\RCParis % Paris Rocquencourt
\begin{document}

\makeRR

\renewcommand{\thefootnote}{\fnsymbol{footnote}}

\footnotetext[1]{INRIA Paris-Rocquencourt, 78153 Le Chesnay Cedex, France~(\texttt{Phuong.Hoang\_Thi\_Thao@inria.fr},
  \texttt{Jerome.Jaffre@inria.fr}, \texttt{Michel.Kern@inria.fr},
  \texttt{Jean.Roberts@inria.fr})}
\footnotetext[3]{Universit\'e Paris 13, UMR 7539, LAGA,
  99 Avenue J-B Cl\'ement, 93430 Villetaneuse, France (\texttt{japhet@math.univ-paris13.fr}).}
\footnotetext[4]{Partially supported by ANDRA, the French agency of nuclear waste management}
\footnotetext[1]{Partially supported by GNR MoMaS.}

\renewcommand{\thefootnote}{\arabic{footnote}}

\pagestyle{myheadings}
\thispagestyle{plain}
%
% ------------------------------------------------
%
%    SECTION 1 : Introduction
%	
% ------------------------------------------------
%
\section{Introduction}
In many simulations of time-dependent physical phenomena, the domain of calculation is actually a union of several
subdomains with different physical properties and in which the time scales may be very different. In particular, this
is the case for the simulation of contaminant transport around a nuclear waste repository, where the time scales vary
over several orders of magnitude due to changes in the hydrogeological properties of the various geological layers
involved in the simulation. Consequently, it is inefficient to use a single time step throughout the entire domain.
\MK{The aim of this article is to investigate, in the context of mixed finite elements
\cite{brezzi1991mixed, RobertsThomas}, two global-in-time domain decomposition methods well-suited to
nonmatching time grids. Advantages of mixed methods include \JR{their} mass conservation property and a natural way to handle
heterogeneous and anisotropic diffusion tensors}.

The first method is a global-in-time substructuring method which \JR{uses a} Steklov-Poincar\'e type operator.
For stationary problems, this kind of method (see \cite{quarteroni1999domain,Toselli:DDM:2005,Mathew:DDM:2008})
is known to be efficient for problems with strong heterogeneity. It uses the so-called Balancing Domain Decomposition
(BDD) preconditioner introduced and analyzed in \cite{Mandel, Mandelweights}, and in \cite{CowsarBDD} for mixed
finite elements. In brief, the method "involves at each iteration the solution of a local problem with Dirichlet data,
a local problem with Neumann data and a "coarse grid" problem to propagate information globally and to insure the
consistency of the Neumann problems" \cite{CowsarBDD}.

The second method uses the Optimized Schwarz Waveform Relaxation (OSWR) approach. The OSWR algorithm is an
iterative method that computes in the subdomains \MK{over} the whole time interval, exchanging space-time boundary data
through more general (Robin or Ventcel) transmission operators in which coefficients can be optimized to improve
convergence rates. Introduced for parabolic and hyperbolic problems in \cite{Gander:1999:OCO}, \JR{it was} extended
to advection-reaction-diffusion problems with constant coefficients in \cite{VMartin}. The optimization of the Robin
(or Ventcel) parameters was analyzed in \cite{OSWR1d1,Bennequin} and extended to discontinuous coefficients
in \cite{OSWR1d2,OSWR2d}. Extensions to heterogeneous problems and non-matching time grids were introduced in
\cite{OSWR1d2,BlayoHJ}. More precisely, in \cite{BlayoHJ,halpern:2008:DGN}, discontinuous Galerkin (DG) for the
time discretization of the OSWR was introduced to handle non-conforming time grids, in one dimension with discontinuous
coefficients. This approach was extended to the bidimensional case in \cite{OSWRDG,OSWRDG2}. One of the advantages
of the DG method in time is that a rigorous analysis can be \JR{carried out} for any degree of accuracy and local time-stepping,
with different time steps in different subdomains (see \cite{OSWRDG,OSWRDG2}). A suitable time projection between
subdomains was \JR{obtained} by an optimal projection algorithm without any additional grid, as in \cite{Projection1d}.
\MK{These papers use Lagrange finite elements}. An extension to vertex-centered finite
\JR{volume} schemes and nonlinear problems is given in \cite{Haeberlein}.
The classical Schwarz algorithm for stationary problems with mixed finite elements was analyzed in
\cite{JeanRobinmixed}.\\

In this work, we extend the first \JR{method} to the case of unsteady problems and construct the time-dependent
Steklov-Poincar\'e operator. For parabolic problems, we need only the Neumann-Neumann preconditioner \cite{NNPrecond}
as there \MK{are no difficulties} concerning consistency for time-dependent Neumann problems. Of course one could make use
of the idea of the "coarse grid" to ensure \JR{a convergence rate independent} of the number of subdomains. However,
we haven't developed this idea here.
\MK{The convergence of a Jacobi iteration for the primal formulation is independently introduced and analyzed in
\cite{Kwok}.}

\MK{For the second} \JR{method}, an extension of the OSWR \MK{method} with Robin transmission
conditions to the mixed formulation is studied and a proof of convergence is given.
For each method a mixed formulation of an interface problem on the space-time
interfaces between subdomains \JR{is derived}. The well-posedness of the subdomain problems involved in the first approach
is addressed  in~\cite{Arbogast,BoffiGastaldi04},
through a Galerkin method and suitable a priori estimates.
\CJ{In this paper we present a more detailed version of the proof for Dirichlet and extend these results to
to prove the well-posedness of the
Robin subdomain problems involved in the OSWR approach.}
In \cite{Showalter10, Visintin09} demonstrations using semigroups are given
for nonlinear evolution problems. 
For strongly heterogeneous problems, it is natural to use different time steps in different subdomains
and we apply the projection algorithm in \cite{Projection1d} adapted to time discretizations to exchange
information on the space-time interfaces, \MK{for the lowest order DG method in time.}
We show the numerical behaviour of \MK{both} methods for \MK{different test cases suggested by ANDRA
for the simulation of underground nuclear waste storage}.
A preliminary version of this work was given in \cite{PhuongDD21}.

The remainder of this paper is organized as follows: in the next section we present the model problem in a mixed
formulation. We prove its well-posedness for \CJ{Dirichlet and} Robin boundary conditions in Section~\ref{sec:Robinbc}.
In Section~\ref{Sec:DD}, we
introduce the equivalent multidomain problem using nonoverlapping domain decomposition and describe the two solution
methods. A convergence proof for the OSWR algorithm for the mixed formulation is given. In Section~\ref{Sec:Time},
we consider the semi-discrete problems in time using different time grids in the subdomains.
\JR{In section~\ref{Sec:Num}, results of 2D numerical experiments}
\MK{showing that the methods preserve the order of the global scheme} \JR{are discussed}.

%
% ------------------------------------------------
%
%    SECTION 2 : Model problem (Primal form)
%	
% ------------------------------------------------
%
\section{A model problem}\label{Sec:Model}
In this section we define our model problem and show the existence and uniqueness of its solution. For an open,
bounded domain $ \Omega $ of $ \mathbb{R}^{d} \; (d=2,3) $ with Lipschitz boundary $ \partial \Omega $ and some fixed
time $ T > 0 $, we consider the following time-dependent diffusion problem
\begin{equation} \label{primal}
\omega \partial_{t} c + \nabla \cdot \left (-\bD \nabla c\right )
= f, \quad \text{in} \; \Omega \times \left (0,T\right),
\end{equation}
with boundary and initial conditions
\begin{align}
c  &= 0, \quad\; \text{on} \; \partial \Omega \times (0,T),\nonumber\\
c(\cdot , 0 )  &= c_{0}, \quad \text{in} \;  \Omega. \label{ic}
\end{align}
Here $ c $ is the concentration of \JR{a contaminant dissolved in a fluid, $ f $ the source term, $ \omega $
the porosity and $\bD$ a symmetric time independent diffusion tensor. We assume that $\omega$ is bounded
above and below by positive constants, $0 < \omega_{-} \le \omega(x) \le  \omega_{+} $, and that there exists
$\delta_{-}$ and $\delta_{+}$ positive constants such that 
$ \xi^{T} \bD^{-1}(x) \xi \geq \delta_{-} \vert \xi \vert ^{2}$, 
and $|\bD(x)\xi| \le \delta_+ |\xi|$, for a.e. $x \in \Omega$
and $\forall \xi \in \mathbb{R}^{d}.$}
For simplicity, we have imposed a homogeneous Dirichlet boundary condition on $ \partial \Omega $.
In practice, we may use non-homogeneous Dirichlet and Neumann boundary conditions for which the analysis
remains valid (see Section \ref{sec:Robinbc} for the extension to Robin boundary conditions).

We now rewrite~\eqref{primal} in an equivalent mixed form by introducing the vector field
$ \br := -\bD \nabla c $. This yields
\begin{equation} \label{mixed}
\left . \begin{array}{rll} \omega \partial_{t} c + \nabla \cdot \br
  & =f, & \text{in} \; \Omega \times \left (0,T\right ),\\
\nabla c + \bD^{-1} \br &=0,  & \text{in} \; \Omega \times \left (0,T\right ).
\end{array} \right .
\end{equation}
To write the variational formulation for~\eqref{mixed} (see~\cite{brezzi1991mixed, RobertsThomas}),
we introduce the spaces
$$
M=L^{2}\left (\Omega\right ) \, \; \text{and} \; \,  \Sigma = H\left (\Div, \Omega\right ).
$$
We multiply the first and second equations in~\eqref{mixed} by $ \mu \in M $ and $ \bv \in \Sigma  $ respectively,
then integrate over $ \Omega $ and apply Green's formula to obtain:\\

\noindent
\hspace{1mm}
\Phuong{\mbox{For a.e. $ t \in (0,T) $, find $ c(t) \in M$ and $ \br (t) \in \Sigma$ such that}}
\vspace{-2.5mm}
\begin{eqnarray} \label{variational-mixed}
%\Phuong{\mbox{For a.e. $ t \in (0,T) $, find $ c(t) \in M$ and $ \br (t) \in \Sigma$ such that}}
%   \hspace{4cm}\nonumber\\    
\left 
.\begin{array}{rll} \frac{d}{dt}(\omega c, \mu ) + (\nabla \cdot \br, \mu )
  & = (f, \mu ), & \forall \mu \in M,\\
 - (\nabla \cdot \bv, c ) + (\bD^{-1} \br, \bv )
  & = 0, & \forall \bv \in \Sigma,
\end{array}\right .
\end{eqnarray}
together with initial condition~\eqref{ic}. 

Here and in the following,
\JR{we will use the convention that if $V$ is a space of functions, then we write $\pmb{V}$ for a
  space of vector functions having each component in $V$.
We also denote by}
$ (\cdot, \cdot ) $ the inner product in $ L^{2}(\Omega) $
or $ \pmb{L^{2}(\Omega)} $ and~$ \| \cdot \| $ the $ L^{2}(\Omega) $-norm or $ \pmb{L^{2}(\Omega)} $-norm.

The well-posedness of problem~\eqref{variational-mixed} is \JR{shown} in
\MK{in~\cite{Arbogast,BoffiGastaldi04}}, \JR{with an argument}
based on a Galerkin's method and a priori estimates:
\begin{theorem}\label{thrm1} 
If $ f $ is in $ L^{2} (0,T; L^{2}(\Omega)) $ and $ c_{0} $ in $H_{0}^{1} (\Omega) $
then \JR{problem~\eqref{variational-mixed},~\eqref{ic}} has a unique solution 
$$
\CJ{(c, \br) \in } \, H^{1}(0,T; L^{2}(\Omega)) \times \left (L^{2}(0,T; H(\emph{\Div}, \Omega))
  \cap L^{\infty} (0,T; \pmb{L^{2}(\Omega)})\right ).
$$ 
Moreover, if \Phuong{$\bD$ is in $\pmb{W^{1,\infty}(\Omega)}$},
$ f$ in $H^{1}(0,T; L^{2}(\Omega)) $ and $ c_{0} $ in $H^{2}(\Omega) \cap H^{1}_{0}(\Omega) $  then 
$$
\CJ{(c, \br) \in } \, W^{1, \infty} (0,T; L^{2}(\Omega)) \times \left (L^{\infty} (0, T; H(\emph{\Div}, \Omega))
  \cap H^{1}(0,T; \pmb{L^{2}(\Omega)})\right ).
$$
\end{theorem}
\CJ{
\textit{Remark.}
We give the proof of Theorem~{thrm1} in the finite dimensional setting since some technical points (those involving
$\partial_t \pmb{r}$, or \pmb{r} at time $t=0$) can only be defined by their finite dimensional Galerkin
approximation.
This is not surprising given the differential-algebraic structure of system~\eqref{Robinvar}:
the second equation has no time derivative. In DAE theory it is well known that the algebraic equations
have to be differentiated \JR{a number of times} (this is what defines the index),
and that this imposes compatibility conditions between the initial data
(note that $\pmb{r}(0)$ is not given).
The index has been extended to PDEs, see for instance~\cite{Martinson:ADI:2000}.

The proof of Theorem~\ref{thrm1}
is carried out in several steps: in Lemma~\ref{lem:approx} we first construct solutions
of certain finite-dimensional approximations of~\eqref{variational-mixed}, then we derive
suitable energy estimates in Lemma~\ref{lem:estimates1} and prove the first part of the theorem. 
The higher regularity of the solution is obtained from the estimates given in
Lemma~\ref{lem:estimateDirichlet2}. 
%
%
%	GALERKIN APPROXIMATIONS
%
%

We need first to introduce some notations:
Let $ \{ \mu_{n} \mid n\in \mathbb{N} \} $ be a Hilbert basis of $ M $ and $ \{ \bv_{n} \mid n\in \mathbb{N} \} $
be a Hilbert basis of $ \Sigma $. 
For each pair of positive integers $ n $ and $ m $, we denote by $ M_{n} $ the finite dimensional subspace
spanned by $ \{ \mu_{i} \}_{i=1}^{n} $, and  $ \Sigma_{m} $ the finite dimensional subspace spanned by
$ \{ \bv_{i} \}_{i=1}^{m} $. Now let $ c_{n}: [0,T] \rightarrow M_{n} $ and $ \br_{m}: [0,T] \rightarrow \Sigma_{m} $
be the solution of the following problem 
\begin{equation} \label{appro-mixed}
\left 
.\begin{array}{rll} (\omega \partial_{t} c_{n}, \mu_{i} ) + (\nabla \cdot \br_{m}, \mu_{i} )
  & = (f(t), \mu_{i} ), & \forall i=1, \hdots, n, \\
- (\nabla \cdot \bv_{j}, c_{n} ) + (\bD^{-1} \br_{m}, \bv_{j} )
  & = 0, & \forall j=1, \hdots, m, \\
\end{array}\right .
\end{equation}
with
\begin{equation} \label{ic-n}
(c_{n}(0), \mu_{i}) = (c_{0}, \mu_{i}), \quad \forall i=1, \hdots, n.
\end{equation}
%
% --- Lemma : Construction of approximate solutions ----
%
\begin{lemma} \label{lem:approx}
\emph{(Construction of approximate solutions)} For each pair $ (n,m) \in \mathbb{N}^{2}$, $n,m \geq 1 $,
there exists a unique solution $ (c_{n}, \br_{m}) $ to problem~\eqref{appro-mixed}.
\end{lemma}
\begin{proof}
We introduce the following notations \vspace{5pt}\\
$ (\bF_{n}(t))_{i}=(f(t),\mu_{i}) , \;
  (\bC_{0})_{i}=(c_{0},\mu_{i}), \;
  (\bW_{n})_{ij} = (\omega \mu_{j}, \mu_{i}),  \; \forall \, 1 \leq i, j \leq n $, \vspace{6pt}
\\
$ (\bA_{m})_{ij} = (\bD^{-1} \bv_{j}, \bv_{i}), \; \forall \, 1 \leq i, j \leq m,  $
$ (\bB_{nm})_{ij} = (\nabla \cdot \bv_{j}, \mu_{i}), \; \forall \, 1 \leq i \leq n, 1 \leq j \leq m.$\vspace{5pt} 
\\
We also denote by $ \bC_{n} (t) $ the vector of degrees of freedom of $ c_{n} (t) $
with respect to the basis $ \{ \mu_{i} \}_{i=1}^{n} $ and  $ \bR_{m}(t) $ that of $ \br_{m}(t) $
with respect to the basis $ \{ \bv_{i}\}_{i=1}^{m} $. 
With this notation, (\ref{appro-mixed}) may be rewritten as
\begin{subequations} \label{appro-matrix}
\begin{align}
\bW_{n} \frac{d \bC_{n}}{dt} (t) + \bB_{nm} \bR_{m}(t) & = \bF_{n}(t), \label{1-matrix}\\
-\bB_{nm}^{T} \bC_{n} (t) + \bA_{m} \bR_{m}(t) & = 0, \label{2-matrix}\\
\bC_{n}(0) & = \bC_{0}.  \label{ic-matrix}
\end{align}
\end{subequations}
As $ \bA_{m} $ is a symmetric and positive definite square matrix of size $ m $
(because of the assumptions concerning $ \bD $), $ \bA_{m} $ is invertible. Thus (\ref{2-matrix}) implies
\begin{equation} \label{UbyP}
\bR_{m}(t) = \bA_{m}^{-1} \bB_{nm}^{T} \bC_{n} (t).
\end{equation}
Substituting (\ref{UbyP}) into (\ref{1-matrix}) and as $ \bW_{n} $ is invertible, we obtain
\begin{equation} \label{ODEs}
\frac{d \bC_{n}}{dt}(t) + \bW_{n}^{-1}\bB_{nm} \bA_{m}^{-1}\bB_{nm}^{T} \bC_{n}(t)
  = \bW_{n}^{-1}\bF_{n}(t),  \quad \text{for a.e.} \; t \in [0,T].
\end{equation}
This is a system of $ n $ linear ODEs of order $ 1 $ with initial condition~(\ref{ic-matrix}).
Hence, there exists a unique function $ \bC_{n}\in \left (C([0,T])\right )^{n} $ with
$ \frac{d \bC_{n}}{dt} \in \left (L^{2}(0,T)\right )^{n}$ satisfying~(\ref{ODEs}) and~(\ref{ic-matrix})
(see~\cite{evans1998partial}). From~(\ref{UbyP}) we obtain $ \bR_{m} \in \left (C([0,T])\right )^{m} $
such that $ \frac{d\bR_{m}}{dt}~\in~\left (L^{2}(0,T)\right )^{m} $ and then  $ (c_{n}, \br_{m}), $
which is the unique solution to~(\ref{appro-mixed}).  
\qquad
\end{proof} 

%
%
% ENERGY ESTIMATES
%
%
In the next step, we derive some suitable a priori estimates similar to those given in \cite{Arbogast} but in a more detailed manner.

\begin{lemma} \label{lem:estimates1}
%\emph{(Estimates)}
There exists a constant $C$ independent of $n$ and $m$ such that  
\begin{multline*}% \label{estimates}
\| c_{n} \|_{L^{\infty}(0,T; L^{2}(\Omega))} + \| \partial_{t} c_{n} \|_{L^{2}(0,T; L^{2}(\Omega))} + \| \br_{m} \|_{L^{\infty} (0,T; \mathbf{L^{2}(\Omega)})} + \| \br_{m} \|_{L^{2}(0,T; H(\emph{\Div}, \Omega))} \\
\hspace{2.2cm} \leq C (\|c_{0}\|_{H^{1}_{0}(\Omega)} + \| f\| _{L^{2}(0,T; L^{2}(\Omega))}), \quad
\forall  n, m \geq 1 . 
\end{multline*}
\end{lemma}
\begin{proof}
We prove this lemma by deriving successively the estimates on $ c_{n} $, $ \partial_{t} c_{n} $ and $ \br_{m} $,
and finally on $ \nabla \cdot \br_{m} $ for the $ H(\Div, \Omega) $-norm. 
%
%
%
%				% ++++++++++++++++++++%
\\
$\bullet$ Let $ n, m \geq 1 $ and take $ c_{n}(t)  \in M_{n} $ and $ \br_{m}(t) \in \Sigma_{m} $ as the test functions
in~(\ref{appro-mixed})
\begin{equation*}
\begin{array}{rl}
(\omega \partial_{t} c_{n}, c_{n}) + (\nabla \cdot \br_{m}, c_{n}) &= (f, c_{n}), \\
- (\nabla \cdot \br_{m}, c_{n})+(\bD^{-1} \br_{m}, \br_{m})  & = 0.
\end{array}
\end{equation*}
Adding these two equations, we obtain
$$ (\omega \partial_{t} c_{n}, c_{n})+ (\bD^{-1} \br_{m}, \br_{m} ) = (f,c_{n}).
$$
Using the properties of $ \omega $ and $ \bD $, and applying the Cauchy-Schwarz inequality, we get
\begin{align*}
(\omega \partial_{t} c_{n}, c_{n}) =\frac{1}{2}\frac{d}{dt} (\omega c_{n}(t), c_{n}(t)) & \geq \frac{\omega_{-}}{2}\frac{d}{dt} \| c_{n}(t) \|^{2}, \\
(\bD^{-1}\br_{m} (t), \br_{m}(t)) & \geq \delta_{-} \| \br_{m} (t)\|,\\
(f (t) ,c_{n}(t)) & \leq \| f(t)\| \, \| c_{n}(t) \| \leq \frac{1}{2\omega_{-}}\| f(t)\|^{2} +\frac{\omega_{-}}{2} \| c_{n}(t)\|^{2}.
\end{align*}
As $ \omega_{-} >0 $, we deduce that
$$ \frac{d}{dt} \|c_{n}(t)\|^{2} + \frac{2\delta_{-}}{\omega_{-}} \|\br_{m} (t) \|^{2} \leq \frac{1}{\omega_{-}^{2}} \| f (t) \|^{2} + \| c_{n}(t) \|^{2}.
$$
Integrating this inequality over $ (0,t) $ for $ t \in [0,T] $, we find
\begin{equation} \label{gronwall1}
\| c_{n}(t)\|^{2} +  \frac{2\delta_{-}}{\omega_{-}}  \int_{0}^{t} \| \br_{m} (s) \|^{2} ds  \leq \| c (0) \|^{2} +\frac{1}{\omega_{-}^{2}}  \int_{0}^{t} \| f(s)\|^{2} ds + \int_{0}^{t} \| c_{n}(s) \|^{2} ds,  
 %& \leq ( \|c_{0}\|^{2} + \frac{1}{\omega_{-}^{2}} \| f\|^{2}_{L^{2}(0,T; L^{2}(\Omega))}) + \int_{0}^{t} \| c_{n}(s)\|^{2} ds, \label{gronwall1}
 \end{equation}
since $ \| c_{n} (0) \|^{2}  = \sum_{i=1}^{n} (c_{0}, \mu_{i})^{2} \leq \sum_{i=1}^{\infty} (c_{0}, \mu_{i})^{2} = \| c_{0}\|^{2} $. \vspace{5pt}\\
Thus (\ref{gronwall1}) implies
$$ \| c_{n}(t)\|^{2}  \leq ( \|c_{0}\|^{2} +  \frac{1}{\omega_{-}^{2}}  \| f\|^{2}_{L^{2}(0,T; L^{2}(\Omega))}) + \int_{0}^{t} \| c_{n}(s)\|^{2} ds. $$
Applying Gronwall's lemma, there exists $ C $ independent of $ n $ or $ m $ such that
\begin{equation} \label{estimate-c}
\| c_{n}\|^{2}_{L^{\infty}(0,T; L^{2}(\Omega))}  \leq C ( \|c_{0}\|^{2} + \| f\|^{2}_{L^{2}(0,T; L^{2}(\Omega))}),
\end{equation}
%
%
%
%
%				% ++++++++++++++++++++%				
\\
$\bullet$ Now we derive the estimate for $ \partial_{t} c_{n} $: Taking $ \partial_{t} c_{n} \in M_{n}  $ as the test function in the first equation of~(\ref{appro-mixed}), we obtain
\begin{equation} \label{1ii}
 (\omega \partial_{t} c_{n}, \partial_{t} c_{n}) + (\nabla \cdot \br_{m}, \partial_{t} c_{n} )=(f, \partial_{t} c_{n}). 
\end{equation}
Differentiating the second equation of~(\ref{appro-mixed}) with respect to $ t $, we find
\begin{equation} \label{2ii}
- (\nabla \cdot \bv, \partial_{t} c_{n}) + (\bD^{-1} \partial_{t} \br_{m}, \bv )  = 0, \quad \forall \bv \in \Sigma_{m}.
\end{equation}
Then we take $ \br_{m} $ as the test function in~(\ref{2ii})
\begin{equation}\label{3ii}
(\bD^{-1} \partial_{t} \br_{m}, \br _{m} ) - (\nabla \cdot \br _{m}, \partial_{t} c_{n}) = 0.
\end{equation}
Adding (\ref{1ii}) and (\ref{3ii}), we see that
$$ (\omega \partial_{t} c_{n}, \partial_{t} c_{n}) +(\bD^{-1} \partial_{t} \br_{m}, \br_{m} ) = (f, \partial_{t} c_{n}).
$$
As $ \bD $ is symmetric and positive definite, by applying the Cauchy-Schwarz inequality to the right hand side as well as using the property of $ \omega $, we obtain
\begin{equation} \label{4ii}
\omega_{-} \| \partial_{t} c_{n} (t)\|^{2} + \frac{d}{dt} \| \sqrt{\bD^{-1}}\br_{m} (t) \|^{2}  \leq \frac{1}{\omega_{-}}\| f(t) \|^{2}.
\end{equation}
Integrating (\ref{4ii}) over $ (0,t) $ for $ t \in [0,T] $, we find
\begin{equation} \label{5ii}
\omega_{-} \int_{0}^{t} \| \partial_{t} c_{n} (s)\|^{2} ds + \|\sqrt{\bD^{-1}} \br_{m} (t) \|^{2}  \leq \|\sqrt{\bD^{-1}} \br_{m} (0)\|^{2} + \frac{1}{\omega_{-}}\int_{0}^{t} \| f(s)\|^{2} ds.
\end{equation}
To bound $ \| \br_{m}(0) \| $, we take $ \br_{m} \in \Sigma_{m} $ as the test function in the second equation of~(\ref{appro-mixed}) and let $ t=0 $
\begin{equation} \label{7ii}
(\bD^{-1} \br_{m} (0), \br_{m} (0)) = (\nabla \cdot \br_{m} (0), c_{n}(0)).
\end{equation}
Noting that (\ref{7ii}) holds for all $ n, m \geq 1 $, we bound the left-hand side as before and let $ n \rightarrow \infty $. Since $ c_{n}(0) \rightarrow c_{0} $ in $ L^{2}(\Omega) $ and $ c_{0} \in H^{1}_{0}(\Omega) $, we have by Green's formula
\begin{equation*}
\delta_{-} \| \br_{m}(0) \|^{2} \leq (\nabla \cdot \br_{m} (0), c_{0})=(\br_{m} (0),-\nabla c_{0})  \leq \| \br_{m}(0) \| \; \| \nabla c_{0} \|. %_{\mathbf{L^{2}(\Omega)}}, 
\end{equation*}
Thus
\begin{equation} \label{8ii}
\| \br_{m}(0) \| \leq C \| c_{0}\|_{H^{1}_{0}(\Omega)}.
\end{equation}
This along with (\ref{5ii}) yields
\begin{equation} \label{estimate-ct}
\| \partial_{t} c_{n} \|^{2}_{L^{2}(0,T; L^{2}(\Omega))} + \| \br_{m} \|^{2}_{L^{\infty} (0,T; \mathbf{L^{2}(\Omega)})} \leq C (\| c_{0}\|^{2}_{H_{0}^{1}(\Omega)} + \| f \|^{2}_{L^{2}(0,T; L^{2}(\Omega))}), \; \forall n, m \geq 1. 
\end{equation}
%
%
%
%
				% ++++++++++++++++++++%		
There only remains to show that $ \| \nabla \cdot \br_{m}\|_{L^{2}(0,T; \mathbf{L^{2}(\Omega)})} $ is bounded. 
\\
$\bullet$ Fixing $ m \geq 1 $, as $ \nabla \cdot \br_{m} (t) \in  M $ we can write
\begin{equation} \label{1iii}
\nabla \cdot \br_{m} (t) = \sum_{i=1}^{\infty} \xi_{m}^{i}(t) \mu_{i}, \; \text{for a.e. $ t \in (0,T) $}, 
\end{equation}
where $  \xi_{m}^{i}(t) = (\nabla \cdot \br_{m} (t) , \mu_{i})$. 
Now we fix $ n \geq 1 $ and multiply the first equation of~(\ref{appro-mixed}) by $ \xi_{m}^{i} (t) $,
sum over~$i=1,\hdots,n$, we see that
\begin{align}\label{2iii}
 (\nabla \cdot \br_{m},  \sum_{i=1}^{n} \xi_{m}^{i} \mu_{i} )
%  &= (f-\omega \partial_{t} c_{n}, \sum_{i=1}^{n} \xi_{m}^{i} \mu_{i})
%  \nonumber\\
 &\leq \frac{1}{2}  ( \| f \| + C\| \partial_{t} c_{n}\| )^{2} + \frac{1}{2}\| \sum_{i=1}^{n} \xi_{m}^{i} \mu_{i} \|^{2}.  
% & \leq  ( \| f \|^{2} + \| \partial_{t} c_{n}\| ^{2}) + \frac{1}{2}\| \sum_{i=1}^{n} \xi_{m}^{i} \mu_{i} \| ^{2}. \label{2iii}
\end{align}
Integrating with respect to time and recalling (\ref{estimate-ct}), we find
\begin{equation*}
\int_{0}^{T} (\nabla \cdot \br_{m},  \sum_{i=1}^{n} \xi_{m}^{i} \mu_{i} )dt \leq   C  ( \| f \|_{L^{2}(0,T; L^{2}(\Omega))}^{2} + \| c _{0}\|_{H^{1}_{0}(\Omega)} ^{2}) + \frac{1}{2} \int_{0}^{T} \| \sum_{i=1}^{n} \xi_{m}^{i} \mu_{i} \| ^{2} dt.
\end{equation*}
Let $ n \rightarrow \infty $ and recall (\ref{1iii}), we obtain 
$$
\int_{0}^{T} \| \nabla \cdot \br_{m}\| ^{2} dt \leq  C ( \| f \|_{L^{2}(0,T; L^{2}(\Omega))}^{2} + \| c _{0}\|_{H^{1}_{0}(\Omega)} ^{2}) + \frac{1}{2} \int_{0}^{T} \| \nabla \cdot \br_{m}\| ^{2} dt.
$$
Thus
\begin{equation*} 
 \| \nabla \cdot \br_{m} \|_{L^{2}(0,T; L^{2}(\Omega))}^{2} \leq C (\| f\|^{2}_{L^{2}(0,T; L^{2}(\Omega))} + \| c _{0}\|^{2}_{H^{1}_{0}(\Omega)}).
\end{equation*}
On the other hand, by recalling inequality (\ref{gronwall1}) with $ t=T $ and by (\ref{estimate-c}), we find
\begin{equation*}  %\label{estimate-r}
\| \br_{m} \|_{L^{2}(0,T; \mathbf{L^{2}(\Omega)})}^{2} \leq C ( \|c_{0}\|^{2} + \| f\|^{2}_{L^{2}(0,T; L^{2}(\Omega))}). 
\end{equation*}
Hence, 
\begin{align*}
\| \br_{m} \|^{2}_{L^{2}(0,T; H(\Div, \Omega))} %& = \int_{0}^{T} \| \br_{m} (t) \|^{2}_{H(\Div, \Omega)} dt = \int_{0}^{T} ( \| \br_{m} (t) \|_{\mathbf{L^{2}(\Omega)}}^{2} + \| \nabla \cdot \br_{m} (t) \|^{2} ) dt \notag \\
& = \| \br_{m} \|^{2}_{L^{2}(0,T; (L^{2}(\Omega))^{2})} + \| \nabla \cdot \br_{m} \|^{2}_{L^{2}(0,T; L^{2}(\Omega))} \\
&\leq C (\| f\|^{2}_{L^{2}(0,T; L^{2}(\Omega))} + \| c _{0}\|^{2}_{H^{1}_{0}(\Omega)}), \quad \forall m \geq 1, % \label{estimate-fullr}
\end{align*}
which ends the proof of Lemma \ref{lem:estimates1}.
\qquad
\end{proof}
%
%
% WEAK SOLUTION
%
%
%\begin{lemma} \label{lem:existence}
%\emph{(Existence)} There exists a solution $ (c, \br) $ in $ H^{1}(0,T; L^{2}(\Omega)) \times L^{2}(0,T; H(\emph{\Div}, \Omega)) \cap L^{\infty} (0,T; \mathbf{L^{2}(\Omega)})  $ of problem (\ref{mixed}).
%\end{lemma}

We now prove the first part of Theorem~\ref{thrm1}: there exists a unique solution $ (c, \br) $ in $ H^{1}(0,T; L^{2}(\Omega)) \times L^{2}(0,T; H(\emph{\Div}, \Omega)) \cap L^{\infty} (0,T; \mathbf{L^{2}(\Omega)})  $ of problem (\ref{mixed}).
\begin{proof} The proof of the first part of Theorem~\ref{thrm1} follows the following steps:.\\
$\bullet$ Lemma \ref{lem:estimates1} implies that for the sequences $ \{c_{n}\}_{n=1}^{\infty} $
  % in $ L^{2}(0,T; L^{2}(\Omega)) $ and  $  \{\br_{m}\}_{m=1}^{\infty} $ in $ L^{2}(0,T; H(\Div, \Omega)) $ 
and~$\{\br_{m}\}_{m=1}^{\infty} $ defined by (\ref{appro-mixed}) and (\ref{ic-n}),  $ \{c_{n}\}_{n=1}^{\infty} $ is bounded
in $ L^{2}(0,T; L^{2}(\Omega)) $, $ \{\partial_{t} c_{n}\}_{n=1}^{\infty} $ is bounded in $ L^{2}(0,T; L^{2}(\Omega)) $ and
$ \{\br_{m}\}_{m=1}^{\infty} $ is bounded in $ L^{2}(0,T; H(\Div, \Omega)) \cap L^{\infty}(0,T; \mathbf{L^{2}(\Omega)})$. 
Thus, there exist subsequences, still denoted by $ \{c_{n} \}_{n=1}^{\infty} $ and  $ \{\br_{m}\}_{m=1}^{\infty} $ and functions $ c \in L^{2}(0,T; L^{2}(\Omega)) $ with $ \partial_{t} c  \in L^{2}(0,T; L^{2}(\Omega))$
and $ \br \in L^{2} (0,T; H(\Div, \Omega)) \cap L^{\infty} (0,T; \mathbf{L^{2}(\Omega)}) $ 
such that
\begin{equation} \label{weakconv}
\begin{array}{l}
c_{n} \rightharpoonup c  \; \text{in} \;   L^{2}(0,T; L^{2}(\Omega)), \\
\partial_{t} c_{n} \rightharpoonup \partial_{t} c \; \text{in} \;   L^{2}(0,T; L^{2}(\Omega)), \\
\br_{m} \rightharpoonup \br \; \text{in} \;  L^{2}(0,T; H(\Div,\Omega)) .
\end{array}
\end{equation}
$\bullet$ Next let $ \eta \in C^{1}([0,T]; M_{n_{0}}) $, $ \bw \in C^{1} ([0,T]; \Sigma_{m_{0}})$
for $ n_{0}, m_{0} \geq 1 $.   
We choose $ n \geq n_{0} $ and $ m \geq m_{0} $, take $ \eta $ and $ \bw $ as the test functions in (\ref{appro-mixed}) and then integrate with respect to time 
\begin{equation}  \label{1stn}
\begin{array}{rl}
\int_{0}^{T} (\omega \partial_{t} c_{n} , \eta )  +  (\nabla \cdot \br_{m}, \eta ) dt & = \int_{0}^{T} (f, \eta) dt,  \\
\int_{0}^{T} - (\nabla \cdot \bw, c_{n}) + (\bD^{-1} \br_{m}, \bw) dt &= 0.
\end{array}
\end{equation}
Because of the weak convergence in (\ref{weakconv}), we also have
\begin{equation} \label{atlimit}
\left . \begin{array}{rl}
\int_{0}^{T} (\omega \partial_{t} c, \eta )  +  (\nabla \cdot \br, \eta ) dt &= \int_{0}^{T} (f, \eta) dt, \\
\int_{0}^{T}  - (\nabla \cdot \bw, c) +(\bD^{-1} \br, \bw) dt &= 0.
\end{array}\right .
\end{equation}
Since the spaces of test functions $ \eta, \bw $ are dense in $ L^{2}(0,T; M) $ and $L^{2}(0,T; \Sigma) $ respectively, it follows from (\ref{atlimit}) that (\ref{variational-mixed}) holds for a.e. $ t \in (0,T) $ (see \cite{evans1998partial}). \vspace{5pt} \\
$\bullet$ There remains to show that $ c(0) = c_{0} $. Toward this end, we take $ \eta \in C^{1}([0,T]; M_{n_{0}}) $ with $ \eta (T) = 0 $. It follows from the first equation of (\ref{atlimit}) that
\begin{equation} \label{c01}
-\int_{0}^{T} (\omega\partial_{t} \eta, c) + (\nabla \cdot \br , \eta) dt = \int_{0}^{T} (f, \eta) dt + (\omega c(0), \eta(0)).
\end{equation}
Similarly, from the first equation of (\ref{1stn}) we deduce
\begin{equation*}
-\int_{0}^{T} (\omega\partial_{t} \eta, c_{n}) + (\nabla \cdot \br_{m} , \eta) dt = \int_{0}^{T} (f, \eta) dt + (\omega c_{n}(0), \eta(0)).
\end{equation*}
Using (\ref{weakconv}), we obtain
\begin{equation}\label{c02}
-\int_{0}^{T} (\omega\partial_{t} \eta, c) + (\nabla \cdot \br , \eta) dt = \int_{0}^{T} (f, \eta) dt + (\omega c_{0}, \eta(0)),
\end{equation}
since $ c_{n}(0) \rightarrow c_{0} $ in $ L^{2}(\Omega) $. As $ \eta(0) $ is arbitrary, by comparing (\ref{c01}) and (\ref{c02}) we conclude that $ c(0) = c_{0} $. 
\\
 $\bullet$ For the uniqueness, as the equations are linear, it suffices to check that $ c =0 $ and $\br = 0 $ for $ f =0 $ and $ c_{0} = 0  $. To prove this, we set $ \mu = c $ and $ \bv = \br $ in (\ref{variational-mixed}) (for $ f = 0 $) and add the two resulting equations:
\begin{equation*}
\frac{1}{2}\frac{d}{dt} (\omega c, c) + (\bD^{-1} \br, \br) = 0.
\end{equation*}
Using the property of $ \omega $ and the fact that $ (\bD^{-1}  \br, \br) \geq \delta_{-} \| \br \|^{2} \geq 0 $, then integrating with respect to $ t $ we see that
\begin{equation*} % \label{unique}
\omega_{-} \| c(t) \|^{2} + 2 \delta_{-} \int_{0}^{t} \| \br (s) \|_{\mathbf{L^{2}(\Omega)}}^{2} ds \leq 0, \quad \text{for a.e. } t \in (0,T), 
\end{equation*}
where $ c(0) = c_{0}= 0 $. Thus $ c = 0$ and $\br = 0 $ for a.e. $ t \in (0,T) $. 
\qquad
\end{proof}
%
%
% -----	Estimates for higher regularity --------
%	
%

We now prove the second part of Theorem~\ref{thrm1}.
The higher regularity of the solution to (\ref{mixed}) is obtained by using the following lemma.
\begin{lemma} \label{lem:estimateDirichlet2}
\emph{(Estimates for improved regularity)} Assume that $\bD$ is in $\pmb{W^{1,\infty}(\Omega)}$,
$ c_{0}$ in $H^{2}(\Omega)~\cap~H^{1}_{0}(\Omega) $
and $ f $in $ H^{1}(0,T; L^{2}(\Omega)) $ then 
\begin{align*}
\| \partial_{t} c\|_{L^{\infty} (0,T; L^{2}(\Omega))} + \| \br \| _{L^{\infty} (0, T; H (\emph{\Div}, \Omega))} + \| \partial_{t} \br \|_{L^{2}(0,T; \mathbf{L^{2}(\Omega)})} \\
\leq C (\| f\|_{H^{1}(0,T; L^{2}(\Omega))} + \| c_{0}\|_{H^{2}(\Omega)}).
\end{align*}
\end{lemma}
\begin{proof}
%\begin{enumerate}
%\item[\textbf{1.}]
As $ f \in H^{1}(0,T; L^{2}(\Omega)) $, the solutions of the ODE system (\ref{appro-matrix}) are more regular in time than before (i.e. up to second-order time derivatives). 
  
  Let $ n, m \geq 1 $. First, we differentiate the first equation of (\ref{appro-mixed}) with respect to $ t $
\begin{equation*} \label{1s}
(\omega \partial_{tt} c_{n}, \mu_{i}) + (\nabla \cdot \partial_{t} \br_{m}, \mu_{i}) = (\partial_{t} f, \mu_{i}), \quad \forall i=1, \hdots, n,
\end{equation*}
then we take $ \partial_{t} c_{n} $ as the test function
\begin{equation} \label{2s}
(\omega \partial_{tt} c_{n},\partial_{t} c_{n}) + (\nabla \cdot \partial_{t} \br_{m}, \partial_{t} c_{n}) = (\partial_{t} f, \partial_{t} c_{n}).
\end{equation}
Similarly, we differentiate the second equation of (\ref{appro-mixed}) with respect to $ t $
\begin{equation*} %\label{3s}
(\bD^{-1} \partial_{t} \br_{m}, \bv_{i}) - (\nabla \cdot v_{i}, \partial_{t} c_{n}) = 0, \quad \forall i=1, \hdots, m,
\end{equation*}
and take $ \partial_{t} \br_{m} $ as the test function
\begin{equation} \label{4s}
(\bD^{-1} \partial_{t} \br_{m},\partial_{t} \br_{m}) - (\nabla \cdot \partial_{t} \br_{m}, \partial_{t} c_{n}) = 0.
\end{equation}
Adding (\ref{2s}) and (\ref{4s}), we find
\begin{equation*}
(\omega \partial_{tt} c_{n},\partial_{t} c_{n}) +(\bD^{-1} \partial_{t} \br_{m},\partial_{t} \br_{m})=(\partial_{t} f, \partial_{t} c_{n}).
\end{equation*}
Bounding $ (\bD^{-1} \partial_{t} \br_{m},\partial_{t} \br_{m}) \geq \delta_{-} \|\partial_{t} \br_{m}\|^{2} $, using the assumption about $ \omega $ and applying the Cauchy-Schwarz inequality, we obtain
%$$ \frac{\omega_{-}}{2} \frac{d}{dt} \| \partial_{t} c_{n} \|^{2} + \delta_{-} \| \partial_{t} \br_{m} \| ^{2} \leq  \frac{1}{2 \omega_{-}}\| \partial_{t} f \|^{2} + \frac{\omega_{-}}{2} \| \partial_{t} c_{n} \|^{2},
%$$
%or
$$ \frac{d}{dt} \| \partial_{t} c_{n} \|^{2} + \frac{2\delta_{-}}{\omega_{-}} \| \partial_{t} \br_{m} \| ^{2} \leq \frac{1}{\omega_{-}^{2}} \| \partial_{t} f \|^{2} + \| \partial_{t} c_{n} \|^{2}.
$$
%Integrate over $ (0,t) $ for $ t \in [0,T] $
For each $ t \in [0,T] $, we may integrate over $ (0,t) $ to obtain
\begin{eqnarray}\label{6s1}
\hspace{-0.4cm}\| \partial_{t} c_{n} (t) \|^{2} + \frac{2\delta_{-}}{\omega_{-}} \int_{0}^{t}\| \partial_{t} \br_{m} \| ^{2} ds
  \leq \| \partial_{t} c_{n}(0) \|^{2} + \frac{1}{\omega_{-}^{2}} \int_{0}^{t}\| \partial_{t} f \|^{2} ds
  + \int_{0}^{t}\| \partial_{t} c_{n} \|^{2}ds.  \hspace{5mm}
\end{eqnarray}
In order to bound $ \| \partial_{t} c_{n}(0) \| $, we use the first equation
of (\ref{appro-mixed}) (with $ \partial_{t} c_{n} $ as the test function, at~$t=0$) 
to obtain
$$
\|\partial_t c_n (0) \| \leq C (\|\nabla \cdot \br_m(0)\|+\|f(0)\|).
$$
Using the second equation of (\ref{appro-mixed}) at $t=0$, and then let $ n \rightarrow \infty $
to get
$$
(\bD^{-1} \br_{m} (0)+\nabla c_{0},\bv) = 0, \quad \forall \bv \in \Sigma_m.
$$
Thus, using density argument and  $c_0 \in H^1_0(\Omega) \cap H^2(\Omega)$, we obtain
$\bD^{-1} \br_{m} (0)=-\nabla c_{0} $ in $H^1(\Omega)$.
Then, we bound
\begin{equation}\label{eq:estim_dtcn0}
 \| \partial_{t} c_{n}(0) \|^{2} \leq C (\| c_{0}\|^{2}_{H^{2}(\Omega)} + \| f (0) \|^{2}).
\end{equation}
Replacing \eqref{eq:estim_dtcn0} in \eqref{6s1}, we obtain
\begin{align}\label{6s}
\| \partial_{t} c_{n} (t) \|^{2} + \frac{2\delta_{-}}{\omega_{-}} \int_{0}^{t} & \| \partial_{t} \br_{m} \| ^{2} ds \nonumber\\[-1mm]
  &\leq C (\| c_{0} \|^{2}_{H^{2}(\Omega)}+\| f \|^{2}_{H^{1}(0,T; L^{2}(\Omega))})
  + \int_{0}^{t}\| \partial_{t} c_{n} \|^{2}ds. 
\end{align}
It now follows from (\ref{6s}) and Gronwall's lemma that
\begin{eqnarray} \label{estimate-ctinfty}
\hspace{-6mm}
\| \partial_{t} c_{n}\|^{2}_{L^{\infty}(0,T; L^{2}(\Omega))} + \frac{2\delta_{-}}{\omega_{-}} \| \partial_{t} \br_{m} \|^{2}_{L^{2}(0,T; \mathbf{L^{2}(\Omega)})}
\leq  C (\| c_{0} \|^{2}_{H^{2}(\Omega)}+\| f \|^{2}_{H^{1}(0,T; L^{2}(\Omega))}).\hspace{3mm}
\end{eqnarray}
Recalling (\ref{2iii}) and using (\ref{estimate-ctinfty}), we obtain
\begin{equation*}
(\nabla \cdot \br_{m} , \sum_{i=1}^{n} \xi_{m}^{i} \mu_{i}) \leq C (\| c_{0} \|^{2}_{H^{2}(\Omega)}+\| f \|^{2}_{H^{1}(0,T; L^{2}(\Omega))})+ \frac{1}{2} \| \sum_{i=1}^{n} \xi_{m}^{i} \mu_{i} \|^{2}.
\end{equation*}
Then, let $ n \rightarrow \infty $, we see that
\begin{equation*} % \label{9s}
\| \nabla \cdot \br_{m} \| ^{2}_{L^\infty(0,T;L^2(\Omega))}
  \leq C (\| c_{0} \|^{2}_{H^{2}(\Omega)}+\| f \|^{2}_{H^{1}(0,T; L^{2}(\Omega))}).
\end{equation*}
This along with (\ref{estimate-ct}) gives
\begin{equation} \label{estimate-rinfty-D}
\| \br_{m} \|_{L^{\infty}(0,T; H(\Div, \Omega))} ^{2} \leq C (\| c_{0} \|^{2}_{H^{2}(\Omega)}+\| f \|^{2}_{H^{1}(0,T; L^{2}(\Omega))}).
\end{equation}
%Finally, by using \eqref{weakconv} we deduce the same bound for $ \br $.
The lemma now follows from (\ref{estimate-ctinfty}), (\ref{estimate-rinfty-D}) and (\ref{weakconv}).
%\end{enumerate}
\qquad
\end{proof}
\noindent\\[-2mm]
% ########################

\indent
In the sequel, we will consider two domain decomposition methods for solving
\JR{\eqref{variational-mixed},~\eqref{ic}}.
The first one involves local Dirichlet subproblems whose well-posedness is an extension of Theorem \ref{thrm1}.
In the second approach, the optimized Schwarz waveform relaxation method, we shall impose Robin transmission
conditions on the interfaces.  Thus, we extend the well-posedness results above to the case of
Robin boundary conditions.
%
% ------------------------------------------------
%
%		SECTION 3: Extend to Robin BCs
%	
% ------------------------------------------------
%
\section{\JR{A local} problem with Robin boundary conditions}\label{sec:Robinbc}
In this \JR{section}, we consider problem \eqref{primal}-\eqref{ic} with Robin boundary conditions on 
$\partial \Omega \times (0,T)$ :
\begin{equation} \label{Robinbc}
- \br \cdot \bn + \alpha c = g, \quad \text{on} \; \partial \Omega \times (0,T),
\end{equation}
where
\JR{$\alpha $ defined on $\partial \Omega$ is a time independent positive, bounded coefficient and
$g$ is a space-time function. We}
define $\check{\alpha}:= \frac{1}{\alpha}$
and suppose that $ 0 <\kappa_{1} \leq \check{\alpha} \leq \kappa_{2} $ a.e. in
$ \partial \Omega $. We denote by $ (\cdot, \cdot) _{\partial \Omega} $ and $ \| \cdot \|_{\partial \Omega}$
the inner product and norm in $ L^{2}(\partial \Omega) $ respectively. To derive a variational formulation
corresponding to boundary condition~\eqref{Robinbc}, we introduce the following Hilbert space
$$
\widetilde{\Sigma} = \iH (\Div, \Omega)
  := \{ \bv \in H(\Div, \Omega) \vert \; \bv \cdot \bn \in L^{2}(\partial \Omega )\},
$$
equipped with the norm 
$$
\| \bv \|^{2}_{\iH(\Div,\Omega)}
  := \| \bv \|_{H(\Div, \Omega)}+ \| \bv \cdot \bn \|^{2}_{\partial \Omega}. 
$$
The weak problem with Robin boundary conditions may now be written as follows:
\begin{eqnarray} \label{Robinvar}
\mbox{For a.e. $ t \in (0,T) $, find $ c(t) \in M$ and $ \br (t) \in \widetilde{\Sigma}$ such that}
   \hspace{4cm}\nonumber\\  
\begin{array}{rll}
(\omega \partial_t c, \mu ) + (\nabla \cdot \br , \mu )
  & = (f, \mu ), & \forall \mu \in M, \\
- (\nabla \cdot \bv, c ) + (\bD^{-1} \br, \bv ) + ( \check{\alpha}  \br \cdot \bn, \bv \cdot \bn )_{\partial \Omega}
  & =-(\check{\alpha} g, \bv \cdot \bn )_{\partial \Omega},  & \forall \bv \in \widetilde{\Sigma}.
\end{array}\quad \quad \quad
\end{eqnarray}
\begin{theorem} \label{thrm2}
If $ f $ is in $ L^{2} (0,T; L^{2}(\Omega)) $,
$ g $ in $H^{1}(0,T; L^{2}(\partial \Omega)) $
and $ c_{0} $ in $H^{1} (\Omega) $, then \JR{problem~\eqref{Robinvar},~\eqref{ic}} has a unique solution
$$
(c, \br) \in  \, H^{1}(0,T; L^{2}(\Omega)) \times \left (L^{2}(0,T; \iH(\emph{\Div}, \Omega))
\cap L^{\infty} (0,T; \pmb{L^{2}(\Omega)})\right ).
$$
Moreover, if $\bD$ is in $\pmb{W^{1,\infty}(\Omega)}$, $ f $ in $H^{1}(0,T; L^{2}(\Omega)) $
and $ c_{0}$ in $ H^{2}(\Omega) $ then
$$
(c, \br) \in  \, W^{1, \infty} (0,T; L^{2}(\Omega)) \times \left (L^{\infty} (0, T; \iH(\emph{\Div}, \Omega))
\cap H^{1}(0,T; \pmb{L^{2}(\Omega)})\right ).
$$
\end{theorem}
\vspace{-3mm}
\begin{proof}
The proof of Theorem \ref{thrm2} relies on energy estimates and Gronwall's lemma, together
with a Galerkin method, as for the proof of Theorem~\ref{thrm1}. We only present here parts of the proof that
  are different from those of the proof of Theorem~\ref{thrm1}.
We construct the finite-dimensional approximation problems to (\ref{Robinvar}) as follows
\begin{eqnarray} \label{Robin-appro}
\hspace{-5mm}    
  \begin{array}{rll} (\omega \partial_{t} c_{n}, \mu_{i} ) + (\nabla \cdot \br_{m} , \mu_{i} )
    \hspace{-2mm}& = (f, \mu_{i} ),&
    %\forall i=1, \hdots, n, \\
 \hspace{-3mm} 1\le i \le n,\\
- (\nabla \cdot \tilde{\bv}_{j}, c_{n} )+ (\bD^{-1} \br_{m}, \tilde{\bv}_{j} )  + (\check{\alpha}  \br_{m} \cdot \bn, \tilde{\bv}_{j} \cdot \bn )_{\partial \Omega}
\hspace{-2mm} & =  (-\check{\alpha} g, \tilde{\bv}_{j} \cdot \bn  )_{\partial \Omega},  &
\hspace{-3mm}  1\le j \le m,\hspace{3mm}
%\forall j=1, \hdots, m,
\end{array}
\end{eqnarray}
where $ c_{n} \in M_{n} $, $ \br_{m} \in \widetilde{\Sigma}_{m} $ and $ \tilde{\bv}_{i}, i=1, \hdots, m $
is the basis of $ \widetilde{\Sigma}_{m} $. We then rewrite (\ref{Robin-appro}) in matrix form
as in~\eqref{appro-matrix}:
\begin{eqnarray} 
\bW_{n} \frac{d \bC_{n}}{dt} (t) + \tilde{\bB}_{nm} \tilde{\bR}_{m}(t) & = \bF_{n}(t), \nonumber \\
-\tilde{\bB}_{nm}^{T} \bC_{n} (t) +\tilde{\bA}_{m} \tilde{\bR}_{m}(t) & = \bG_{m} (t), \nonumber 
\end{eqnarray}
where $ \tilde{\bR}_{m} $ is the vector of degrees of freedom of $ \br_{m} $ with respect
to the basis~$ \{ \tilde{\bv}_{i} \}_{i=1}^{m} $;
$$
(\tilde{\bA}_{m})_{ij} = (\bD^{-1} \tilde{\bv}_{j}, \tilde{\bv}_{i})
  + (\check{\alpha} \tilde{\bv}_{j} \cdot \bn, \tilde{\bv}_{i} \cdot \bn )_{\partial \Omega},
     \qquad \forall 1 \leq i,j \leq m,
$$
is symmetric and positive-definite,
$$
(\tilde{\bB}_{nm})_{ij} = (\nabla \cdot \tilde{\bv}_{j}, \mu_{i}) \; \text{and} \;
(\bG_{m} (t))_{i} = (-\check{\alpha} g(t), \tilde{\bv}_{i} \cdot \bn )_{\partial \Omega},
   \quad \forall \, 1 \leq i \leq n, 1 \leq j \leq m.  
$$
Thus, there exists a unique solution $ (c_{n}, \br_{m}) $ to (\ref{Robin-appro}).

Now to prove the existence of a solution to (\ref{Robinvar}), we derive suitable energy estimates
in the same manner as in Section~\ref{Sec:Model} but with an extra term $ \br \cdot \bn $ on the boundary. 
\begin{lemma}\label{lemma:energyestimRobin} 
Let $ f \in L^{2} (0,T; L^{2}(\Omega)) $ , $ g \in H^{1}(0,T; L^{2}(\partial \Omega )$
and $ c_{0} \in H^{1} (\Omega) $. \vspace{5pt}\\
The following estimates hold 
\begin{eqnarray*}
 \begin{array}{c}
 \hspace{-1.3cm} (i) \; \| c\|_{L^{\infty} (0,T; L^{2}(\Omega))}
  + \| \br \|_{L^{2}(0,T; \pmb{L^{2}(\Omega)})} +  \| \br \cdot \bn \|_{L^{2}(0,T; L^{2}(\partial \Omega))} \\[1mm]
 \hspace{2cm}\leq C ( \| c_{0}\|_{L^{2}(\Omega)} + \| f\| _{L^{2}(0,T; L^{2}(\Omega))}
  + \| g\|_{L^{2}(0,T; L^{2}(\partial \Omega))}),\\[3mm]
 \hspace{-1cm} (ii) \; \| \partial_{t} c \|_{L^{2}(0,T; L^{2}(\Omega))} + \| \br \| _{L^{\infty} (0,T; \pmb{L^{2}(\Omega)})}
  +  \| \br \cdot \bn \|_{L^{\infty} (0,T; L^{2}(\partial \Omega))} \\[1mm]
   \hspace{2cm}\leq C ( \| c_{0}\| _{H^{1}(\Omega)} + \| f\| _{L^{2}(0,T; L^{2}(\Omega))}
  + \| g\|_{H^{1}(0,T; L^{2}(\partial \Omega))}),\\[3mm]
(iii) \; \| \br \|_{L^{2}(0,T; \iH(\text{div},\Omega))} 
 \leq C ( \| c_{0}\|_{H^1(\Omega)} + \| f\| _{L^{2}(0,T; L^{2}(\Omega))}
  + \| g\|_{H^1(0,T; L^2(\partial \Omega))}).\hspace{-1.1cm}
\end{array}
\end{eqnarray*}
\end{lemma}
\begin{lemma} \label{lem:estimates2}
\emph{(Estimates \JR{with greater} regularity)} Assume that $\bD$ is in $\pmb{W^{1,\infty}(\Omega)}$,
$ c_{0} $ in $ H^{2}(\Omega) $,
$ f $ in $H^{1}(0,T; L^{2}(\Omega)) $ \JR{and $ g $ in $ H^{1}(0,T; L^{2}(\partial\Omega)) $}, then 
\begin{align*}
\| \partial_{t} c\|_{L^{\infty} (0,T; L^{2}(\Omega))} + \| \br \| _{L^{\infty} (0, T; \iH (\emph{\Div}, \Omega))}
  + \| \partial_{t} \br \|_{L^{2}(0,T; \pmb{L^{2}(\Omega)})} \\
  \leq C (\| f\|_{H^{1}(0,T; L^{2}(\Omega))} + \| c_{0}\|_{H^{2}(\Omega)}
  + \| g\|_{H^{1}(0,T; L^{2}(\partial \Omega))}).
\end{align*}
\end{lemma}
%
%""""""
%
\begin{proof}\emph{(of Lemma \ref{lemma:energyestimRobin}).} 
In order to prove \emph{(i)}, as before, we take $ c_{n} $ and $ \br_{m} $ as \JR{test} functions
in~\eqref{Robin-appro} \JR{and add the two equations:}
$$
(\omega \partial_{t} c_{n}, c_{n}) + \left (\bD^{-1} \br_{m}, \br_{m}\right )
+ \left(\check{\alpha}  \br_{m} \cdot \bn, \br_{m} \cdot \bn \right)_{\partial \Omega}
  = \left (f, c_{n}\right ) + \left( -\check{\alpha} g, \br_{m} \cdot \bn \right)_{\partial \Omega}.
$$
\JR{The} assumptions concerning $ \omega $, $\pmb{D}$ and $ \check{\alpha} $ give
$$
(\omega \partial_{t} c_n, c_n) \geq \frac{\omega_{-}}{2} \frac{d}{dt} \| c_n\|^{2}, \; \;
(\bD^{-1} \br_m, \br_m ) \geq \delta_{-} \| \br_m \| ^{2},\; \;
\left (\check{\alpha}  \br_m \cdot \bn, \br_m \cdot \bn \right )_{\partial \Omega}
  \geq \kappa_{1} \| \br_m \cdot \bn \|^{2}_{\partial \Omega},
$$
\JR{and the} Cauchy-Schwarz inequality:
\begin{equation}\label{eq:CauchySchw_f}
\mid \left (f,c_n\right ) \mid  \leq \| f\| \| c_n\| \leq \frac{1}{2\omega_{-}} \| f\|^{2} + \frac{\omega_{-}}{2} \|c_n \|^{2}.
\end{equation}
Similarly, for each $ \epsilon >0 $,
\begin{equation}\label{eq:CauchySchw_g}
\mid -\left (\check{\alpha} g,\br_m \cdot \bn \right )_{\partial \Omega} \mid
\leq \kappa_{2} \| g\|_{\partial \Omega}  \; \| \br_m \cdot \bn \|_{\partial \Omega}
\leq \kappa_{2} \left (\frac{1}{2\epsilon} \| g\|_{\partial \Omega}^{2}
       + \frac{\epsilon}{2} \|\br_m \cdot \bn \|_{\partial \Omega}^{2} \right ).
\end{equation}
Choosing $ \epsilon =\frac{\kappa_{1} }{\kappa_{2} } $, we then obtain
\begin{equation*}
\frac{\omega_{-}}{2}\frac{d}{dt}\| c_n \|^{2} + \delta_{-} \| \br_m \|^{2}
    + \frac{\kappa _{1} }{2} \| \br_m \cdot \bn \|_{\partial \Omega}^{2}
 \leq \frac{1}{2\omega_{-}} \| f\|^{2} + \frac{\kappa_{2}^{2} }{2\kappa _{1}} \| g\|_{\partial \Omega}^{2}
    + \frac{\omega_{-}}{2} \| c_n\|^{2} .
\end{equation*}
Integrating this inequality over $ \left (0,t\right )$ for \JR{$ t \in (0,T]$},
and using $ \|c_{n}(0)\|^{2} \leq \| c_{0}\|^{2} $, we~get
\begin{multline*}
\| c_n\left (t\right ) \|^{2} + \frac{2 \delta_{-}}{\omega_{-}} \int_{0}^{t}\| \br_{m}\left (s\right ) \|^{2} \, ds
    + \frac{\kappa _{1}}{\omega_{-}} \int_{0}^{t} \| \br_m \left (s\right ) \cdot \bn \|_{\partial \Omega}^{2} \, ds \\
 \leq C \left (\|c_{0}\|^{2}+\| f\|^{2}_{L^{2}\left (0,T; L^{2}\left (\Omega\right )\right )}
    + \| g\|^{2}_{L^{2}\left (0,T; L^{2}\left (\partial \Omega \right )\right )}\right ) 
    + \int_{0}^{t} \| c_n\left (s\right )\|^{2} \, ds,
\end{multline*}
with $ C =\max(1, \frac{1}{\omega_{-}^{2}},\frac{\kappa_{2}^{2} }{\omega_{-}\kappa_{1}}) $.
\JR{Then an application of Gronwall's lemma completes the proof of $ (i) $}. \\

\JR{For \emph{(ii)}}, we follow the same steps as in~(\ref{1ii})-(\ref{4ii}):
taking $ \partial_{t} c_n \in L^{2}(0,T; M)  $ as the test function in the first equation
of~\eqref{Robin-appro}, we obtain
\begin{equation} \label{1ii-2}
 (\omega \partial_{t} c, \partial_{t} c) + (\nabla \cdot \br_m, \partial_{t} c ) = (f, \partial_{t} c). 
\end{equation}
Differentiating the second equation of~\eqref{Robin-appro} with respect to $ t $, we \JR{obtain}
\begin{align*} \label{2ii-2}
- (\nabla \cdot \bv, \partial_{t} c_n )+ (\bD^{-1} \partial_{t} \br_m, \bv )
  + (\check{\alpha}  \partial_{t} \br_m \cdot \bn, \bv \cdot \bn )_{\partial \Omega}
= - (\check{\alpha} \partial_{t} g, \bv \cdot \bn  )_{\partial \Omega}, \;
\forall \bv \in \widetilde{\Sigma}.
\end{align*}
Then we take $ \bv=\br_m $ in the previous equation
%in~\eqref{2ii-2}
and add the resulting equation to~\eqref{1ii-2} to obtain
$$
(\omega \partial_{t} c_n, \partial_{t} c_n) +(\bD^{-1} \partial_{t} \br_m, \br_m )
+ \left(\check{\alpha}  \partial_{t}\br_m \cdot \bn, \br_m \cdot \bn \right)_{\partial \Omega}
= (f, \partial_{t} c_n)- \left( \check{\alpha} \partial_{t}g, \br_m \cdot \bn \right)_{\partial \Omega}.
$$
As $ \bD $ is symmetric and positive definite, by applying the Cauchy-Schwarz inequality to the right hand side
as well as using the property of $ \omega $, we obtain
\begin{equation*}
\omega_{-} \|  \partial_{t} c \|^{2} + \frac{1}{2} \frac{d}{dt}\|\sqrt{\bD^{-1}} \br_m \|^{2}
  +\frac{\kappa_{1}}{2} \frac{d}{dt} \| \br_m \cdot \bn \|_{\partial \Omega}^{2}
\leq \mid \left (f, \partial_{t} c\right ) \mid
      + \mid \left ( \check{\alpha} \partial_{t} g, \br_m \cdot \bn\right )_{\partial \Omega} \mid.
\end{equation*}
We then apply the Cauchy-Schwarz inequality for the right-hand side
(as in~\eqref{eq:CauchySchw_f},~\eqref{eq:CauchySchw_g}, replacing $c$ and $g$ by $\partial_t c$
and $\partial_t g$), and take $\epsilon = \frac{\kappa_{1}}{\kappa_{2} }$,
$ C=\max(\frac{1}{\omega_{-}},\frac{\kappa_{2}^{2} }{\kappa_{1}}) $ to obtain
\begin{equation*}
\omega_{-} \| \partial_{t} c_n \|^{2} + \frac{d}{dt}\| \sqrt{\bD^{-1}} \br_m \|^{2}
    + \kappa_{1} \frac{d}{dt} \| \br_m \cdot \bn \|_{\partial \Omega}^{2}
 \leq C \left (\| f\|^{2} + \| \partial_{t} g\|_{\partial \Omega}^{2}\right )
    + \kappa_{1} \| \br_m \cdot \bn \|_{\partial \Omega}^{2}.
\end{equation*}
Integrating over $ \left (0,t\right ) $ for $ t \in [0,T] $, we find 
\begin{multline} \label{iii}
\omega_{-} \int_{0}^{t} \| \partial_{t} c_n\left (s\right ) \|^{2} ds
   + \|\sqrt{\bD^{-1}} \br_m \left (t\right )\|^{2} + \kappa_{1} \| \br_m\left (t\right ) \cdot \bn\|_{\partial \Omega}^{2} \\
\leq   C (\| f\|_{L^{2}(0,T; L^{2}(\Omega))}^{2} + \| g\|^{2}_{H^{1}(0,T; L^{2}(\partial \Omega))})
   + \| \sqrt{\bD^{-1}}\br_m\left (0\right )\|^{2} + \kappa_{1} \| \br_m\left (0\right ) \cdot \bn \|_{\partial \Omega}^{2} \\
   + \kappa_{1} \int_{0}^{t} \| \br_m\left (s\right ) \cdot \bn \|^{2}_{\partial \Omega} ds.
\end{multline}
So \JR{there} only remains to bound the term
$ (\| \sqrt{\bD^{-1}}\br_m\left (0\right )\|^{2} + \kappa_{1} \| \br_m\left (0\right ) \cdot \bn \|_{\partial \Omega}^{2} ) $.
Toward this end, we use the second equation of~\eqref{Robin-appro}
with $\bv=\br_m$ and for $ t=0 $ to obtain:
\begin{align*}
\delta_{-} \| \br_m(0) \|^{2} + \kappa_{1} \| \br_m(0) \cdot \bn &\|^{2}
 \leq (\nabla \cdot \br_m(0), c_n(0)) + \left (-\check{\alpha} g(0), \br_m(0) \cdot \bn \right )_{\partial \Omega}.
\end{align*}
Let $ n \rightarrow \infty $, as $ c_{n}(0) \rightarrow c_{0} $ we have
\begin{align*}
\delta_{-} \| \br_m(0) \|^{2} + \kappa_{1} \| &\br_m(0) \cdot \bn \|^{2}
 \leq (\nabla \cdot \br_m(0), c_{0}) + \left (-\check{\alpha} g(0), \br_m(0) \cdot \bn \right )_{\partial \Omega} \\
& \leq (-\br_m (0), \nabla c_{0}) + \left (c_{0} -\check{\alpha} g(0), \br_m(0) \cdot \bn \right )_{\partial \Omega} \\
& \leq \frac{\delta_{-}}{2} \| \br_m(0) \|^{2} + \frac{1}{2 \delta_{-} } \| \nabla c_{0}\|^{2}
+ \frac{\kappa_{1}}{2} \| \br_m(0) \cdot \bn \|^{2}
+ \frac{\kappa_{2}}{2 \kappa_{1}} \|c_{0} -  g(0)\|_{\partial \Omega}^{2},
\end{align*}
or
\begin{equation*}
\delta_{-} \| \br_m(0) \|^{2} + \kappa_{1} \| \br_m(0) \cdot \bn \|^{2}
  \leq C \left (\| c_{0}\|_{H^{1}(\Omega)}^{2} + \| g\|^{2}_{H^{1}(0,T; L^{2}(\partial \Omega)}\right ).
\end{equation*}
This along with \eqref{iii} \JR{and Gronwall's lemma} yields $(ii)$.
We now estimate $\| \nabla \cdot \br_{m} \| ^{2}$ as in Section~\ref{Sec:Model}:
we derive~\eqref{2iii} from~\eqref{1iii} and the first equation of~\eqref{Robin-appro}
(after multiplying by $ \xi_{m}^{i} (t) $ and summing over~$i=1,\hdots,n$).
Then, using the bound for $\|\partial_t c\|_{L^2(0,T;L^2(\Omega))}$ in $(ii)$, we obtain
\begin{equation}\label{eq:estim-divrm}
  \| \nabla \cdot \br_{m} \| ^{2}_{L^2(0,T;L^2(\Omega))} \leq C (\| c_{0} \|^{2}_{H^{1}(\Omega)}+\| f \|^{2}_{L^2(0,T; L^{2}(\Omega))}
+\| g\|^{2}_{H^{1}(0,T; L^{2}(\partial \Omega))} ).
\end{equation}
This along with (i) gives
\begin{equation*} 
 \| \br_{m} \|_{L^{2}(0,T; \iH(\Div, \Omega))} ^{2}
   \leq C (\| c_{0} \|^{2}_{H^{1}(\Omega)}+\| f \|^{2}_{L^2(0,T; L^{2}(\Omega))} + \| g\|^{2}_{H^{1}(0,T; L^{2}(\partial \Omega))} ),
\end{equation*}
\JR{and the proof of Lemma~\ref{lemma:energyestimRobin} is completed}. \qquad
\end{proof}
%
%
% higher regularity estimates
%
%

We now prove Lemma \ref{lem:estimates2} for the higher regularity of the solution to~\eqref{mixed}.
\begin{proof}\emph{(of Lemma \ref{lem:estimates2})}.
Let $ n, m \geq 1 $.
\JR{ Differentiate both} equations of~\eqref{Robin-appro}
with respect to $ t $, \JR{take}
$ \mu=\partial_{t} c_n  $ and $ \bv=\partial_{t}\br_m  $
as the test functions,  \JR{and} add the two resulting equations to obtain
\begin{align*}
(\omega \partial_{tt} c_n,\partial_{t} c_n) +(\bD^{-1} \partial_{t} \br_m,\partial_{t} \br_m)
&+ \left(\check{\alpha}  \partial_{t} \br_m \cdot \bn, \partial_{t} \br_m \cdot \bn \right)_{\partial \Omega}\\
&= \left (\partial_{t} f, \partial_{t} c_n\right )
    - \left( \check{\alpha} \partial_{t} g, \partial_{t} \br_m \cdot \bn \right)_{\partial \Omega}.
\end{align*}
Then, the assumptions concerning $ \omega $, $\pmb{D}$, $ \check{\alpha} $,
and the Cauchy-Schwarz inequality give
\begin{equation*}
\frac{\omega_{-}}{2}\frac{d}{dt}\| \partial_{t} c_n \|^{2} + \delta_{-} \| \partial_{t} \br_m \|^{2}
    + \frac{\kappa _{1} }{2} \| \partial_{t} \br_m \cdot \bn \|_{\partial \Omega}^{2}
 \leq \frac{1}{2\omega_{-}} \| \partial_{t} f\|^{2} + \frac{\kappa_{2}^{2} }{2\kappa _{1}} \| \partial_{t} g\|_{\partial \Omega}^{2}
    + \frac{\omega_{-}}{2} \| \partial_{t} c_n\|^{2} .
\end{equation*}
Integrating this inequality over $ (0,t) $, for  for \JR{$ t \in (0,T] $}, we obtain 
\begin{eqnarray}\label{eq:estimateintt}
&\hspace{-3.2cm}\| \partial_{t} c_n (t) \|^{2} + \frac{2 \delta_{-}}{\omega_{-}} \int_{0}^{t}\| \partial_{t}\br_m (s) \|^{2} \, ds
    + \frac{\kappa _{1}}{\omega_{-}} \int_{0}^{t} \| \partial_{t} \br_m (s) \cdot \bn \|_{\partial \Omega}^{2} \, ds
 \nonumber\\
&\leq C (\|\partial_{t} c_n(0)\|^{2}+\| \partial_{t} f\|^{2}_{L^{2}(0,T; L^{2}(\Omega))}
    + \| \partial_{t} g\|^{2}_{L^{2}\left (0,T; L^{2}\left (\partial \Omega \right )\right )} ) 
    + \int_{0}^{t} \| \partial_{t} c_n\left (s\right )\|^{2} \, ds, \qquad
\end{eqnarray}
with $ C =\max(1, \frac{1}{\omega_{-}^{2}},\frac{\kappa_{2}^{2} }{\omega_{-}\kappa_{1}}) $.
\JR{To bound} $\|\partial_{t} c_n(0)\|$,
we use the first equation of~\eqref{Robin-appro}
\JR{at $t=0$ with $\mu=\partial_{t} c_n $}, and the Cauchy-Schwarz inequality to obtain
$$
\| \partial_{t} c_n (0) \|^2
  \le C (\| f(0) \|^{2} +\| \nabla \cdot \br_m (0) \|^2)
  \le  C (\| f(0) \|^{2} +\| c_0 \|^2_{H^2(\Omega)}).
$$
\JR{Here we have used the fact that
$\pmb{D}^{-1} \br_m(0)= -\nabla c_n(0)$ in $\mathcal{D}^{\prime}(\Omega)$
given by the second equation of~\eqref{Robin-appro}, and hence in $L^2(\Omega)$ since $c_0 \in H^2(\Omega)$.
}
From this inequality and \eqref{eq:estimateintt}, we have
\begin{multline}\label{estimate-ct2}
\| \partial_{t} c_n (t) \|^{2} + \frac{2 \delta_{-}}{\omega_{-}} \int_{0}^{t}\| \partial_{t}\br_m (s) \|^{2} \, ds
    + \frac{\kappa _{1}}{\omega_{-}} \int_{0}^{t} \| \partial_{t} \br_m (s) \cdot \bn \|_{\partial \Omega}^{2} \, ds\\
  \leq C (\| c_{0} \|^{2}_{H^{2}(\Omega)}+\| f \|^{2}_{H^{1}(0,T; L^{2}(\Omega))}+\| g \|^{2}_{H^{1}(0,T; L^{2}(\partial\Omega))})
 + \int_{0}^{t}\| \partial_{t} c_n \|^{2}ds. 
\end{multline}
It now follows from~\eqref{estimate-ct2} and Gronwall's lemma that
\begin{multline}\label{estimate-ctinfty2}
\| \partial_{t} c_n\|_{L^{\infty} (0,T; L^{2}(\Omega))}
  + \| \partial_{t} \br_m \|_{L^{2}(0,T; \pmb{L^{2}(\Omega)})}
  +  \| \partial_{t} \br_m \cdot \bn \|_{L^{2}(0,T; L^{2}(\partial \Omega))} \\
\leq C (\| c_{0} \|^{2}_{H^{2}(\Omega)}+\| f \|^{2}_{H^{1}(0,T; L^{2}(\Omega))}+\| g \|^{2}_{H^{1}(0,T; L^{2}(\partial\Omega))}).
\end{multline}
To obtain the estimate in the $\iH(\Div, \Omega)$-norm,
we follow the same steps as for~\eqref{eq:estim-divrm} to obtain
\begin{equation*} 
\| \nabla \cdot \br_m \| ^{2}_{L^\infty(0,T; L^{2}(\Omega))}
  \leq C (\| c_{0} \|^{2}_{H^{2}(\Omega)}+\| f \|^{2}_{H^{1}(0,T; L^{2}(\Omega))}+\| g \|^{2}_{H^{1}(0,T; L^{2}(\partial\Omega))}).
\end{equation*}
This along with the inequality $(i)$ of Lemma~\ref{lemma:energyestimRobin} gives
\begin{equation} \label{estimate-rinfty}
\| \br \|_{L^\infty(0,T; \iH(\Div, \Omega))} ^{2}
  \leq C (\| c_{0} \|^{2}_{H^{2}(\Omega)}+\| f \|^{2}_{H^{1}(0,T; L^{2}(\Omega))}+\| g \|^{2}_{H^{1}(0,T; L^{2}(\partial\Omega))}).
\end{equation}
The lemma now follows from~\eqref{estimate-ctinfty2} and~\eqref{estimate-rinfty}.
\qquad
\end{proof} 

Thanks to Lemma~\ref{lemma:energyestimRobin}, we can finish the proof of Theorem~\ref{thrm2} using similar
  arguments as for the proof of  Theorem~\ref{thrm1}.
\qquad
\end{proof} 
}
%
% ------------------------------------------------
%
%		SECTION 4: Space-time DDM
%	
% ------------------------------------------------
%
\section{Space-time domain decomposition methods}
\label{Sec:DD}
In this section, we present two nonoverlapping domain decomposition {methods for solving problem (\ref{mixed}).
For simplicity, we consider a decomposition of $ \Omega $ into two non overlapping subdomains
$ \Omega_{1} $ and $ \Omega_{2} $ separated by an interface $\Gamma$:
$$
\Omega_{1} \cap \Omega_{2} = \emptyset;\quad  \Gamma
= \partial \Omega_{1} \cap \partial \Omega_{2} \cap \Omega, \quad \Omega=  \Omega_{1} \cup \Omega_{2} \cup\Gamma.
$$
Also for the sake of simplicity we have assumed throughout this section and the next that the boundary condition
given on $\partial\Omega$ is a homogeneous Dirichlet condition.  However, 
the analysis given below can be generalized to the case of multiple subdomains
and more general boundary conditions (see Section \ref{Sec:Num}).} 

For $i=1,2$, let $ \pmb{n}_{i} $ denote the unit outward pointing vector field on $\partial\Omega_i$,
and for any scalar, vector or tensor valued function $\varphi$ defined on $\Omega$, let $\varphi_i$ denote
the restriction  of $ \varphi$ to $ \Omega_{i} $. Using this notation, problem (\ref{mixed})
can be reformulated as an equivalent multidomain problem consisting of the following space-time subdomain problems
\begin{equation} \label{multi-mixed}
\begin{array}{rll} \omega_{i}\partial_{t} c_{i} + \nabla \cdot \pmb{r}_{i}
        & =f & \text{in} \; \Omega_{i} \times (0,T),\\
\nabla c_{i} +\pmb{D}^{-1}_{i} \pmb{r}_{i}
        &=0  & \text{in} \; \Omega_{i} \times (0,T), \\
c_{i}    & = 0 & \text{on} \; \partial \Omega_{i} \cap \partial \Omega \times (0,T),\\
c_{i}(0) & = c_{0} & \text{in} \; \Omega_{i},
\end{array} \quad \text{for $ i=1,2 $},
\end{equation}
together with the transmission conditions on the space-time interface \vspace{-0.1cm}
\begin{equation} \label{physicsTC}
\begin{array}{l}
c_{1} = c_{2} \\
\pmb{r}_{1} \cdot \pmb{n}_{1} + \pmb{r}_{2} \cdot \pmb{n}_{2} =0
\end{array} \quad \text{on} \; \Gamma \times \left (0,T\right ), \vspace{-0.1cm}
\end{equation}
Alternatively, and equivalently, one may impose the transmission conditions
\begin{equation} \label{RobinTCs}
\left. \begin{array}{ll} -\br_{1} \cdot \bn_{1} + \prm_{1,2} c_{1} & = -\br_{2} \cdot \bn_{1} + \prm_{1,2} c_{2}\\
-\br_{2} \cdot \bn_{2} + \prm_{2,1} c_{2} & = -\br_{1} \cdot \bn_{2} + \prm_{2,1} c_{1}
\end{array} \right . \quad \text{ on } \Gamma \times \left (0,T\right ),
\end{equation}
where $\prm_{1,2}$ and $\prm_{2,1}$ are a pair of positive parameters.  The first method that we consider
is based on (\ref{multi-mixed}) together with the "natural" transmission conditions (\ref{physicsTC})
while the second method is based on (\ref{multi-mixed}) together with
the Robin transmission conditions~(\ref{RobinTCs}).  For the latter method the parameters $\prm_{i,j}$
may be optimized to improve the convergence rate of the iterative scheme
(see~\cite{Bennequin, OSWR1d1, OSWR1d2, OSWRwave}).  

For both methods the multidomain problem is formulated through the use of interface operators as a problem
posed on the space-time interface.  For the first method the interface operators
are time-dependent Steklov-Poincar\'e (Dirichlet-to-Neumann) operators  while for the second
they are Robin-to-Robin operators.  Associated with a Jacobi algorithm
this latter method is known as the Optimized Schwarz Waveform Relaxation (OSWR) method.
\CJ{Rewriting the OSWR 
method as a space-time interface problem solved by a more general (Krylov) method was done in~\cite{Haeberlein};
here we extend that work to a problem written in mixed form.}
%
%
%........................................
%
%	Subsection 4.1  - Method 1
%
%........................................
%
\subsection{Method~1: Using the time-dependent Steklov-Poincar\'e operator}\label{subsec:M1}
To introduce the interface problem for this method we introduce several
operators, but first we define some notation:
$$
\Lambda =\All{H^{1}(0,T; H^{\frac{1}{2}}_{00}(\Gamma))},\quad \mbox{ and, for }i=1,2,\quad
M_i =L^2(\Omega_i)\quad \mbox{ and }\quad \Sigma_i = H({\Div}, \Omega_i).
$$
%and note that the dual space  of $\Lambda$ is \CJ{$\Lambda^*=L^{2}(0,T; H^{-\half}_{00}(\Gamma))$}.
%
We also define $H_{*}^{1} (\Omega_i)=\{v\in H^{1} (\Omega_i), \; v=0 \text{ over } \partial \Omega_i \cap \partial \Omega \}, \, \text{for} \, i=1,2 $. \\
Next, let $\iD_{i}, i=1,2,$ be the solution operator that associates to the boundary,  right-hand-side,
and initial data $(\lambda, f, c_{0})$ the solution $(c_{i},\br_{i})$ of the
subdomain problem
\begin{equation} \label{Schur-subprob}
\begin{array}{rll}
\omega_{i}\partial_{t} c_{i} + \nabla \cdot \br_{i}
        & =f & \text{in} \; \Omega_{i} \times \left (0,T\right ),\\
\nabla c_{i} +\bD^{-1}_{i} \br_{i}
        &=0  & \text{in} \; \Omega_{i} \times \left (0,T\right ), \\
c_{i}    & = 0 & \text{on} \; \partial \Omega_{i} \cap \partial \Omega \times (0,T), \\
c_{i}    & =\lambda & \text{on} \; \Gamma \times \left (0,T\right ),\\
c_{i}(0) & = c_{0} & \text{in} \; \Omega_{i}.
\end{array} \vspace{-0.1cm}
\end{equation}
An extension of Theorem~\ref{thrm1} (to \JRo{the case of} non-homogeneous Dirichlet boundary conditions) guarantees that
$$
\begin{array}{lccc}
	\iD_i\,:& \Lambda \times L^{2} (0,T; L^{2}(\MK{\Omega_i}))\times H_{*}^{1} (\Omega_i)&
          \longrightarrow &H^{1}(0,T; M_i) \times L^{2}(0,T; \Sigma_i) \\[.1cm]
	& (\lambda, f, c_{0})& \mapsto & (c_{i},\br_{i})=(c_{i}(\lambda, f, c_{0}),\br_{i}(\lambda, f, c_{0}))
\end{array}
$$
is a well defined operator.  We also make use of the normal trace operator
$$
\begin{array}{lccc}
	\iF_i\,:& H^{1}(0,T;M_i)) \times L^{2}(0,T; \Sigma_i)& \longrightarrow & \All{L^{2}(0,T;  (H^{\half}_{00}(\Gamma))^{\prime})}\\[.1cm]
	& (c_{i},\br_{i})& \mapsto & \br_{i}\cdot \bn_{i}{ \mid_{\Gamma\times (0,T)}}
\end{array}
$$
which is then used to define the following operators:
$$
\begin{array}{lccc}
	\iS_i\,:& \Lambda & \longrightarrow & \All{L^{2}(0,T;  (H^{\half}_{00}(\Gamma))^{\prime})}\\[.1cm]
	& \lambda& \mapsto & - \iF_i\iD_i(\lambda,0,0)
\end{array}
$$
and
$$
\begin{array}{lccc}
	\chi_i\,:& L^{2} (0,T; L^{2}(\Omega_i))\times H_{*}^{1} (\Omega_i)& \longrightarrow & \All{L^{2}(0,T;  (H^{\half}_{00}(\Gamma))^{\prime})}\\[.1cm]
	& (f,c_0)& \mapsto & \iF_i\iD_i(0,f,c_0).
\end{array}
$$
Now letting $\iS=\iS_1+\iS_2$ and $\chi=\chi_1+\chi_2$ we may rewrite problem~(\ref{multi-mixed}),~(\ref{physicsTC})
as the interface problem
\begin{equation} \label{SchurIP}
	 \iS \lambda = \chi(f,c_0), \qquad \text{on} \; \; \Gamma \times \left (0,T\right ). \vspace{-0.1cm}
\end{equation}
The weak formulation of this problem is then
\All{\begin{equation} \label{SchurlPwk}
\begin{array}{l}
\hspace{-3cm}	\mbox{Find $ \lambda \in \Lambda$  such that:}\\[.2cm]
\hspace{-.5cm} \int_{0}^{T}\langle\iS\lambda,\eta\rangle =\int_{0}^{T}\langle\chi(f,c_0),\eta\rangle , \; \;  \forall \eta\in\Lambda,
\end{array}
\end{equation}
where $\langle\cdot,\cdot\rangle$ denotes the duality pairing between $H^{\half}_{00}(\Gamma)$ and $ (H^{\half}_{00}(\Gamma))^{\prime}$.}
The operator $\iS$ is the time-dependent Steklov-Poincar\'e operator, and to investigate its properties
we write the weak formulation of the interface problem~(\ref{Schur-subprob}) for $f=0$ and $c_0=0$:
\begin{equation} \label{Schursub-mixed}
\begin{array}{l}
\hspace{-1cm}
\mbox{For a.e. $ t \in (0,T) $, find $ c_{i}(t) \in  M_{i} $ and $ \br_{i}(t)  \in \Sigma_{i} $ such that}\\[.2cm]
	\begin{array}{rll}
		 \frac{d}{dt} ( \omega_{i}c_{i}, \mu )_{\Omega_{i}} + (\nabla \cdot \br_{i}, \mu)_{\Omega_{i}} & = 0, 
		 		& \forall \mu \in M_{i}, \\
		- (\nabla \cdot \bv, c_{i} )_{\Omega_{i}} + (\bD_{i}^{-1} \br_{i}, \bv )_{\Omega_{i}} 
				& = -\int_{\Gamma} \lambda (\bv \cdot \bn_{i}),  & \forall \bv \in \Sigma_{i}. \\
	\end{array} 
\end{array}
\end{equation}
For $\lambda\in\Lambda$ and for $ i=1, 2,$ we will denote by $(\cil,\ril)$
the solution of (\ref{Schursub-mixed}) for the data function $\lambda$.  Then for $\eta,\lambda\in\Lambda$ and
for almost every $t\in (0,T),$ we have 
$$
\begin{array}{rll}(\omega_{i} \partial_{t} \cil, \cie )_{\Omega_{i}} + (\nabla \cdot \ril, \cie)_{\Omega_{i}}
  & = 0, \\[.1cm]
- (\nabla \cdot \rie, \cil )_{\Omega_{i}} + (\bD_{i}^{-1} \ril, \rie )_{\Omega_{i}}
  & = -\int_{\Gamma} \lambda (\rie \cdot \bn_{i}).   
\end{array} 
$$
Now adding the first equation to the second equation in which the roles of $\lambda$ and
$\eta$ are reversed, integrating over time and summing on $i$, we obtain
$$
\begin{array}{rll}{{\sum_{i=1}^2\int_0^T}\left((\omega_{i} \partial_{t} \cil, \cie )_{\Omega_{i}} 
+ (\bD_{i}^{-1} \rie, \ril )_{\Omega_{i}}\right)}
  & = -{{\sum_{i=1}^2\int_0^T}\int_{\Gamma}  \eta(\ril \cdot \bn_{i}) }. \end{array} 
$$
Thus we see that
$$
\All{\int_{0}^{T} \langle\iS\lambda,\eta\rangle} =-{{\sum_{i=1}^2\int_0^T}  \!\!\!   \int_{\Gamma} (\ril \cdot \bn_{i}) \eta }
    ={\sum_{i=1}^2\int_0^T}   \!\!\!   \left((\omega_{i} \partial_{t} \cil, \cie )_{\Omega_{i}} 
      + (\bD_{i}^{-1} \ril, \rie )_{\Omega_{i}}\right),
$$
from which we conclude that $\iS$ is \MK{a positive} definite but non-symmetric, space-time interface operator.
Thus the existence and uniqueness of the solution of the space-time interface problem~(\ref{SchurlPwk})
does not follow in a standard way, and we have not pursued this question here.

Nonetheless, we solve a discretized version of problem~(\ref{SchurIP}) iteratively by using a Krylov method
(e.g. GMRES). Once the discrete approximation to $ \lambda $ is obtained, we can construct the multi-domain
solution of the discretized problem. Following the work in~\cite{NNPrecond, Mandelweights} for elliptic problems
with strong heterogeneities, we apply a Neumann-Neumann type preconditioner enhanced with averaging weights:
\begin{equation}\label{NNPrecondweight}
 \left (\sigma_{1} \iS_{1}^{-1} + \sigma_{2} \iS_{2}^{-1}\right ) \iS \lambda = \tilde{\chi}, 
 \end{equation}
where $ \sigma_{i}: \Gamma \times (0,T) \rightarrow [0,1] $ is such that $ \sigma_{1}+\sigma_{2} = 1$,
and $ \iS_{i}^{-1} $, the Neumann-to-Dirichlet operator, is the (pseudo)-inverse of $ \iS_{i} $, for $ i=1,2. $
%
%........................................
%
%	Subsection 4.2  - Method 2
%
%........................................
%
\subsection{Method~2: Using Optimized Schwarz Waveform Relaxation (OSWR)}\label{subsec:M2}
The function spaces that are needed to give the interface formulation of method 2 are
$$ \Xi :=  H^{1}(0,T; L^{2}(\Gamma)),\quad \mbox{ and, for }i=1,2,\quad
M_i =L^2(\Omega_i)\quad \mbox{ and }\quad \tSigma_i = \iH({\Div}, \Omega_i).
$$
To define the \MK{Robin-to-Robin} operator we first define for $i=1,2,$ the following solution operator $\iR_i$
which depends on the parameter $\prm_{i,j};\,\,j=3-i:$
$$
\begin{array}{lccc}
	\iR_i\,:& \Xi \times L^{2} (0,T; L^{2}(\Omega_i))\times H_{*}^{1} (\Omega_i)& 		
		\longrightarrow &\Xi \times H^{1}(0,T; M_i) \times L^{2}(0,T; \tSigma_i) \\[.1cm]
	& (\xi, f, c_{0})& \mapsto & (\xi,c_{i},\br_{i})=(\xi,c_{i}(\xi, f, c_{0}),\br_{i}(\xi, f, c_{0}))
\end{array}
$$
where $(c_{i},\br_{i})=(c_{i}(\xi, f, c_{0}),\br_{i}(\xi, f, c_{0}))$ is the solution to the problem
\begin{equation} \label{Schwarz-subprob}
\begin{array}{rll} 
	\omega_{i}\partial_{t} c_{i} + \nabla \cdot \br_{i} & =f & \text{in} \; \Omega_{i} \times \left (0,T\right ),\\
	\nabla c_{i} +\bD^{-1}_{i} \br_{i} &=0  & \text{in} \; \Omega_{i} \times \left (0,T\right ), \\
	c_{i} & = 0 & \text{on} \; \partial \Omega_{i} \cap \partial \Omega \times (0,T), \\
	-\br_i\cdot\bn_i+\prm_{i,j}c_{i} & =\xi& \text{on} \; \Gamma \times \left (0,T\right ),\\  %(\xi_i  ?)
	c_{i}(0) & = c_{0} & \text{in} \; \Omega_{i}.
\end{array} \vspace{-0.1cm}
\end{equation}
(As stated earlier the \CJ{parameters $\prm_{i,j}$} will be chosen is such a way as 
to optimize the convergence of the algorithm). The existence and uniqueness of the solution of
problem~(\ref{Schwarz-subprob}) is guaranteed by \MK{Theorem~\ref{thrm2}}. \\
Next, to impose the interface conditions \eqref{RobinTCs} we will need the following interface operators defined
for $i=1,2,$ and \CJ{$j=3-i$}: 
$$
\begin{array}{lccc}
\hspace{-0.2cm}\iB_i: & \hspace{-0.2cm} \JRo{\left (\Xi \times H^{1}(0,T; M_j) \times L^{2}(0,T; \tSigma_j) \right ) \cap \text{Im}(\iR_{j})} &  \hspace{-0.3cm} \longrightarrow & \All{\Xi}\\[.1cm]
   & (\xi,c_j,\br_{j})&  \hspace{-0.6cm}
     \mapsto &  \hspace{-0.5cm}(-\br_{j}\cdot \bn_{i}+\frac{\prm_{i,j}}{\prm_{j,i}}(\xi+\br_j\cdot\bn_j)){ \mid_{\Gamma\times (0,T)}}
\end{array}
$$	
\JRo{\begin{remark} \label{remarkRobin}
To see that $ \text{Im}(\iB_{i}) \subset \Xi $ (instead of simply $ L^{2}(0,T; L^{2}(\Gamma)) $, we note that \eqref{Robinvar} implies that
$\pmb{D}^{-1} \br(t)= -\nabla c(t)$ in $\mathcal{D}^{\prime}(\Omega)$ for a.e. $t \in (0,T)$. Since $\br(t) $ is in $H(\Div,\Omega)$, we have $c(t)\in H^1(\Omega)$,  for a.e. $t \in (0,T)$. Consequently, $ c_{i}(t)$ is in $ H^{1}(0,T; H^{1}(\Omega_{i})) $. This along with the fact that $ \xi \in \Xi $ implies that $ \br_{i} \cdot \bn_{i}{ \mid_{\Gamma\times (0,T)}} \in \Xi $.
\end{remark}}\\
Now, defining
$$
\begin{array}{lccccc}
	\iS_R\,:& \Xi\times\Xi & \longrightarrow & \All{\Xi} \times  \All{\Xi}\\[.1cm]
	& \begin{pmatrix}\xi_1  \\[.1cm]  \xi_2\end{pmatrix}	& \mapsto 
	&\begin{pmatrix} \xi_1- \iB_1\iR_2(\xi_2,0,0)  \\[.1cm]  \xi_2-\iB_2\iR_1(\xi_1,0,0) \end{pmatrix}
\end{array}
$$
and
$$
\begin{array}{lccc}
	\chi_R\,:&L^{2} (0,T; L^{2}(\Omega_i))\times \MK{H_{*}^{1} (\Omega_i)}& \longrightarrow & \All{\Xi}\times \All{\Xi}\\[.1cm]
	& (f,c_0)& \mapsto &\begin{pmatrix} \iB_1\iR_2(0,f,c_0) \\  \iB_2\iR_1(0,f,c_0) \end{pmatrix},
\end{array}
$$
we can write the interface problem as 
\begin{equation} \label{SchwarzIPshort}
\iS_{R} \begin{pmatrix}\xi_{1} \\ \xi_{2} \end{pmatrix}
 = \chi_{R} (f,c_0)\quad \text{on} \; \Gamma \times (0,T).
\end{equation}
\CJa{We then write \eqref{SchwarzIPshort}} in weak form as
\All{
\begin{equation} \label{SchwarzlPwk}
\begin{array}{l}
  \mbox{Find $ (\xi_1,\xi_2) \in \Xi\times\Xi$  such that}\\[.2cm]
    \begin{array}{rll}
      \int_{0}^{T} \int_{\Gamma} \iS_R\begin{pmatrix}\xi_1\\ \xi_2\end{pmatrix} \cdot \begin{pmatrix}\zeta_1\\
        \zeta_2\end{pmatrix}
          =\int_{0}^{T} \int_{\Gamma} \chi_R(f,c_0) \cdot \begin{pmatrix}\zeta_1\\ \zeta_2\end{pmatrix} , \; \;
            \forall (\zeta_1,\zeta_2)\in  \Xi\times\Xi.
      \end{array} 
\end{array}
\end{equation}}
In order to study the interface operator $\iS_R$, we \CJ{proceed as in the Section \ref{subsec:M1}} by giving
the weak formulation of the relevant subdomain problems (here \eqref{Schwarz-subprob} for $i=1,2$ and $j=3-1$)
for $f=0$ and $c_0=0$: 
\begin{equation} \label{Schwarzsub-mixed}
\begin{array}{l}
	\mbox{For a.e. $ t \in (0,T) $, find $ c_{i}(t) \in  M_{i} $ and $ \br_{i}(t)  \in \tSigma_{i} $
          such that,  $\forall \mu \in M_{i}$  and $ \forall \bv \in \tSigma_{i},$} \\[.2cm]
	\begin{array}{rll}
		 \frac{d}{d t} ( \omega_{i} c_{i}, \mu )_{\Omega_{i}} + (\nabla \cdot \br_{i}, \mu)_{\Omega_{i}}& = 0, 
		 		 \\[.1cm]
		- (\nabla \cdot \bv, c_{i} )_{\Omega_{i}} + (\bD_{i}^{-1} \br_{i}, \bv )_{\Omega_{i}} 
				+  \int_{\Gamma}\frac{1}{\prm_{i,j}} (\br_i \cdot \bn_{i}) (\bv \cdot \bn_{i})
				& = -\int_{\Gamma}\frac{1}{\prm_{i,j}} \xi (\bv \cdot \bn_{i}). 
	\end{array} 
\end{array}
\end{equation}
Now for any $\zeta\in\Xi$ letting  $\ciz\in H^1(0,T;M_i)$ and $\riz\in L^2(0,T;\tSigma_i)$
be such that $\iR_i(\zeta,0,0)=(\zeta,\ciz,\riz),$  we have for any pair of elements $\xi$ and $\zeta$ in $\Xi$
and for a.e. $ t \in (0,T) $ that
$$
\begin{array}{rl}
    ( \omega_{i}\partial_t \cix, \ciz )_{\Omega_{i}} + (\nabla \cdot \rix, \ciz)_{\Omega_{i}}  = 0,\hspace{2.5cm}&  \\[.2cm]
      - (\nabla \cdot \riz, \cix )_{\Omega_{i}} + (\bD_{i}^{-1} \rix, \riz )_{\Omega_{i}} 
	+  \int_{\Gamma}\frac{1}{\prm_{i,j}} (\rix \cdot \bn_{i}) (\riz \cdot \bn_{i})    \\[.2cm]
	& \hspace{-3.5cm}= -\int_{\Gamma}\frac{1}{\prm_{i,j}} \xi (\riz \cdot \bn_{i}). \\
\end{array}
$$
Next we add the first of these two equations to the second in which the roles of $\zeta$ and $\xi$
have been interchanged to obtain
\begin{equation}\begin{array}{rll}\label{eq:utile}
	( \omega_{i}\partial_t \cix, \ciz )_{\Omega_{i}} + (\bD_{i}^{-1} \rix, \riz )_{\Omega_{i}} 
	+  \int_{\Gamma}\frac{1}{\prm_{i,j}} (\rix \cdot \bn_{i}) (\riz \cdot \bn_{i}) \\[.2cm]
	\hspace{0cm}  = -\int_{\Gamma}\frac{1}{\prm_{i,j}} \zeta (\rix \cdot \bn_{i}),
\end{array}\end{equation}
and this holds for any pair of elements $\xi$ and $\zeta$ in $\Xi.$ 
%Now using the definition of $\iS_R$ we obtain
%\All{
%$$
%\begin{array}{rll}
%  \int_{0}^{T} \int_{\Gamma} \iS_R\begin{pmatrix}\xi_1\\[.1cm] \xi_2\end{pmatrix} \cdot
%    \begin{pmatrix}\zeta_1\\[.1cm] \zeta_2\end{pmatrix}
%      &=& \int_{0}^{T} \int_{\Gamma}\begin{pmatrix} 
%	  \xi_1 +\br_{2}(\xi_2)\cdot\bn_1 -\frac{\prm_{1,2}}{\prm_{2,1}}(\xi_2 +\br_{2}(\xi_2)\cdot\bn_2) \\[.1cm] 
%	  \xi_2 +\br_{1}(\xi_1)\cdot\bn_2 -\frac{\prm_{2,1}}{\prm_{1,2}}(\xi_1  +\br_{1}(\xi_1)\cdot\bn_1)
%    \end{pmatrix} \cdot
%    \begin{pmatrix}\zeta_1\\[.1cm] \zeta_2
%    \end{pmatrix}\\[.4cm]
%      &\hspace{-2.5cm}=& \hspace{-2.5cm}
%          \sum_{i=1}^2 \int_0^T \int_\Gamma (\xi_i - \frac{\prm_{i,j}}{\prm_{j,i}} \xi_j)\zeta_i
%	   - \sum_{i=1}^2 \int_0^T \int_\Gamma  ({\prm_{1,2}}+{\prm_{2,1}} )
%	     \frac{1}{\prm_{j,i}}(\br_{j}(\xi_j)\cdot\bn_j)\zeta_i .
%\end{array} 
%$$}
Now we consider the case in which the parameters $\prm_{i,j},\, i=1,2,\, j=3-i,$ are constant
and apply \eqref{eq:utile} with $ \xi =\xi_j $ and $ \zeta=\zeta_i, $ to  obtain
$$
\begin{array}{rll}
  \All{\int_{0}^{T} \int_{\Gamma} \!\!\iS_R\begin{pmatrix}\xi_1\\[.1cm] \xi_2\end{pmatrix} \cdot 
    \begin{pmatrix}\zeta_1\\[.1cm] \zeta_2\end{pmatrix}  \!\!  }
      = \sum_{i=1}^2 \int_0^T\!\!\! \Big\{ \int_\Gamma (\xi_i - \frac{\prm_{i,j}}{\prm_{j,i}} \xi_j)\zeta_i	
	 +({\prm_{1,2}}+{\prm_{2,1}} )\Big\{( \omega_{i}\partial_t c_i(\xi_j), c_i(\zeta_i) )_{\Omega_{i}} \\[.4cm]
	 + (\bD_{i}^{-1} \br_i(\xi_j), \br_i(\zeta_i ))_{\Omega_{i}} 
	 +  \int_{\Gamma}\frac{1}{\prm_{i,j}} ( \br_i(\xi_j) \cdot \bn_{i}) ( \br_i(\zeta_i ) \cdot \bn_{i}) \Big\}\Big\}
\end{array} 
$$
As for method 1, we obtain a non-symmetric, space-time interface operator, but here it is also not positive definite.
We solve the discretized problem iteratively using Jacobi iterations or GMRES.  The former choice is equivalent to
the OSWR algorithm, and in the next subsection we show that this mixed form of the algorithm converges.
%
%
% OSWR algorithm + CONVERGENCE proof
%
%
%
\subsubsection{The OSWR algorithm} 
We consider the general case \JR{in which} $ \Omega $ is decomposed into $ I $ non-overlapping subdomains
$ \Omega_{i} $.
We denote by $ \Gamma_{i,j} $ the interface between two neighboring subdomains $ \Omega_{i} $ and $ \Omega_{j} $,
$ \Gamma_{i,j} = \partial \Omega_{i} \cap \partial \Omega_{j} \cap \Omega $. Let $ \mathcal{N}_{i} $ be the set
of indices of the neighbors of the subdomain $ \Omega_{i}$, $ i=1, \hdots, I $.
The OSWR method may be written as follows: at the $ k^{th} $ iteration, we solve in each subdomain the problem
\vspace{-0.1cm}
\begin{equation} \label{algorithm}
\left . \begin{array}{lll} \partial_{t} c_{i}^{k} + \nabla \cdot \br_{i}^{k} &
  = f, & \text{in} \; \Omega_{i} \times \left (0,T\right ),\\
\nabla c_{i}^{k} +\bD_{i}^{-1} \br_{i}^{k}
  &=0,  & \text{in} \; \Omega_{i} \times \left (0,T\right ),\\
-\br_{i}^{k} \cdot \bn_{i} + \alpha_{i,j} c_{i}^{k}
  &=-\br_{j}^{k-1} \cdot \bn_{i} + \alpha_{i,j} c_{j}^{k-1}, & \text{on} \; \Gamma_{i,j} \times \left (0,T\right ),
    \forall j \in \mathcal{N}_{i}, 
\end{array} \right .
\end{equation}
where, \JR{for $ i=1, \cdots, I $, $ j \in \iN_{i} $, $ \alpha_{i,j} > 0 $ is a Robin parameter}. 
The initial value is that of $ c_{0} $ in each subdomain. Moreover,
\JR{$ \left (g_{i,j}\right ) :=-\br_{j}^{0} \cdot \bn_{i} + \alpha_{i,j} c_{j}^{0}$ is
an initial guess on $\Gamma_{i,j}$, for $ i=1, \cdots, I $, $ j \in \mathcal{N}_{i} $,
in order to start the first iterate. }
%
%
%	Theorem (convergence)
%
%
\begin{theorem}\label{thrm:convergence}
Let \CJ{$\bD \in \pmb{W^{1,\infty}(\Omega)}$}, 
$ f \in H^{1}(0,T; L^{2}(\Omega)) $ \JR{and} $ c_{0} \in H^{2}(\Omega) \cap H^{1}_{0}(\Omega) $ and \JR{let
$ \alpha_{i,j} \in L^{\infty}(\partial \Omega_{i}) $ be } such that $ \alpha_{i,j} \geq \alpha_{0} > 0 $ for
$ i=1, \cdots, I $, $ j \in \iN_{i} $.
\JR{Algorithm~\eqref{algorithm}, initialized by $ (g_{i,j}) $ in
$ H^{1}\left (0,T; L^{2}\left (\Gamma_{i,j}\right )\right ) $, $ i=1, \cdots, I $, $ j \in \mathcal{N}_{i} $,
defines} a sequence of iterates 
$$
(c_{i}^{k}, \br _{i}^{k})\in W^{1, \infty}(0,T; L^{2}(\Omega_{i})) \times \left (L^{2}(0,T;
  \iH (\emph{\Div}, \Omega_{i})) \cap H^{1}(0,T; \pmb{L^{2}(\Omega_{i})})\right ),
$$ 
for $ i=1, \cdots, I $,
that converges to the weak solution $ (c,\br) $ of  problem~\eqref{mixed}.
\end{theorem}
\begin{proof}
%\CJ{We remark that from \JR{the second equation of} \eqref{Schwarzsub-mixed} we have
%$-\pmb{D}^{-1} \br(t)= \nabla c(t)$ in $\mathcal{D}^{\prime}(\Omega)$ for a.e. $t \in (0,T)$. Thus, the
%property $\br(t) \in H(\Div,\Omega)$ implies that $c(t)\in H^1(\Omega)$,  for a.e. $t \in (0,T)$.
%Then, the
%}
%sequence $ (c_{i}^{k}, \br _{i}^{k})_{k} $ is well-defined due to the well-posedness of the subdomain problems,
%where the scalar unknown $ c_{i}^{k}(t)$ is \JR{additionally} in $ H^{1}(0,T; H^{1}(\Omega_{i})) $.
%This along with the fact that $ (g_{i,j}) \in H^{1}\left (0,T; L^{2}\left (\Gamma_{i,j}\right )\right ) $ \MK{implies} that
%the Robin term $ -\br_{i}^{k} \cdot \bn_{i} + \alpha_{i,j} c_{i}^{k}$ is in
%$H^{1}\left (0,T; L^{2}\left (\Gamma_{i,j}\right )\right )$ for $ i=1, \cdots, I $, $ j \in \mathcal{N}_{i} $,
%and for all $ k >0 $. 
\JRo{The sequence $ (c_{i}^{k}, \br _{i}^{k})_{k} $ is well-defined according to Theorem~\ref{thrm2} and Remark~\ref{remarkRobin}}. Now, to prove the convergence of algorithm~\eqref{algorithm}, as the equations are linear, we can take
\JR{$ f =0$ and $ c_{0}=0 $} and show that the sequence $ \left (c_{i}^{k}, \br_{i}^{k}\right )_{k} $
of iterates converges to zero in suitable norms. \\
%
%
%  ZERO: variational form
%
%
To begin, we write the variational formulation of~\eqref{algorithm} (with $ f=0 $):\\

\noindent
\hspace{1mm}
\JR{For a.e.~$ t~\in~(0,T) $,
find $ c^{k}_{i}\left (t\right ) \in M_{i} $ and $ \br_{i}^{k} \left (t\right ) \in \tilde{\Sigma_{i}} $ such that}
\vspace{-2.5mm}
\begin{equation} \label{mixed-sub}
\left .\begin{array}{lll} \frac{d}{dt} (\omega c_{i}^{k}, \mu_{i})_{\Omega_{i}}
  +  (\nabla \cdot \br_{i}^{k}, \mu_{i})_{\Omega_{i}} &= 0, & \forall \mu_{i} \in M_{i}, \\
  - (\nabla \cdot \bv_{i}, c_{i}^{k})_{\Omega_{i}} + (\bD_{i}^{-1} \br_{i}^{k}, \bv_{i})_{\Omega_{i}}
  & = \sum_{j \in \iN_{i}} \int_{\Gamma_{i,j}} c_{i}^{k} (-\bv_{i} \cdot \bn_{i}), & \forall \bv_{i} \in \tilde{\Sigma_{i}}.
\end{array}\right .
\end{equation}
%
%
%	ONE: c --> 0 and r --> 0 (L^{2} norm)
%
%
Choosing $\mu_{i}= c_{i}^{k} $ and $ \bv_{i}=\br_{i}^{k} $ in~\eqref{mixed-sub}, then adding the two resulting equations
and \JR{replacing} the boundary term by using the equation
\CJa{
\begin{multline*}
\left (-\br_{i}^{k}\cdot \bn_{i} + \alpha_{i,j} c_{i}^{k}\right )^{2}
   - \left (-\br_{i}^{k}\cdot \bn_{i} - \alpha_{j,i} c_{i}^{k}\right )^{2}\\
 = 2\left (\alpha_{i,j} + \alpha_{j,i}\right ) c_{i}^{k}\left (-\br_{i}^{k} \cdot \bn_{i}\right )
   + \left (\alpha_{i,j}^{2} - \alpha_{j,i}^{2}\right ) \left (c_{i}^{k}\right )^{2},
\end{multline*}
}
we obtain
\begin{multline*}
\frac{1}{2} \frac{d}{dt} (\omega_{i} c_{i}^{k}, c_{i}^{k})_{\Omega_{i}}
  + \left (\bD^{-1}_{i} \br_{i}^{k}, \br_{i}^{k}\right )_{\Omega_{i}}
  + \sum_{j \in \iN_{i}} \int_{\Gamma_{i,j}} \frac{1}{2\left (\alpha_{i,j}+\alpha_{j,i}\right )}
     \left (-\br_{i}^{k} \cdot \bn_{i} - \alpha_{j,i} c_{i}^{k}\right )^{2} \\ 
 = \sum_{j \in \iN_{i}} \int_{\Gamma_{i,j}} \frac{1}{2\left (\alpha_{i,j}+\alpha_{j,i}\right )} \left (-\br_{i}^{k} \cdot \bn_{i}
    + \alpha_{i,j} c_{i}^{k}\right )^{2} + \frac{1}{2} \sum_{j \in \iN_{i}}\int_{\Gamma_{i,j}} \left (\alpha_{j,i} - \alpha_{i,j}\right )
    \left (c_{i}^{k}\right )^{2}.
\end{multline*}
We then integrate over $ \left (0,t\right ) $ for a.e. $ t \in (0,T] $ and apply the Robin boundary conditions.
By using the properties of $ \omega $ and $ \bD $ and recalling that the Robin coefficients
\JR{$ \alpha_{i,j}$ belong to} $L^{\infty }\left (\Gamma_{i,j}\right )$, $ i \in 1, \cdots, I $, $ j \in \iN_{i} $,
we obtain, \JR{for some constant $ C $,}
\begin{multline*}
\omega_{-} \| c_{i}^{k}\left (t\right ) \|^{2}_{\Omega_{i}}
  + 2 \delta_{-} \int_{0}^{t} \|\br_{i}^{k}\left (s\right ) \|^{2} ds
  + \sum_{j \in \iN_{i}}  \int_{0}^{t}\int_{\Gamma_{i,j}} \frac{1}{\alpha_{i,j}+\alpha_{j,i}}
    \left (-\br_{i}^{k} \cdot \bn_{i} - \alpha_{j,i} c_{i}^{k}\right )^{2} \\
\leq \sum_{j \in \iN_{i}}  \int_{0}^{t}\int_{\Gamma_{i,j}} \frac{1}{\alpha_{i,j}+\alpha_{j,i}} \left (-\br_{j}^{k-1} \cdot \bn_{i}
  + \alpha_{i,j} c_{j}^{k-1}\right )^{2} + C \int_{0}^{t} \| c_{i}^{k}\left (s\right )\|^{2}_{\Omega_{i}} ds.
\end{multline*}
Now we sum over all subdomains and \JR{define for $k \ge 1$ and for a.e. $ t \in (0,T] $}
\begin{align*}
E^{k}\left (t\right )
  &= \sum_{i=1}^{I} \left ( \omega_{-} \| c_{i}^{k}\left (t\right ) \|^{2}_{\Omega_{i}}
      + 2 \delta_{-} \int_{0}^{t} \|\br_{i}^{k}\left (s\right ) \|^{2} ds \right ), \\
B^{k} \left (t\right )
  &= \sum_{i=1}^{I} \sum_{j \in \iN_{i}}  \int_{0}^{t}\int_{\Gamma_{i,j}} \frac{1}{\alpha_{i,j}+\alpha_{j,i}}
     \left (-\br_{j}^{k} \cdot \bn_{i} + \alpha_{i,j} c_{j}^{k}\right )^{2}.
\end{align*}
Then we have, for all $ k >0 $
$$
E^{k} \left (t\right )+ B^{k}\left (t\right ) \leq B^{k-1} \left (t\right )
  + C \sum_{i=1}^{I} \int_{0}^{t} \| c_{i}^{k}\left (s\right )\|^{2}_{\Omega_{i}} ds.
$$
Now sum over the iterates for any given $ K >0 $:
\begin{equation}\label{gronwall}
 \sum_{k=1}^{K} E^{k} \left (t\right ) \leq B^{0}\left (t\right )
   + C \sum_{k=1}^{K} \sum_{i=1}^{I} \int_{0}^{t} \| c_{i}^{k}\left (s\right )\|^{2}_{\Omega_{i}} ds, 
\end{equation}
where
\CJa{
\begin{align*}
B^{0}\left (t\right )
%&=  \sum_{i=1}^{I} \sum_{j \in \iN_{i}}  \int_{0}^{t}\int_{\Gamma_{i,j}} \frac{1}{\alpha_{i,j}+\alpha_{j,i}}
%       \left (-\br_{j}^{0} \cdot \bn_{i} + \alpha_{i,j} c_{j}^{0}\right )^{2}\\
&= \sum_{i=1}^{I} \sum_{j \in \iN_{i}}  \int_{0}^{t}\int_{\Gamma_{i,j}} \frac{1}{\alpha_{i,j}+\alpha_{j,i}} (g_{i,j})^{2},
\end{align*}
}
for $ g_{i,j} $ the initial guess
\JR{on  $\Gamma_{i,j}$}.
From the definition of $ E^{k} $, \JR{since} $ \delta_{-} >0 $, we have
$$
\sum_{k=1}^{K} \sum_{i=1}^{I} \omega_{-} \| c_{i}^{k}\left (t\right ) \|^{2}_{\Omega_{i}}
  \leq B^{0}\left (t\right ) + C \sum_{k=1}^{K} \sum_{i=1}^{I} \int_{0}^{t} \| c_{i}^{k}\left (s\right )\|^{2}_{\Omega_{i}} ds.
$$
Thus, by applying Gronwall's lemma, we obtain for any $ K>0 $ and a.e. $ t \in \left (0,T\right ) $
\begin{equation}\label{cL2}
 \sum_{k=1}^{K} \sum_{i=1}^{I} \| c_{i}^{k}\left (t\right ) \|^{2}_{\Omega_{i}}
   \leq e^{\frac{CT}{\omega_{-}}} \frac{B^{0}\left (T\right )}{\omega_{-}}.
\end{equation}
This along with~\eqref{gronwall} implies
\begin{equation} \label{rL2}
 \sum_{k=1}^{K} \sum_{i=1}^{I} 2 \delta_{-} \int_{0}^{t}  \|\br_{i}^{k}\left (s\right ) \|^{2} ds
   \leq (1+\frac{CT}{\omega_{-}}  e^{\frac{CT}{\omega_{-}}}) B^{0}\left (T\right ), \quad \forall K >0.
\end{equation}
The inequalities~\eqref{cL2},~\eqref{rL2} imply that the sequence $ c_{i}^{k} $ tends to $ 0 $
in $ L^{\infty}\left (0,T; L^{2}\left (\Omega_{i}\right )\right ) $ and $ \br_{i}^{k} $ converges to $ 0 $
in $ L^{2}\left (0,T; \pmb{L^{2}(\Omega_{i})}\right ) $ for each $ i \in 1, \cdots, I $ as
$ k \rightarrow \infty $. \\
%
%
%	TWO: \partial _{t} c --> 0, \nabla \cdot \br -->0 (L^{2}-norm)
%
%
To \JR{show} convergence in higher norms, we differentiate the first and the second equations of~\eqref{mixed-sub}
with respect to $ t $, then take
\JR{$ \mu_{i} = \partial_{t}c_{i}^{k}$ and $\bv_{i} = \partial _{t} \br_{i}^{k} $} and add
\JR{the}
resulting equations together, we see that (after bounding the left hand side using the assumptions on $ \omega $
and $ \bD $)
\begin{equation*}
\frac{\omega_{-}}{2}\frac{d}{dt} \| \partial_{t} c_{i}^{k} \|^{2}_{\Omega_{i}}
   + \delta_{-} \| \partial _{t} \br_{i}^{k} \|^{2}_{\Omega_{i}}
 \leq \sum_{j \in \iN_{i}} \int_{\Gamma_{i,j}} \partial _{t} c_{i}^{k} (-\partial_{t} \br_{i}^{k} \cdot \bn_{i} ).
\end{equation*}
We \JR{proceed} as in the previous \JR{argument} with the use of Robin boundary conditions after differentiating
with respect to $ t $
\begin{equation*}
-\partial _{t}\br_{i}^{k} \cdot \bn_{i} + \alpha_{i,j} \partial _{t}c_{i}^{k}
  = -\partial _{t}\br_{j}^{k-1} \cdot \bn_{i} + \alpha_{i,j} \partial _{t} c_{j}^{k-1},  \; \;
    \text{on} \; \Gamma_{i,j} \times \left (0,T\right ), \forall j \in \mathcal{N}_{i}. 
\end{equation*}
We then obtain, for all $ k >0 $
$$
\tilde{E}^{k} \left (t\right )+ \tilde{B}^{k}\left (t\right )
 \leq \tilde{B}^{k-1} \left (t\right )
      + C \sum_{i=1}^{I} \int_{0}^{t} \| \partial _{t} c_{i}^{k}\left (s\right )\|^{2}_{\Omega_{i}} ds.
$$
where
\begin{align*}
\tilde{E}^{k}\left (t\right )
  &= \sum_{i=1}^{I} \left ( \omega_{-} \| \partial _{t} c_{i}^{k}\left (t\right ) \|^{2}_{\Omega_{i}}
      + 2 \delta_{-}  \int_{0}^{t} \| \partial _{t} \br_{i}^{k}\left (s\right ) \|^{2} ds \right ), \\
\tilde{B}^{k} \left (t\right )
  &= \sum_{i=1}^{I} \sum_{j \in \iN_{i}}  \int_{0}^{t}\int_{\Gamma_{i,j}} \frac{1}{\alpha_{i,j}+\alpha_{j,i}}
         \left (-\partial _{t} \br_{j}^{k} \cdot \bn_{i} + \alpha_{i,j} \partial _{t} c_{j}^{k}\right )^{2}.
\end{align*}
%
%Now sum over the iterates for any $ K >0 $ 
%:
%
%\begin{equation} \label{inequal}
% \sum_{k=1}^{K} \tilde{E}^{k} \left (t\right )
%  \leq \tilde{B}^{0}\left (t\right )
%     + C \sum_{k=1}^{K} \sum_{i=1}^{I} \int_{0}^{t} \| \partial _{t}  c_{i}^{k}\left (s\right )\|^{2}_{\Omega_{i}}ds, 
% \end{equation}
%
%with
%\CJa{
%
%\begin{align*}
%\tilde{B}^{0}\left (t\right )
%&=  \sum_{i=1}^{I} \sum_{j \in \iN_{i}}  \int_{0}^{t}\int_{\Gamma_{i,j}} \frac{1}{\alpha_{i,j}+\alpha_{j,i}}
%     \left (-\partial _{t} \br_{j}^{0} \cdot \bn_{i} + \alpha_{i,j} \partial _{t} c_{j}^{0}\right )^{2} \\
%&= \sum_{i=1}^{I} \sum_{j \in \iN_{i}}  \int_{0}^{t}\int_{\Gamma_{i,j}} \frac{1}{\alpha_{i,j}
% +\alpha_{j,i}} (\partial _{t} g_{i,j})^{2}.
%\end{align*}
%}
%
%From the definition of $ \tilde{E}^{k} $, \eqref{inequal} implies
%$$
%\sum_{k=1}^{K} \sum_{i=1}^{I} \omega_{-} \| \partial _{t} c_{i}^{k}\left (t\right ) \|^{2}_{\Omega_{i}}
%  \leq \tilde{B}^{0}\left (t\right )
%     + C \sum_{k=1}^{K} \sum_{i=1}^{I} \int_{0}^{t} \| \partial _{t}  c_{i}^{k}\left (s\right )\|^{2}_{\Omega_{i}} ds.
%$$
%Now we can apply the Gronwall's lemma \CJa{to obtain} for any $ K>0 $
\MK{Now, as before, we  sum over the iterates for any $ K >0 $ and apply Gronwall's lemma
to obtain for any $ K>0 $}
and a.e. $ t \in \left (0,T\right ) $
\begin{equation} \label{ctLinfty}
 \sum_{k=1}^{K} \sum_{i=1}^{I} \| \partial _{t} c_{i}^{k}\left (t\right ) \|^{2}_{\Omega_{i}}
   \leq e^{\frac{CT}{\omega_{-}}}\frac{\tilde{B}^{0}\left (T\right )}{\omega_{-}},
%\end{equation}
%
%with
%\begin{align*}
   \quad \text{with }
\tilde{B}^{0}\left (t\right )
%&=  \sum_{i=1}^{I} \sum_{j \in \iN_{i}}  \int_{0}^{t}\int_{\Gamma_{i,j}} \frac{1}{\alpha_{i,j}+\alpha_{j,i}}
%     \left (-\partial _{t} \br_{j}^{0} \cdot \bn_{i} + \alpha_{i,j} \partial _{t} c_{j}^{0}\right )^{2} \\
= \sum_{i=1}^{I} \sum_{j \in \iN_{i}}  \int_{0}^{t}\int_{\Gamma_{i,j}} \frac{1}{\alpha_{i,j}+\alpha_{j,i}} (\partial _{t} g_{i,j})^{2}.
%\end{align*}
\end{equation}
%This implies that the sequence $ \partial _{t} c_{i}^{k} $ converges to $ 0 $ in
%$ L^{\infty}\left (0,T; L^{2}\left (\Omega_{i}\right )\right ) $, which along with~\eqref{cL2}
\CJa{This along with~\eqref{cL2} } \JR{shows} that the sequence
$ c_{i}^{k} $ converges to $ 0 $ in $ W^{1, \infty}(0,T; L^{2}(\Omega_{i}) $ as $ k \rightarrow \infty $,
  for $ i~=~1,~\cdots,~I $. \\
%
%
%	div(r) --> 0
%
Now we choose $ \mu_{i} = \nabla \cdot \br_{i}^{k} $ in the first equation of~\eqref{mixed-sub} to obtain
for a.e. $ t \in (0,T) $
\begin{equation*}
\| \nabla \cdot \br_{i}^{k} \|^{2} = -\left (\partial _{t} c_{i}^{k}, \nabla \cdot \br_{i}^{k} \right ) \leq  \| \partial _{t} c_{i}^{k}\| \, \| \nabla \cdot \br_{i}^{k} \|. 
\end{equation*}
or
$$
\| \nabla \cdot \br_{i}^{k} \| \leq  \| \partial _{t} c_{i}^{k}\| \quad \forall t \in (0,T).
$$
Hence, by~\eqref{ctLinfty} we have
\begin{equation} \label{divrLinfty}
\| \nabla \cdot \br_{i}^{k} \| _{L^{\infty}(0,T; L^{2}(\Omega_{i}))}
   \rightarrow 0 \; \text{as} \; k \rightarrow \infty.
\end{equation}
\CJ{This shows that the sequence $ \br_{i}^{k} $ converges to $ 0 $
in $ L^{2}\left (0,T; H (\Div, \Omega_{i})\right ) $}. Moreover, it follows from the definition of
$ \tilde{E}^{k} $ and~\eqref{ctLinfty} that 
\begin{equation*} \label{rtL2}
 \sum_{k=1}^{K} \sum_{i=1}^{I} 2 \delta_{-} \int_{0}^{t} \| \partial _{t} \br_{i}^{k}\left (s\right ) \|^{2} ds
   \leq (1+\frac{CT}{\omega_{-}} e^{\frac{CT}{\omega_{-}}}) \tilde{B}^{0}\left (T\right ), \quad \forall K > 0.
\end{equation*}
So that the sequence $ \partial_{t} \br_{i}^{k} $ also converges to $ 0 $
in $ L^{2} (0,T; \pmb{L^{2}(\Omega_{i})}) $. 
\qquad
\end{proof}
%
%
%
% -----------------------------------------------------------------------------
%
%    SECTION 5: Nonconforming time discretizations and projections in time
%	
% -----------------------------------------------------------------------------
%
\section{Nonconforming time discretizations and projections in time}\label{Sec:Time}
One of the main \JR{advantages of} Method~1 or Method~2 is that these methods are global in time and thus enable
the use of independent time discretizations in the subdomains.
\JR{At the space-time interface, data is transferred from one space-time subdomain to a neighboring subdomain}
by using a suitable projection. 

We consider semi-discrete problems in time with nonconforming time grids. Let $ \iT_{1} $ and $ \iT_{2} $
be two \MK{possibly} different partitions of the time interval $ (0,T) $ into sub-intervals (see Figure \ref{Fig:Time}).
We denote by $ J_{m}^{i} $ the time interval $ (t_{m-1}^{i}, t_{m}^{i}] $ and by
$ \Delta t^{i}_{m} := (t_{m}^{i} - t_{m-1}^{i}) $ for $ m=1, \hdots, M_{i} $ and $ i=1,2 $,
  \JR{where for simplicity of exposition we have again supposed that we have only two subdomains}.
We use the lowest order discontinuous Galerkin method \cite{BlayoHJ, OSWRDG, thomee1997galerkin},
which is a modified backward Euler method. The same idea can be generalized to \JR{higher order methods}. 
\begin{figure}[h]
\vspace{1.2cm}
\centering
\begin{minipage}[c]{0.5 \linewidth}
\setlength{\unitlength}{1pt} 
\begin{picture}(140,70)(0,0)
\thicklines
\put(0,3){\includegraphics[scale=0.55]{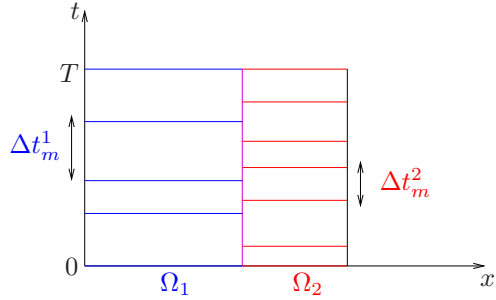} \\}
\put(-1,1){$ 0 $}
\put(-3,74){$ T $}
\put(35,-5){\textcolor{blue}{$ \Omega_{1} $}}
\put(85,-5){\textcolor{red}{$ \Omega_{2}$}}
\put(-22,47){\textcolor{blue}{$\Delta t^{1}_{m} $}}
\put(118,33){\textcolor{red}{$ \Delta t^{2}_{m} $}}
\put(156,-3){$ x $}
\put(1,98){$ t $}
\end{picture}
\end{minipage}
\caption{Nonconforming time grids in the subdomains.}
\label{Fig:Time} \vspace{-0.4cm}
\end{figure}
\noindent We denote by $ P_{0}(\mathcal{T}_{i}, W) $ the space of piecewise constant functions in time on grid
$ \mathcal{T}_{i} $ with values in $ W $, where $ W = H^{\frac{1}{2}}(\Gamma) $ for Method~1 and
$ W = L^{2}(\Gamma) $ for Method~2: \vspace{-0.1cm}
\CJ{
$$
P_{0}(\mathcal{T}_{i}, W) = \left \{ \phi: (0,T) \rightarrow W,
\phi \text{ is \JR{constant on} } J_{m}^{i}, \ \forall m=1, \dots, M_{i} \right \}. \vspace{-0.1cm}
$$
}
In order to exchange data on the space-time interface between different time grids, we define the following
$ L^{2} $ projection $ \Pi_{ji} $ from  $ P_{0} (\mathcal{T}_{i}, W) $ onto $ P_{0}(\mathcal{T}_{j},W) $
(see \cite{OSWRwave, OSWRDG})~: for $ \phi \in P_{0} (\mathcal{T}_{i}, W)$,
$ \Pi_{ji} \phi \hspace{-2pt} \mid_{J^{j}_{m}} $ is the average value of $ \phi $ on $ J^{j}_{m} $,
for $ m=1, \dots, M_{j} $: 
\begin{equation*}
\Pi_{ji} \left ( \phi \right )\mid_{J_{m}^{j}}
  = \frac{1}{\mid J^{j}_{m}\mid} \sum_{l=1}^{M_{i}} \int_{J^{j}_{m} \cap J^{i}_{l}} \phi. 
\end{equation*}
We use the algorithm described in \cite{Projection1d} for effectively performing this projection.
With these tools, we are now able to weakly enforce the transmission conditions over the time intervals. \\
We still denote by $ (c_{i}, \br_{i}) $, for $ i=1,2 $, the solution of the \JR{problem semi-discrete in time
corresponding to problem~\eqref{Schursub-mixed} or~\eqref{Schwarzsub-mixed}.} 
\subsection{For Method~1} 
As there is only one unknown $ \lambda $ on the interface, we need to choose $ \lambda $ piecewise constant in time
on one grid, either $ \iT_{1} $ or $ \iT_{2} $. For instance, let $ \lambda \in P_{0} (\iT_{2}, H^{\frac{1}{2}} (\Gamma)) $
and take $ c_{2} = \Pi_{22} (\lambda) = \CJ{\text{Id}} (\lambda) $. The \JR{weak continuity}
of the concentration in time across the interface
is fulfilled by letting 
$$
c_{1} = \Pi_{12} (\lambda ) \in P_{0} (\iT_{1}, H^{\frac{1}{2}} (\Gamma)).
$$
The semi-discrete (nonconforming in time) counterpart of the flux continuity
\CJ{in the second equation of~\eqref{physicsTC}}
is weakly enforced by integrating it over each time interval $ J_{m}^{2} $ of grid $ \iT_{2} $ :
$\forall m=1,...,M_{ 2}$,
\begin{equation} \label{IPschurTime}
 \int _{\Gamma} \int_{J^{2}_{m}} \biggl(\Pi_{21} \bigl (\pmb{r}_{1} (\Pi_{12} (\lambda), f, c_{0}) \cdot \pmb{n}_{1} \bigr)+ \Pi_{22}\bigl (\pmb{r}_{2} (\Pi_{22}(\lambda), f, c_{0}) \cdot \pmb{n}_{2}\bigr) \biggr)\, dt = 0.
\end{equation}
\textit{Remark.}
\JR{Obviously one can choose $ \lambda $ to be constant in time on yet}
another grid (neither $ \iT_{1} $ nor $ \iT_{2} $),
\MK{and this can be useful}
in some applications (e.g. flow in porous media with fractures). 
\subsection{For Method~2}
In Method~2, there are two interface unknowns representing the Robin terms \CJa{from each subdomain}.
Thus we let $ \xi_{i} \in P_{0} (\iT_{i}, L^{2}(\Gamma)) $, for $ i=1,2 $. The semi-discrete in time counterpart
of~\eqref{RobinTCs} is weakly enforced \MK{as follows}:
\begin{equation} \label{IPschwarzTime}
\begin{array}{ll}
\int_{\Gamma} \int_{J^{1}_{m}} \biggl (\xi_{1} - \Pi_{12}
  \bigl (-\pmb{r}_{2} (\xi_{2}, f, c_{0}) \cdot \pmb{n}_{1} + \alpha_{1,2} c_{2} (\xi_{2}, f, c_{0})\bigr)\biggr)
  \, dt = 0,  & \forall m=1, \cdots, M_{1},\\
\int_{\Gamma} \int_{J^{2}_{m}} \biggl ( - \Pi_{21}
  \bigl (-\pmb{r}_{1} (\xi_{1}, f, c_{0}) \cdot \pmb{n}_{2} + \alpha_{2,1} c_{1} (\xi_{1}, f, c_{0})\bigr )+\xi_{2}\biggr)
  \, dt = 0, & \forall m=1, \cdots, M_{2},
\end{array} \vspace{-0.2cm}
\end{equation}
where  $\left (c_{i}(\xi_{i}, f, c_{0}), \br_{i} (\xi_{i}, f, c_{0}) \right ), \; i=1,2 $ is the solution
\JR{to~\eqref{Schwarzsub-mixed}}.\\\\
\textit{Remark.} For conforming time grids, the two schemes defined by applying GMRES for the
two interface problems~\eqref{IPschurTime},~\eqref{IPschwarzTime} respectively converge to the
same monodomain solution. In \MK{the} nonconforming case, due to different projections,
the two schemes become different and in the next section, we will study and
compare the errors in time for the two approaches. 
%
% ---------------------------------------------------------
%
%    SECTION 6: NUMERICAL RESULTS
%	
% ---------------------------------------------------------
%
\section{Numerical results}\label{Sec:Num}
In this section, we \MK{carry out}  numerical experiments in 2D to illustrate the performance of the two methods
presented above. We consider $ \bD = d \pmb{I} $ isotropic and \JR{constant on each subdomain}, where $ \pmb{I} $
is the 2D identity matrix. Consequently, we \JR{may} denote by $ d_{i} $, the diffusion coefficient in the subdomains.
For the spatial discretization, we use mixed finite elements  with the lowest order Raviart-Thomas spaces
on rectangles \cite{brezzi1991mixed, RobertsThomas}. \vspace{3pt}

In the first test problem (see Section~\ref{Subsec:test1}), we consider the two subdomain case with discontinuous
coefficients. We vary the jumps in the diffusion coefficients and we see how it affects the convergence speed.
We also \MK{analyze the behavior of the error} versus the time steps in the nonconforming case.
In the second test problem (see Section~\ref{Subsec:test2}), suggested by ANDRA as a first \JR{step} towards repository
simulations, we consider several subdomains. We observe how \JR{both} methods handle this application with the strong
heterogeneity and long time computations. 
%
%
% Test 1
%
%
\subsection{\JR{A two} subdomain case}\label{Subsec:test1}
The computational domain $ \Omega$ is the unit square,
and the final time is $ T=1 $.  We split $ \Omega $ into two nonoverlapping subdomains
$ \Omega_{1} = (0,0.5) \times (0,1) $ and $ \Omega_{2} = (0.5,1) \times (0,1) $ as depicted
in Figure~\ref{Fig:Test1domain}.
\begin{figure}[h]
\vspace{0.5cm}
\centering
\begin{minipage}[c]{0.4\linewidth}
\setlength{\unitlength}{1pt} 
\begin{picture}(140,70)(0,0)
\thicklines
\put(40,-30){\includegraphics[scale=0.7]{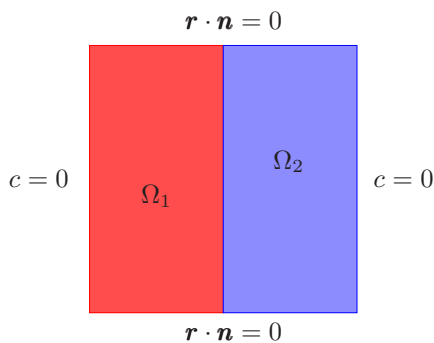} \\}
\put(60,12){$ \Omega_{1} $}
\put(110,25){$ \Omega_{2} $}
\put(10, 18){$ c=0 $}
\put(148, 18){$ c=0 $}
\put(77, -40){$ \br \cdot \bn =0 $}
\put(77, 78){$ \br \cdot \bn =0 $}
\end{picture}
\end{minipage} \vspace{1.3cm}\\
\caption{Domain decomposition and boundary conditions.}
\label{Fig:Test1domain} \vspace{-0.4cm}
\end{figure}
The initial condition is $ c_{0}=\exp\left ((x-0.55)^{2}+0.5(y-0.5)^{2}\right ) $ and the right-hand side is $ f=0 $.
The porosity is \MK{$ \omega _{1} = \omega_{2} =1 $}, the diffusion coefficients are $ d_{1}  $ and $ d_{2} $
in $ \Omega_{1} $ and $ \Omega_{2} $ respectively ($ d_{1} \neq d_{2} $). We fix $ d_{2} = 0.2 $ and vary $ d_{1} $
as shown in Table \ref{Tab:Test1Diff}. \JR{We let $ \ratioD $ denote the diffusion ratio $ d_{2}/d_{1} $.} For the spatial discretization, we use a uniform rectangular mesh
with size $ \Delta x_{1} = \Delta x_{2} = 1/200 $. For the time discretization, we use nonconforming time grids
with $ \Delta t_{1}$ and $\Delta t_{2} $, given in Table \ref{Tab:Test1Diff}, adapted to different
\JR{diffusion ratios}.
\begin{table}[h]
\centering
\begin{tabular}{|l|l|l|l|l|}
  \hline
	    		\JR{$\ratioD$}                      				& $ d_{1} $               					& $ 1/\Delta t_{1} $                    			 & $d_{2}$    & $ 1/\Delta t_{2} $   \\ \hline
	    		\textcolor{blue} {$10$}					& \textcolor{blue} {$ 0.02 $}      & \textcolor{blue}{$150$}                       & $ 0.2 $     & $ 200 $  \\ \hline 
	    		\textcolor{blue}{$100$}					& \textcolor{blue}{$ 0.002 $}      & \textcolor{blue}{$50$}                        & $ 0.2 $     & $ 200 $  \\ \hline 
	    		\textcolor{blue} {$1000$}				& \textcolor{blue}{$ 0.0002 $}      & \textcolor{blue} {$20$}                     & $ 0.2 $     & $ 200 $  \\ \hline 
\end{tabular}
\caption{Diffusion coefficients and corresponding nonconforming time steps.}
\label{Tab:Test1Diff} \vspace{-0.7cm}
\end{table}
We first analyze the convergence \MK{behavior} of \JR{each method}. We solve a problem
\JR{with~$ c_{0} = 0$ and $f = 0 $ (thus $c=0$ and $\br=0$).
We start with a random initial guess on the space-time interface.
We remark that one iteration of Method~1 with the preconditioner
costs twice as much as one iteration of Method~2 (in terms of number of subdomain solves).
Thus to compare \JR{the two} approaches, we plot the error (in logarithmic scale)
in the $ L^{2} (0,T; L^{2}(\Omega)) $-norm of the concentration $ c $ and the vector field $ \br $, 
versus the number of subdomain solves
(instead of versus the number of iterations). We stop the iteration when the errors (both in $ c $ and $ \br $) are less than $ 10^{-6} $.}
In Figure~\ref{Fig:Test1Convergence}, the convergence
of the two methods \JRo{(with GMRES)} for different diffusion ratios is shown. 
\begin{figure}[h]
\begin{flushleft}
\begin{minipage}[c]{0.30 \linewidth}
\begin{center}
\includegraphics[scale=0.19]{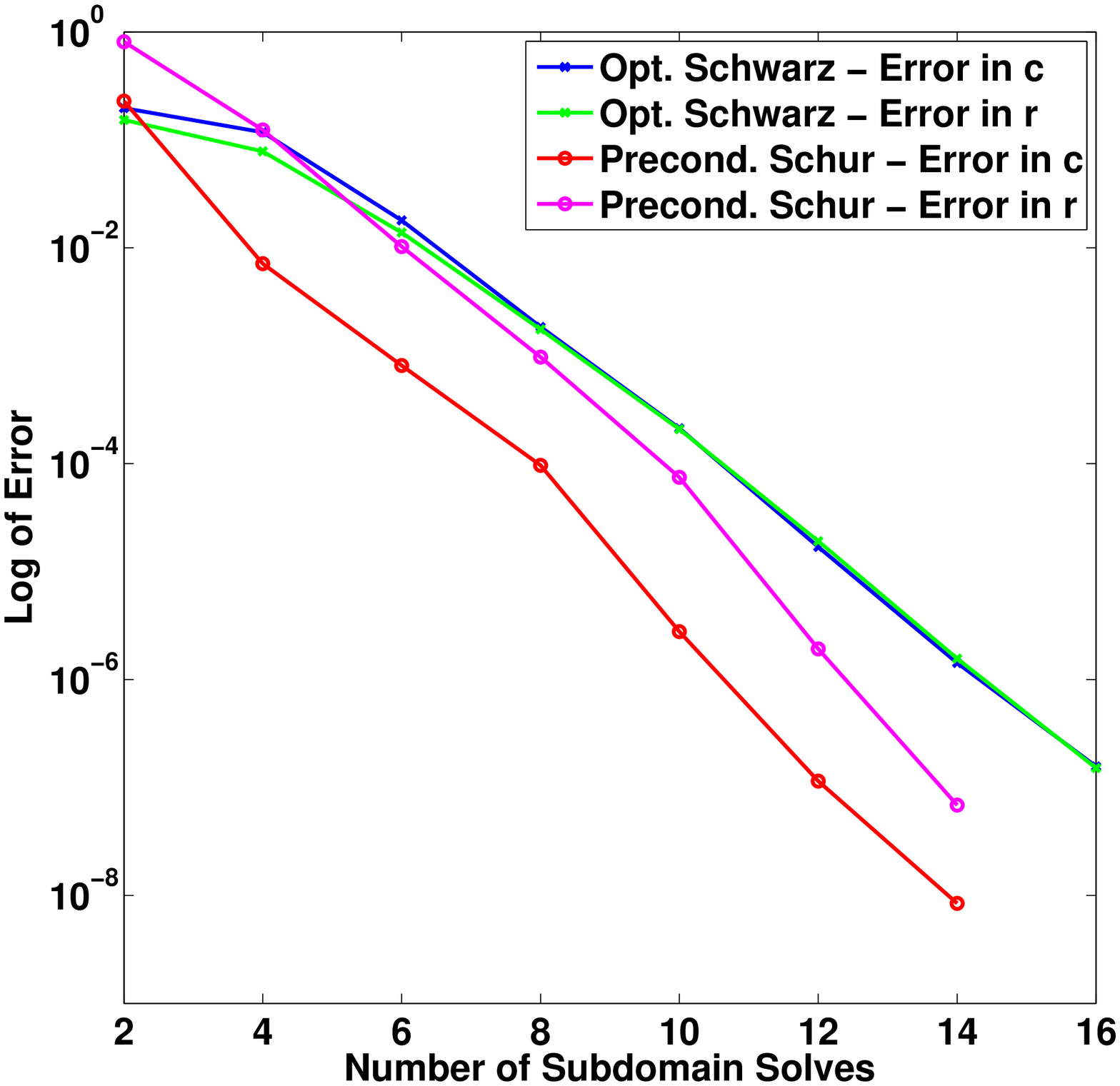}\\
\textcolor{blue}{\hspace{0.8cm} $\ratioD = 10 $}
\end{center}
\end{minipage} \hspace{5pt}
\begin{minipage}[c]{0.30 \linewidth}
\begin{center}
\includegraphics[scale=0.19]{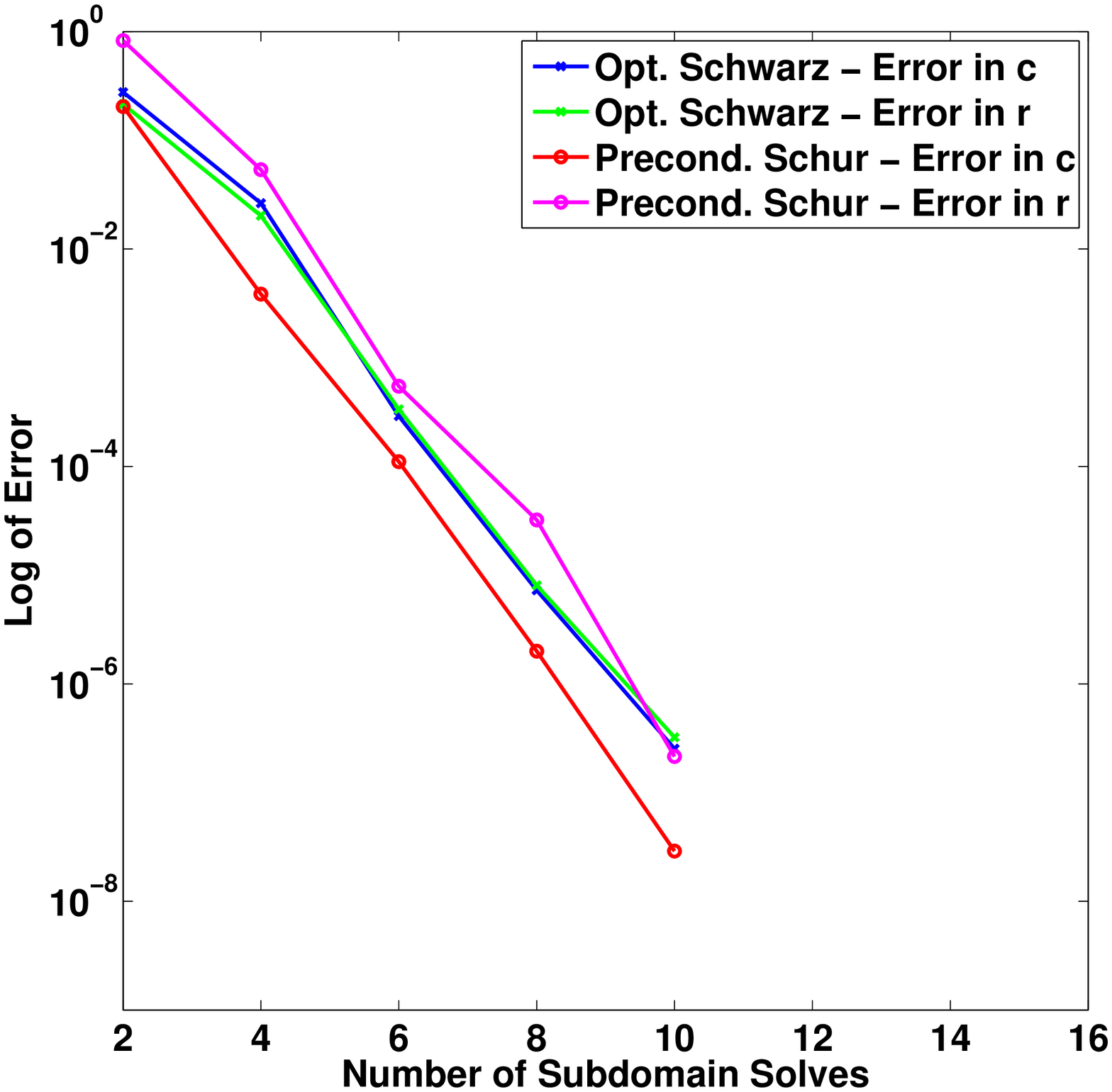}\\
\textcolor{blue}{\hspace{0.8cm} $\ratioD = 100 $}
\end{center}
\end{minipage} \hspace{5pt}
\begin{minipage}[c]{0.30 \linewidth}
\begin{center}
\includegraphics[scale=0.19]{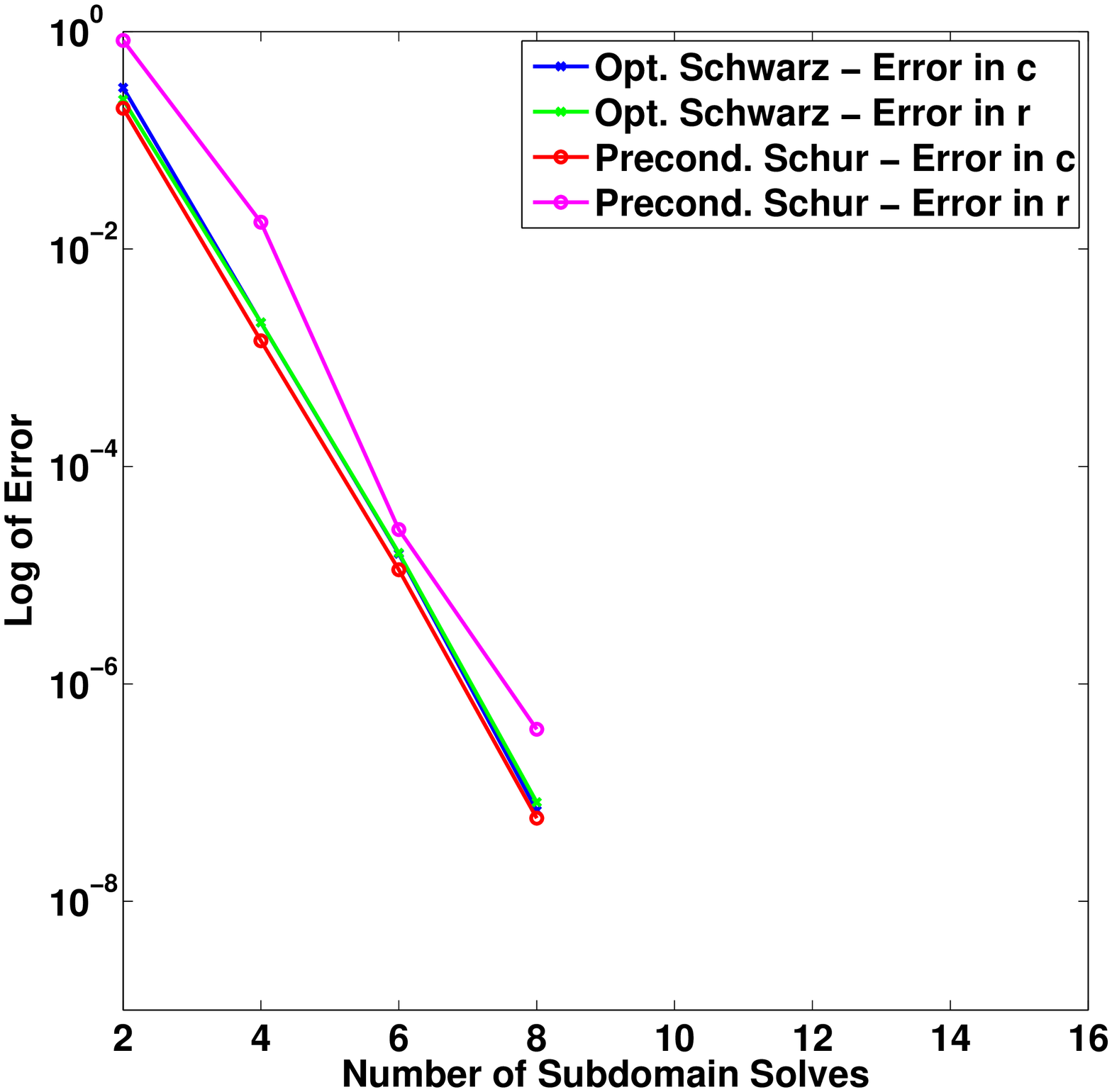}\\
\textcolor{blue}{\hspace{0.8cm} $\ratioD = 1000 $}
\end{center}
\end{minipage} 
\end{flushleft}
\vspace{-0.2cm}
\caption{Convergence curves for different \JR{diffusion ratios}: errors in $ c $ for Method~1 (red) and Method~2 (blue);
  errors in $ \br $ for Method~1 (magenta) and Method~2~(green).} 	
	\label{Fig:Test1Convergence} \vspace{-0.5cm}
\end{figure}
We see that both methods work well.
Method~1 (Schur) converges faster than Method~2 (Schwarz) for small \JR{diffusion ratios $ \ratioD $}. However, when \JR{$ \ratioD $} is increased,
they are comparable. %and the larger the contrast is, the fewer number of subdomain solves is needed to converge to the same tolerance. 
We also observe that the errors in $ c $ and $ \br $ are nearly the same for Method~2 while the error in $ \br $
is greater than the error in $ c $ for Method~1. 
Both methods handle the heterogeneities efficiently. To obtain such a good performance, we have used the following
formula for calculating the weights in~\eqref{NNPrecondweight} (see~\cite{Mandelweights})
\begin{equation*}
\sigma_{i} = \left (\frac{d_{i}}{d_{1}+d_{2}}\right )^{2}, \quad i=1,2.
\end{equation*}
Consider now the case with \JR{$ \ratioD = 10 $}. For Method~2, we vary Robin parameters $ \alpha_{1,2} $ and $ \alpha_{2,1} $
and plot the logarithmic scale of the residual after $ 20 $ Jacobi iterations in Figure \ref{Fig:Test1Robin}.
We see that the optimized Robin parameters (the red star), which are calculated by numerically minimizing
the convergence factor \cite{Bennequin, OSWR2d, OSWR1d2}, are located close to \JR{those} giving
the smallest residual after the same number of iterations.
\begin{figure}[h]
\vspace{-0.2cm}\centering
\begin{minipage}[c]{0.5 \linewidth}
\includegraphics[scale=0.3]{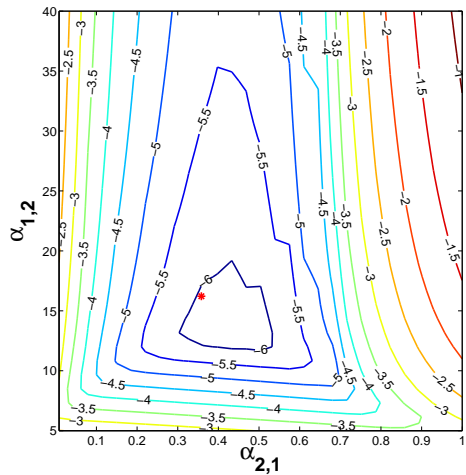}
\end{minipage} \vspace{-0.3cm}
\caption{Level curves for the residual (in logarithmic scale) after $ 20 $ Jacobi iterations for various values
  of the parameters $ \alpha_{1,2} $ and $ \alpha_{2,1} $. The red star shows the optimized parameters computed
  by numerically minimizing the continuous convergence factor. } 	
\label{Fig:Test1Robin} \vspace{-0.5cm}
\end{figure}

\MK{Next,} we analyze the accuracy in time for different \JR{diffusion ratios} and corresponding choices of
nonconforming time steps.
Toward this end, we consider four initial time grids (for $ \Delta t_{c} $ and $ \Delta t_{f} $ given)
\begin{itemize}
	\item Time grid 1 (fine-fine): conforming with $ \Delta t_{1} = \Delta t_{2} = \Delta t_{f} $.
	\item Time grid 2 (coarse-fine): nonconforming with $ \Delta t_{1} = \Delta t_{c} $ and
              $ \Delta t_{2} = \Delta t_{f} $.
	\item Time grid 3 (fine-coarse): nonconforming with $ \Delta t_{1} = \Delta t_{f} $ and
              $ \Delta t_{2} = \Delta t_{c} $.
	\item Time grid 4 (coarse-coarse): conforming with $ \Delta t_{1} = \Delta t_{2} = \Delta t_{c} $.
\end{itemize}
The time steps are then refined several times by a factor of 2. In space, we fix a conforming rectangular mesh
and we compute \MK{a reference solution by solving problem~\eqref{variational-mixed} directly
on a very fine time grid,} with $ \Delta t = \Delta t_{f}/ 2^{6} $.
The converged multidomain solution is such that the relative residual is smaller than $ 10^{-11} $. We show in Figures
\ref{Fig:Test1Contrast10} and \ref{Fig:Test1Contrast100} the errors in \JR{the} $ L^{2}(0,T; L^{2}(\Omega)) $-norms of
the concentration $ c $ and \JR{the} vector field $ \br $ versus the time step $ \Delta t = \max (\Delta t_{c}, \Delta t_{f}) $
for different \JR{diffusion ratios}. We only give the results \MK{for} Method~1 because the curves \JR{for Method~2} look exactly the same.
For \JR{$ \ratioD= 10 $}, we take $ \Delta t_{c} = 1/94 $ and $ \Delta t_{f} = 1/128 $; for \JR{$ \ratioD= 100 $}, we take
$ \Delta t_{c} = 1/40 $ and $ \Delta t_{f} = 1/160 $ (for \JR{$ \ratioD= 1000 $}, the same results hold for
$ \Delta t_{c} = 1/16 $ and $ \Delta t_{f} = 1/160 $ but we don't present it here). We first observe that \JR{first order convergence} is preserved in
the nonconforming case. Moreover, the error obtained in the nonconforming case (Time grid $ 2 $, in blue)
is nearly the same as in the finer conforming case (Time grid $ 1 $, in red). This means that nonconforming time grids
preserve the solution's accuracy in time and one must refine the time step where the solution varies \JR{most (i.e. where the diffusion coefficient is larger)}.
\begin{figure} [h]
\vspace{-0.3cm}
\centering
\begin{minipage}{0.45 \linewidth}
\includegraphics[scale=0.30]{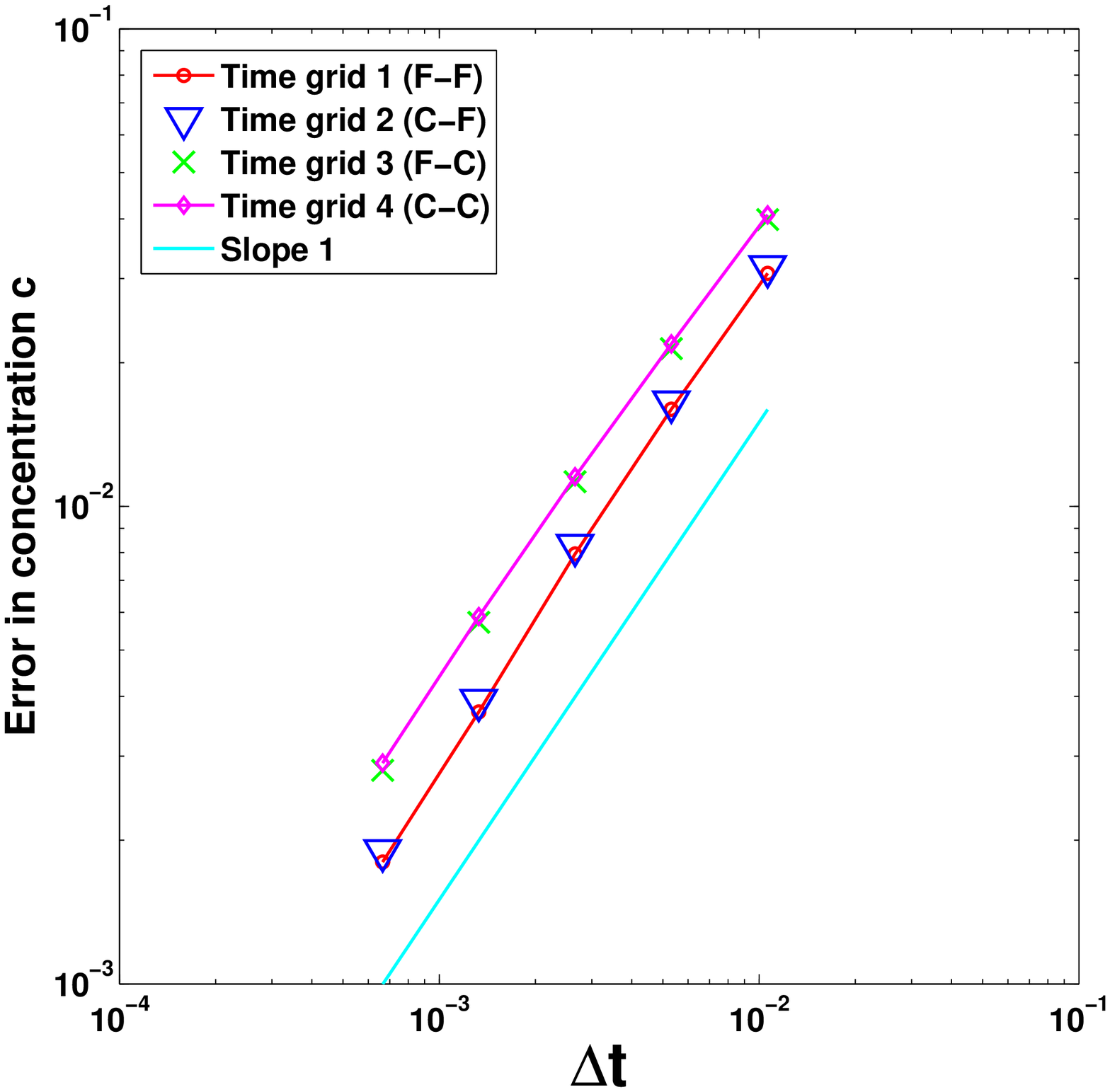} 
\end{minipage} \hspace{2pt}
\begin{minipage}{0.45 \linewidth}
\includegraphics[scale=0.30]{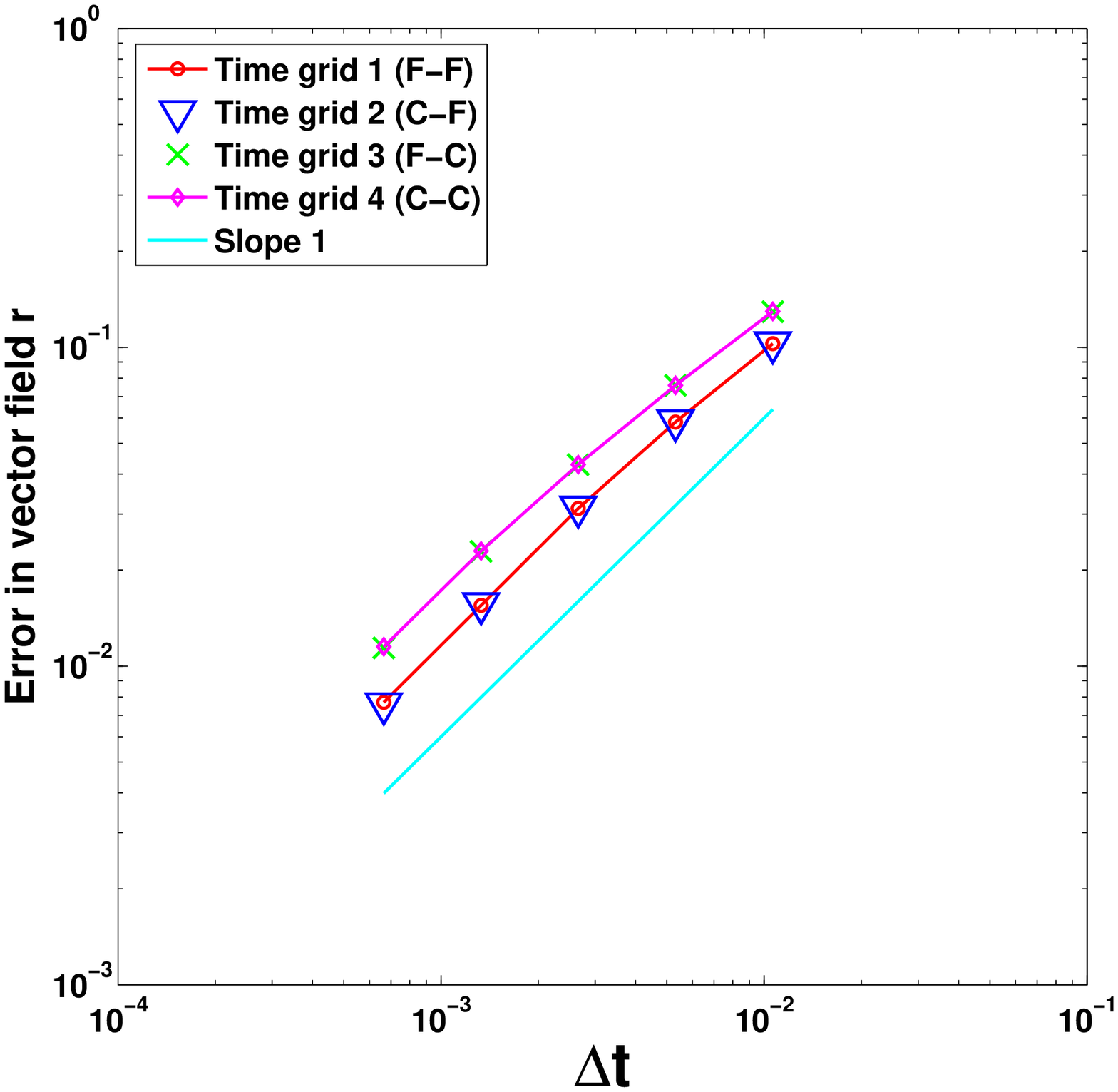} 
\end{minipage} 
\vspace{-0.2cm}
\caption{Errors in $ c $ (left) and $ \br $ (right) in logarithmic scales between \JR{the reference and the multidomain solutions} versus the time step for \JR{$\ratioD =10 $}.} 	
\label{Fig:Test1Contrast10} \vspace{-0.5cm}
\end{figure}
\begin{figure}[h]
\vspace{-0.4cm}
\centering
\begin{minipage}{0.45 \linewidth}
\includegraphics[scale=0.30]{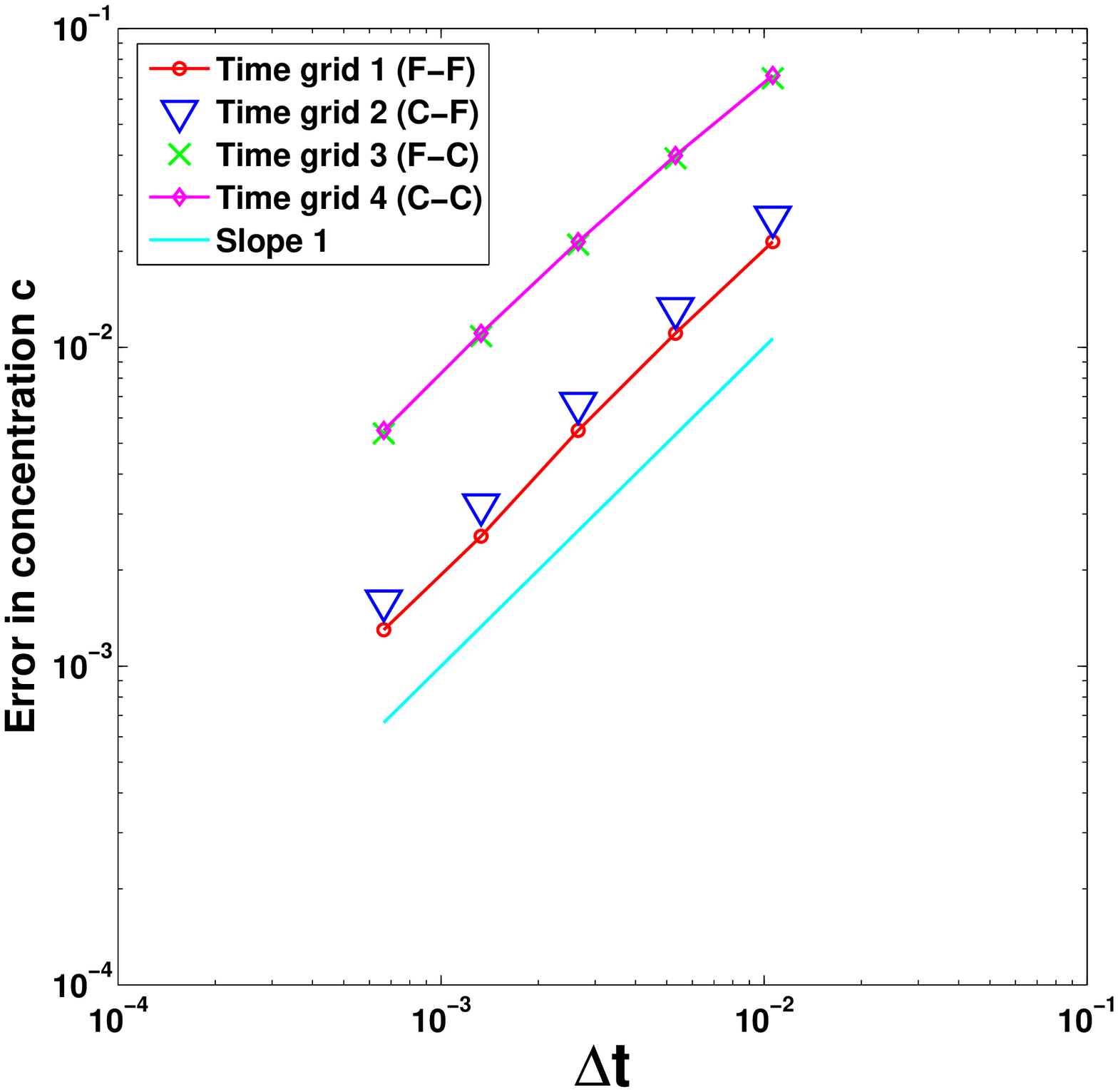} 
\end{minipage} \hspace{2pt}
\begin{minipage}{0.45 \linewidth}
\includegraphics[scale=0.30]{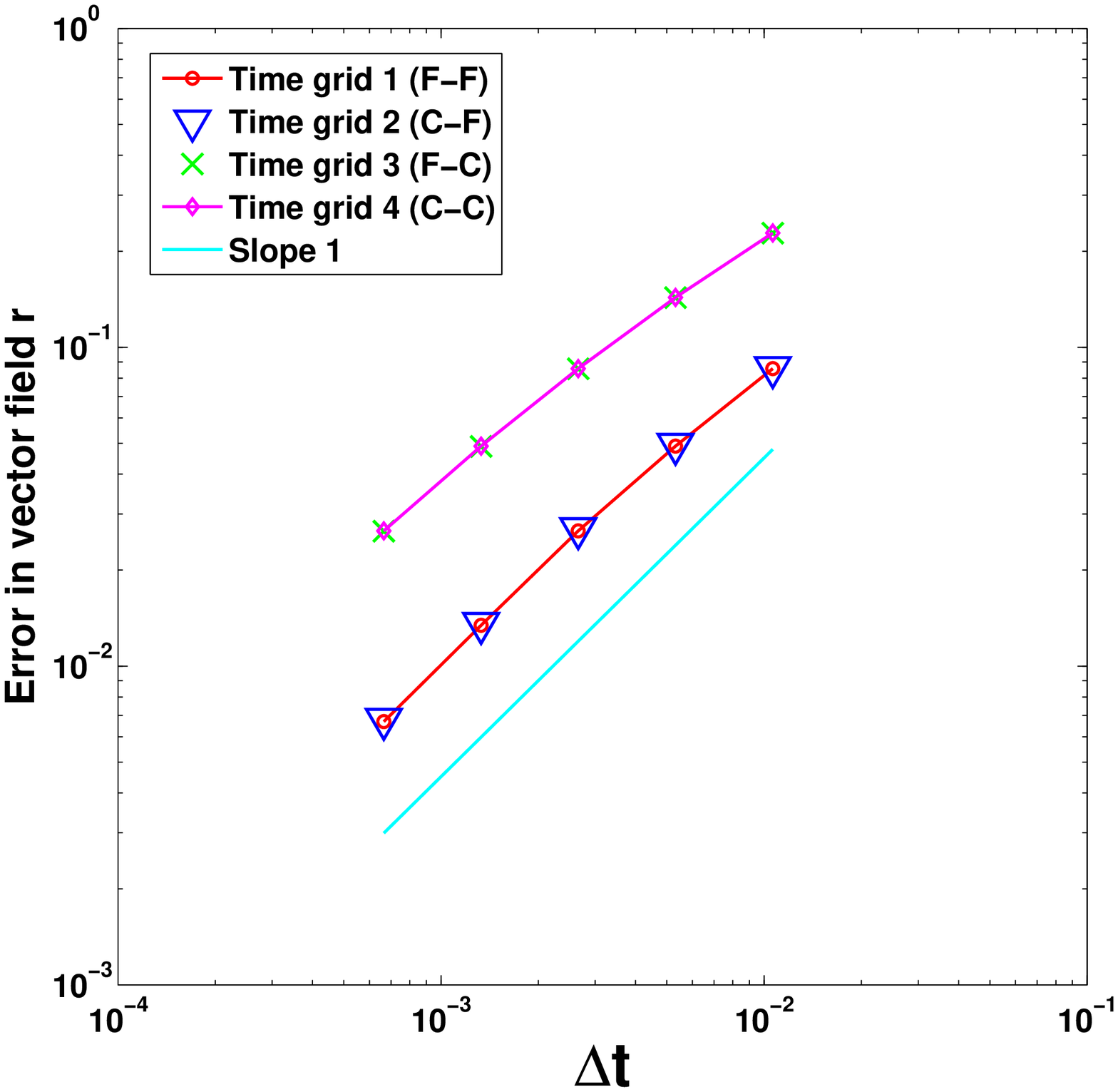} 
\end{minipage}  
\vspace{-0.2cm}
\caption{Errors in $ c $ (left) and $ \br $ (right)  in logarithmic scales between \JR{the reference and the multidomain solutions} versus the time step for \JR{$\ratioD =100 $}.} 	
\label{Fig:Test1Contrast100} \vspace{-0.7cm}
\end{figure}	
%
%
% Test 2
%
%
\subsection{\JR{A} porous medium test case}\label{Subsec:test2}
In this subsection, we consider a \JR{simplified version of a problem simulating contaminant transport in and around a nuclear waste repository site}.
The test case is described in Figure \ref{Fig:Test2domain}, where the repository is \JR{shown} in red and the clay layer
in yellow. \JR{The domain is a $ 3950 $m by $ 140 $m rectangle and the repository is a centrally located $ 2950 $m by $ 10 $m rectangle}. The initial condition is $ c_{0}=0 $, the source term is $ f=0 $ in the clay layer and 
\begin{equation}\label{sourcef} 
f = \left  \{ \begin{array}{ll} 10^{-5}  \; \text{s$ ^{-1} $} & \text{if} \; t \leq 10^5 \, \text{years},  \\
		0  & \text{if} \; t > 10^5 \, \text{years},
		\end{array} \right.  \quad \text{in the repository}.
\end{equation}
We impose homogeneous Dirichlet conditions on top and bottom, and homogeneous Neumann conditions on the left
and right hand sides. We decompose $ \Omega $ into $ 9 $ subdomains as depicted in Figure \ref{Fig:Test2Decom}
with $ \Omega_{5} $ representing the repository. The porosity is $ \omega_{5} = 0.2$
 and $ \omega _{i} = 0.05 $, $ i \neq 5 $. The diffusion coefficients are $ d_{5} = 2\, 10^{-9} $ m$ ^{2}/ $s
 and $ d_{i} = 5 \, 10^{-12} $ m$ ^{2}/ $s, $ i \neq 5 $. \JR{So the diffusion ratio is $ \ratioD=400 $}. 
\begin{figure}[h]
\centering
\includegraphics[scale=0.6]{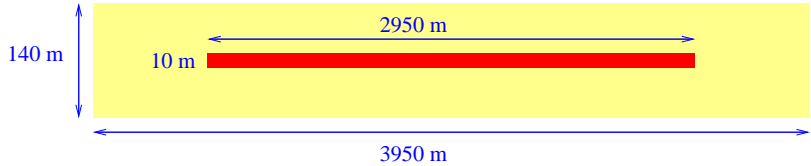}
\caption{Geometry of the domain.}
\label{Fig:Test2domain} \vspace{-0.7cm}
\end{figure}
~
\begin{figure}[h]
\centering
\begin{minipage}[c]{0.64\linewidth}
\setlength{\unitlength}{1pt} 
\begin{picture}(140,70)(0,0)
\thicklines
\put(0,-10){\includegraphics[scale=0.6]{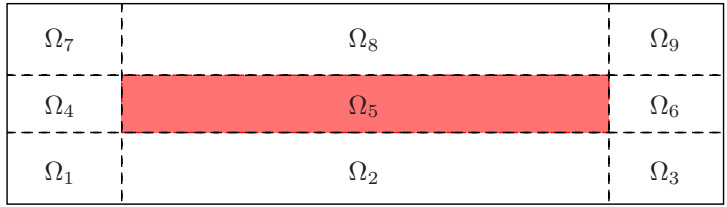} \\}
\put(15,0){$ \Omega_{1} $}
\put(130,0){$ \Omega_{2} $}
\put(244,0){$ \Omega_{3}$}
\put(15,25){$ \Omega_{4} $}
\put(130,25){$ \Omega_{5} $}
\put(244,25){$ \Omega_{6}$}
\put(15,50){$ \Omega_{7} $}
\put(130,50){$ \Omega_{8} $}
\put(244,50){$ \Omega_{9}$}
\end{picture}
\end{minipage} \vspace{0.3cm}
\caption{The decomposition into 9 subdomains (blow up in the y-direction).}
\label{Fig:Test2Decom} \vspace{-0.5cm}
\end{figure}
For the spatial discretization, we use a non-uniform but conforming rectangular mesh with a finer discretization
in the repository (a uniform mesh with $ 600 $ points in the $ x $ direction and $ 30 $ points in the $ y $ direction)
and a coarser discretization in the clay layer (the mesh size progressively increases with distance \JR{from} the repository
by a factor of $ 1.05 $). For the time discretization, we use nonconforming time grids with $ \Delta t_{5} = 2000 $
years and $ \Delta t_{i} = 10,000 $ years, $ i \neq 5 $. 
For this application, we are interested in the long-term behaviour of the repository, say \JR{over} one million years.
Thus, we test the performance of the two methods for a "short" time interval ($ T=200,000 $ years) and for
a longer time interval ($ T=1,000,000 $ years). The same time steps, $\Delta t_{i} $, are used for both cases.
As in the first test problem, we analyze the convergence results by solving a problem with $ f = 0 $. 
For Method~2, as we have a \JR{small, thin} object embedded in a large area, it has been shown in
\cite{OSWR3sub, OSWR3subDD20} that it is important to derive an adapted optimization for Robin parameters.
Thus, we consider two different optimization techniques: the classical one (Opt.~1) as used in the first test problem,
and an adapted version (Opt.~2) \cite{OSWR3sub, OSWR3subDD20} where we take into account the \JR{dimension} of the subdomains.

%To observe the performance of these two techniques numerically, 
In Figure~\ref{Fig:test2ConvGMRES} we compare the errors
in the concentration $ c $ (on the left) and in the vector field $ \br $ (on the right) both over a shorter time interval (on top) and over a longer time interval (on bottom) \JRo{where GMRES is used in all cases as the iterative solver}: Method~1 (red), Method~2 with Opt.~1 (blue) and Method~2 with Opt.~2 (green). They are comparable and perform well in the case of multiple subdomains.
We also note that the longer the time interval, the larger the number of subdomain solves needed to converge
to a given tolerance (here $ 10^{-6} $). Thus, the use of time windows (see \cite{BlayoHJ, OSWRDG2}) could \JRo{considerably improve the performance of all the algorithms, especially with} an adapted choice of the initial guess on the interface based on the solution on the previous time window.
%
%\begin{figure}[h]
%\vspace{-0.3cm}
%\centering
%\begin{minipage}{0.45 \linewidth}
%\includegraphics[scale=0.25]{Test4-ConvergenceC-nx600ny30-dt2000nt100-Gmres.eps} \\
%\includegraphics[scale=0.25]{Test4-ConvergenceC-nx600ny30-dt2000nt500-Gmres.eps} 
%\end{minipage} \hspace{2pt}
%\begin{minipage}{0.45 \linewidth}
%\includegraphics[scale=0.25]{Test4-ConvergenceC-nx600ny30-dt2000nt100-Jacobi.eps} \\
%\includegraphics[scale=0.25]{Test4-ConvergenceC-nx600ny30-dt2000nt500-Jacobi.eps} 
%\end{minipage} 
%\caption{Convergence curves for different algorithms and time intervals: with GMRES (on the left) and with Jacobi
%  (on the right), for short time $T = 200,000$ years (on top) and for long time $ T = 1,000,000 $ years (on bottom).} 	
%\label{Fig:test2Conv} \vspace{-0.5cm}
%\end{figure} 
%	
\begin{figure}[h]
\centering
\begin{minipage}{0.45 \linewidth}
\includegraphics[scale=0.25]{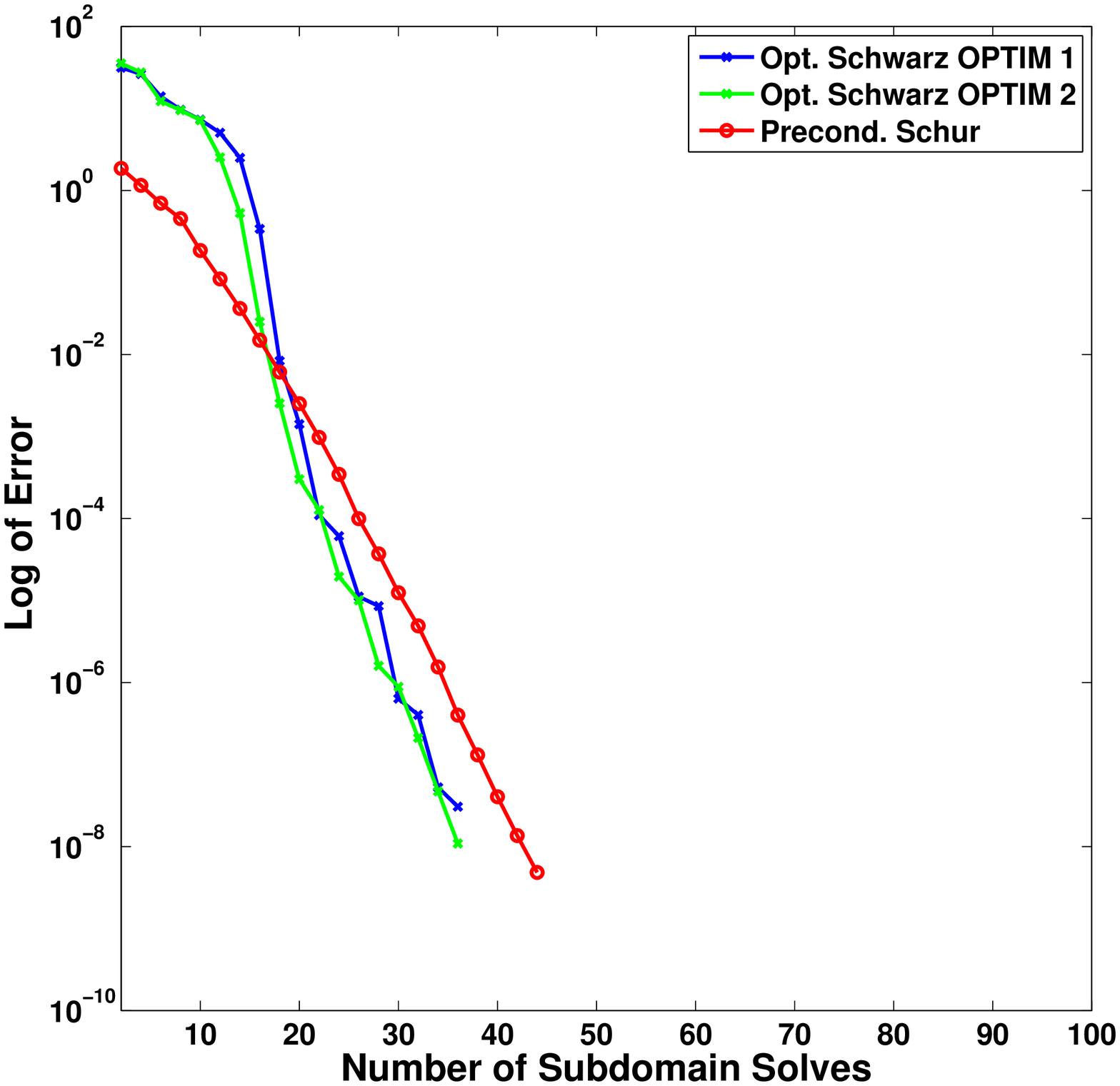} \\
\includegraphics[scale=0.25]{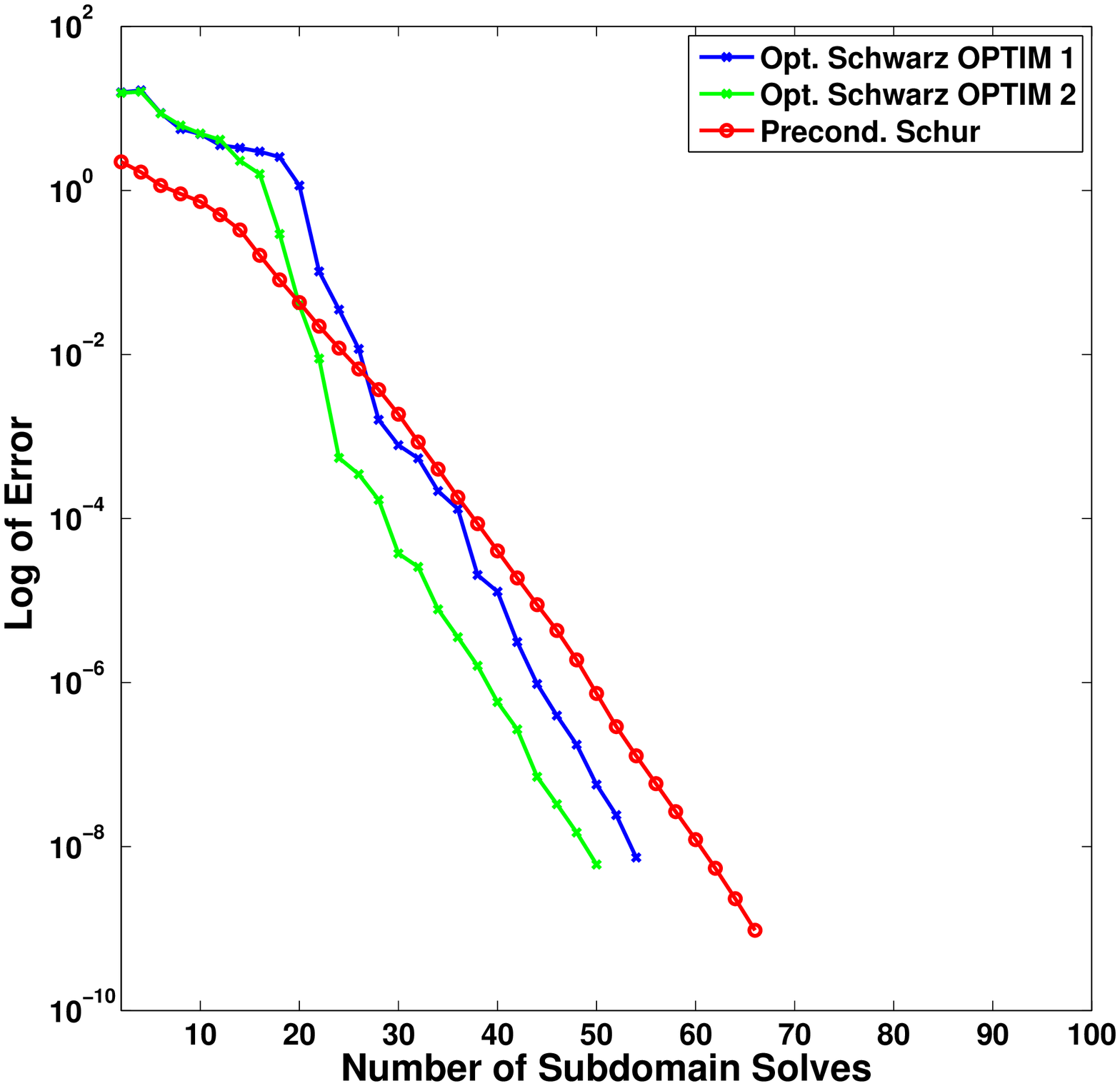} 
\end{minipage} \hspace{2pt}
\begin{minipage}{0.45 \linewidth}
\includegraphics[scale=0.25]{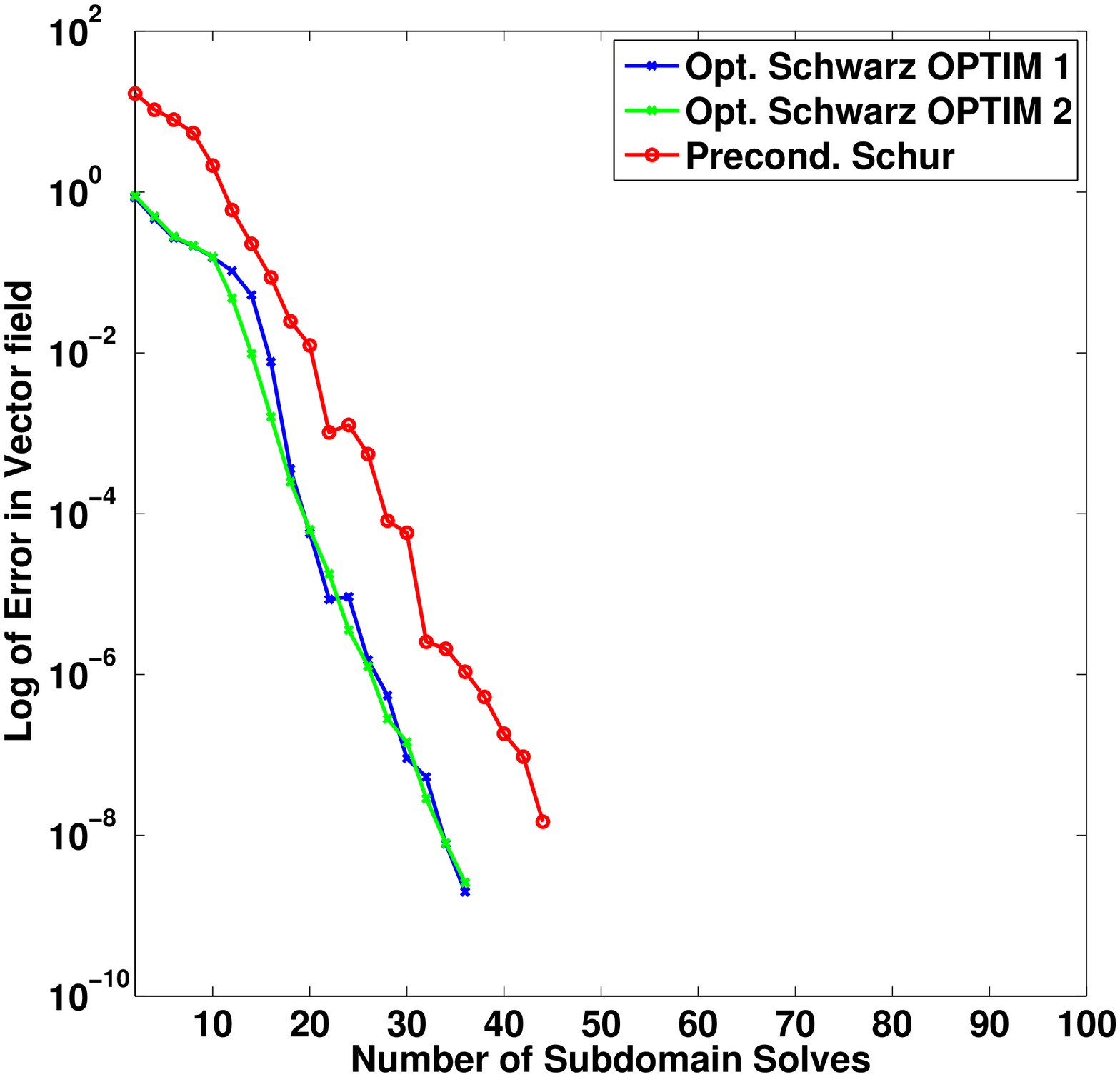} \\
\includegraphics[scale=0.25]{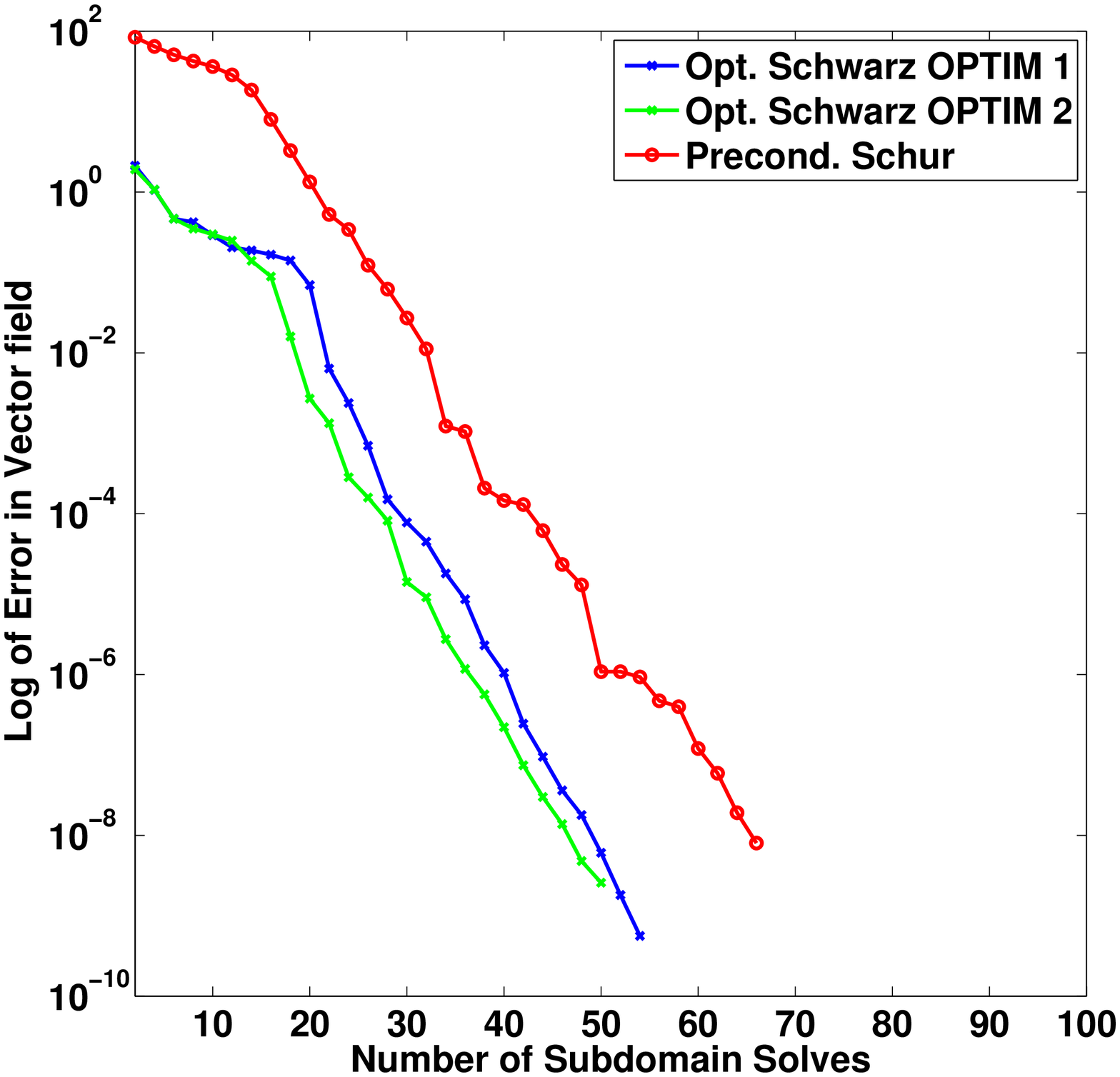} 
\end{minipage} 
\caption{Convergence curves for different time intervals with GMRES: error in $ c $ (on the left) and error in $ \br $ (on the right), for short time $T = 200,000$ years (on top) and for long time $ T = 1,000,000 $ years (on bottom).} 	
\label{Fig:test2ConvGMRES} \vspace{-0.7cm}
\end{figure} 
In Figure~\ref{Fig:test2ConvJacobi}, we plot the errors in the concentration $ c $ over different time intervals for Method~2 \JRo{with} Jacobi iteration: with Opt.~1 (blue) and Opt.~2 (green) (the errors in the vector field $ \br $ behave similarly). We observe that Opt.~2 \JRo{efficiently} handles the long time computation case while Opt.~1 doesn't.
\begin{figure}[h]
\vspace{-0.3cm}
\centering
\begin{minipage}{0.45 \linewidth}
\includegraphics[scale=0.25]{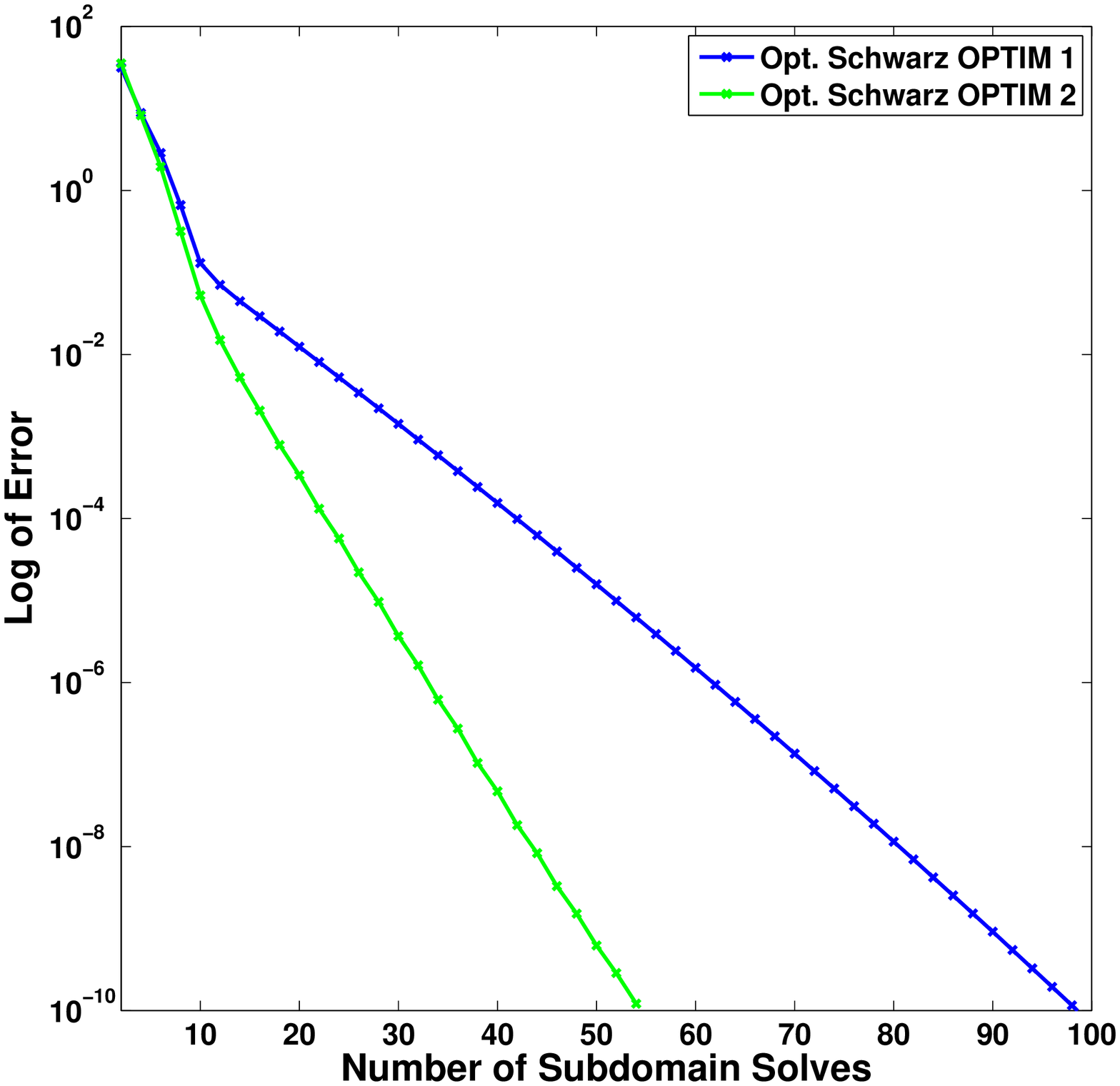}
\end{minipage} \hspace{2pt}
\begin{minipage}{0.45 \linewidth}
\includegraphics[scale=0.25]{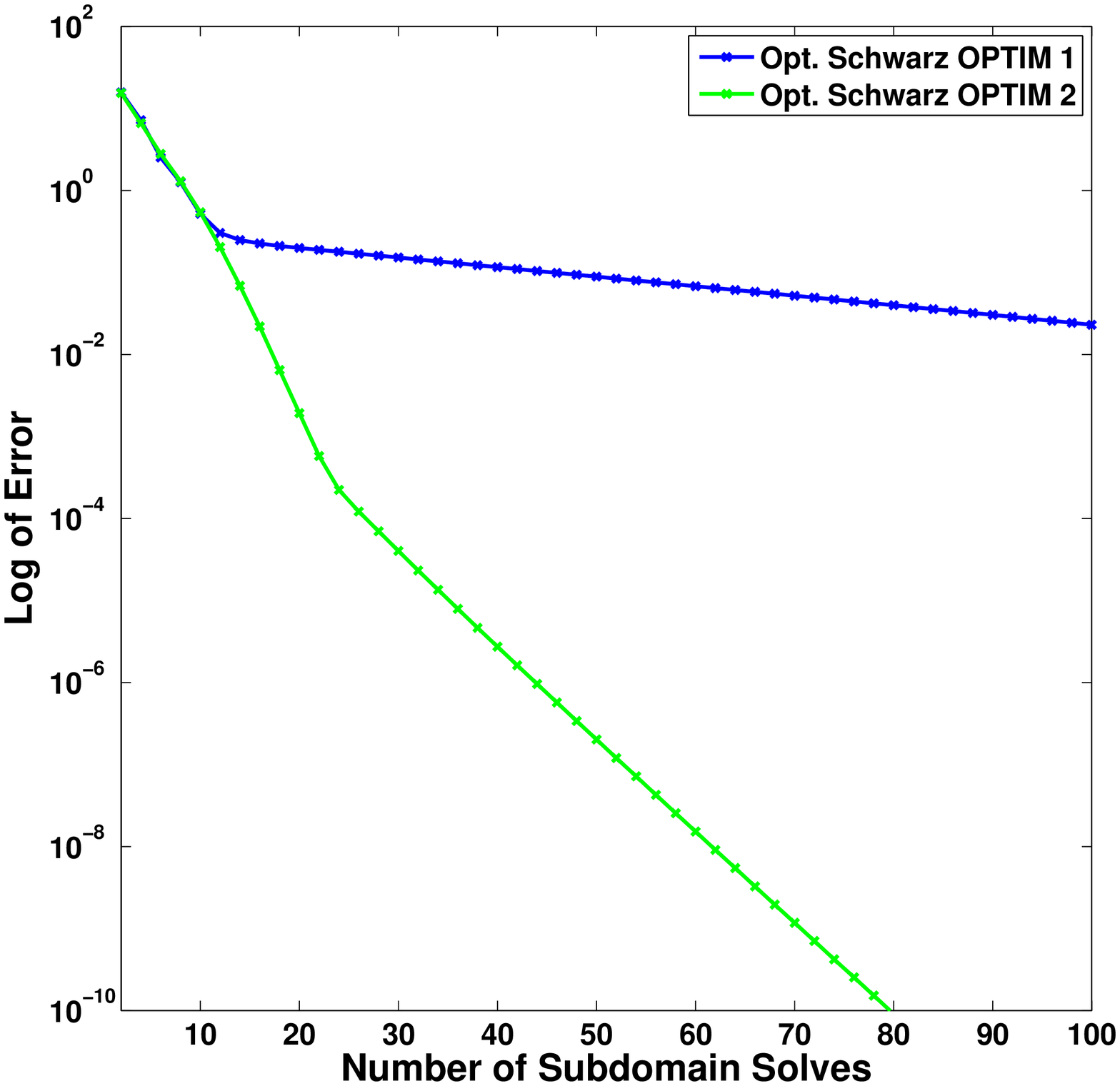} 
\end{minipage} 
\caption{Convergence curves for different time intervals using Jacobi iteration: for short time $T = 200,000$ years (on the left) and for long time $ T = 1,000,000 $ years (on the right).} 	
\label{Fig:test2ConvJacobi} \vspace{-0.5cm}
\end{figure}

Next we consider the case with $ f \neq 0 $ as defined in~\eqref{sourcef} and over the long time interval,
$ T = 1,000,000 $ years. The discretizations in space and in time (nonconforming) are the same as above.
We verify the performance of Method~1 and Method~2 (with Opt.~2) using GMRES and zero initial guess on the space-time
interfaces. The tolerance of the iteration is $ 10^{-6} $. In Figure~\ref{Fig:Sol}, the evolution of the solution at
different times is depicted (both methods give similar results).
%We remark that the color bars in these pictures are different.
As time goes on and under the effect of diffusion, the contaminant slowly migrates from the repository
to the surrounding area. Moreover, its concentration $ c $ increases until injection stops
(i.e. after 100,000 years) and then decreases. 
\begin{figure}[h]
\vspace{-0.1cm}
\centering
\begin{minipage}{2 \linewidth}
\hspace{-0.5cm}
%\includegraphics[width=0.35\textwidth]{T1.eps}
%\hspace{-1.8cm}
%\includegraphics[width=0.35\textwidth]{T2.eps}
%
%\hspace{-1.7cm}
%\includegraphics[width=0.35\textwidth]{T3.eps}
%\hspace{-1.8cm}
%\includegraphics[width=0.35\textwidth]{T4.eps}
\includegraphics[width=0.53\textwidth]{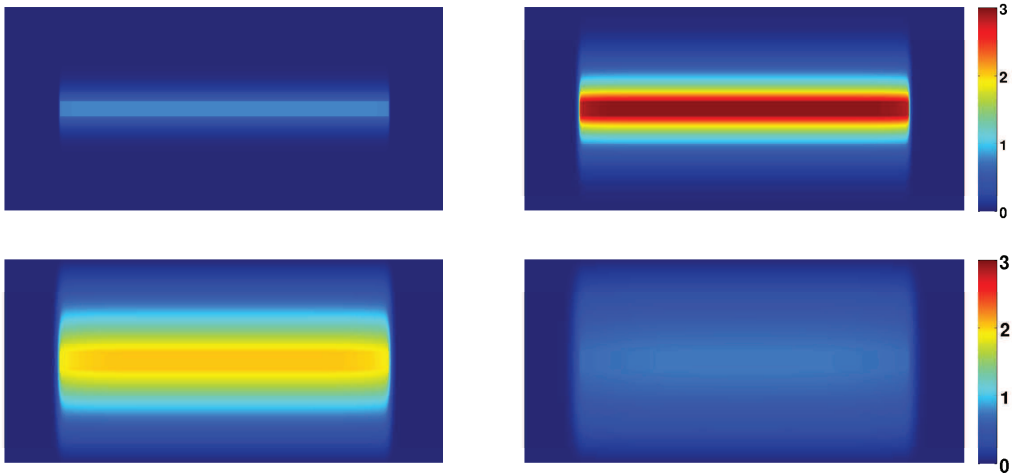}
\end{minipage}
\caption{\CJ{Snapshots of the multi-domain solution after 20,000 years (top left), 100 000 years (top right),
    200 000 years (bottom left), and 1,000,000 years (bottom right), with a
         blow up in the y-direction.}} 	
\label{Fig:Sol} \vspace{-0.8cm}
\end{figure}
In Figure~\ref{Fig:test2relres} the relative residuals \MK{for}~each method versus the number of subdomain solves
are shown, as the monodomain solution with nonconforming grids is unknown. Both methods work well and we observe
that \JR{Method 1 converges linearly while Method 2 initially converges extremely rapidly, the convergence becoming linear after the first few iterations}.
\begin{figure}[h]
\vspace{-0.3cm}
\centering
\begin{minipage}{0.45 \linewidth}
\includegraphics[scale=0.27]{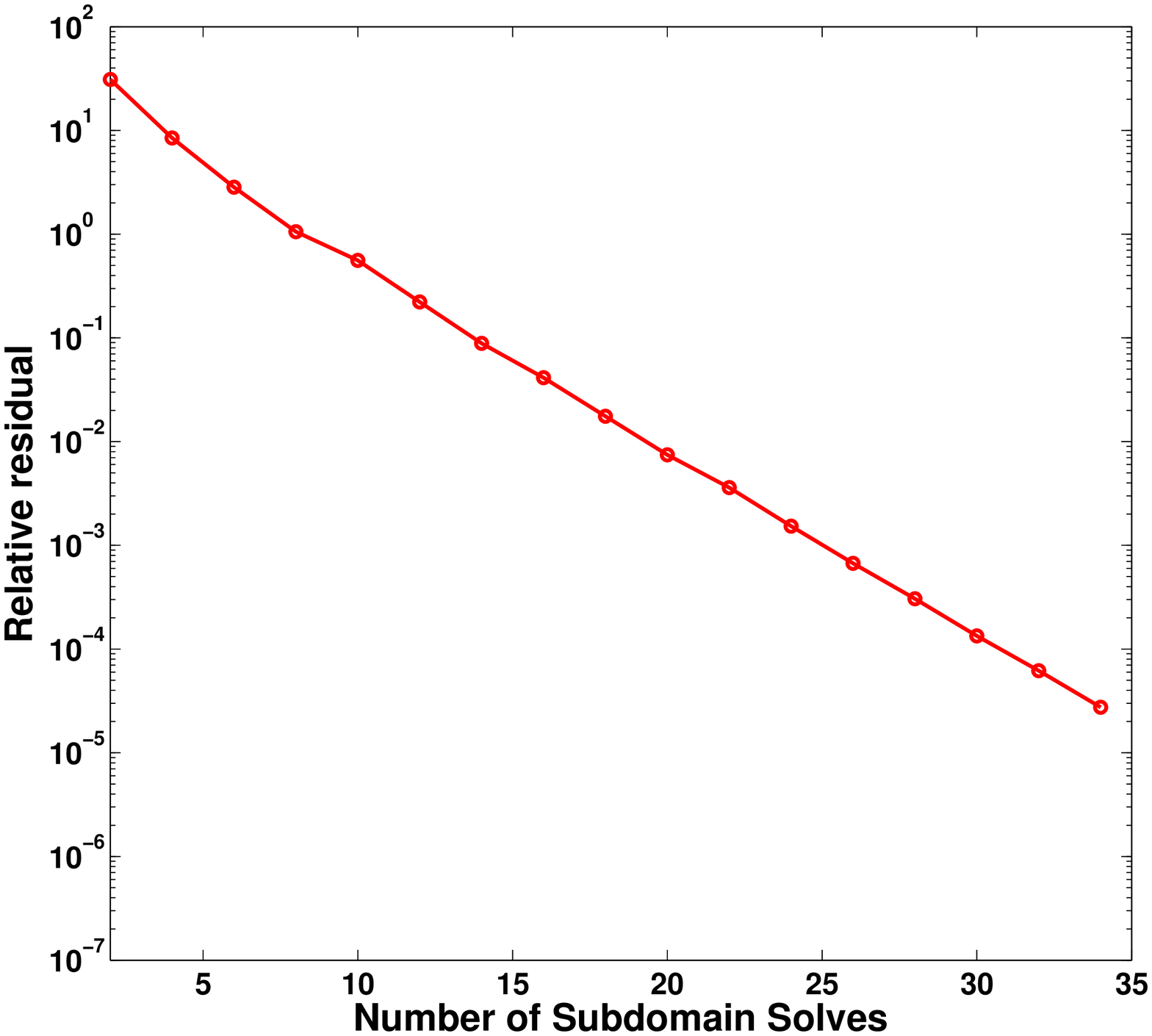} 
\end{minipage} \hspace{10pt}
\begin{minipage}{0.45 \linewidth}
\includegraphics[scale=0.27]{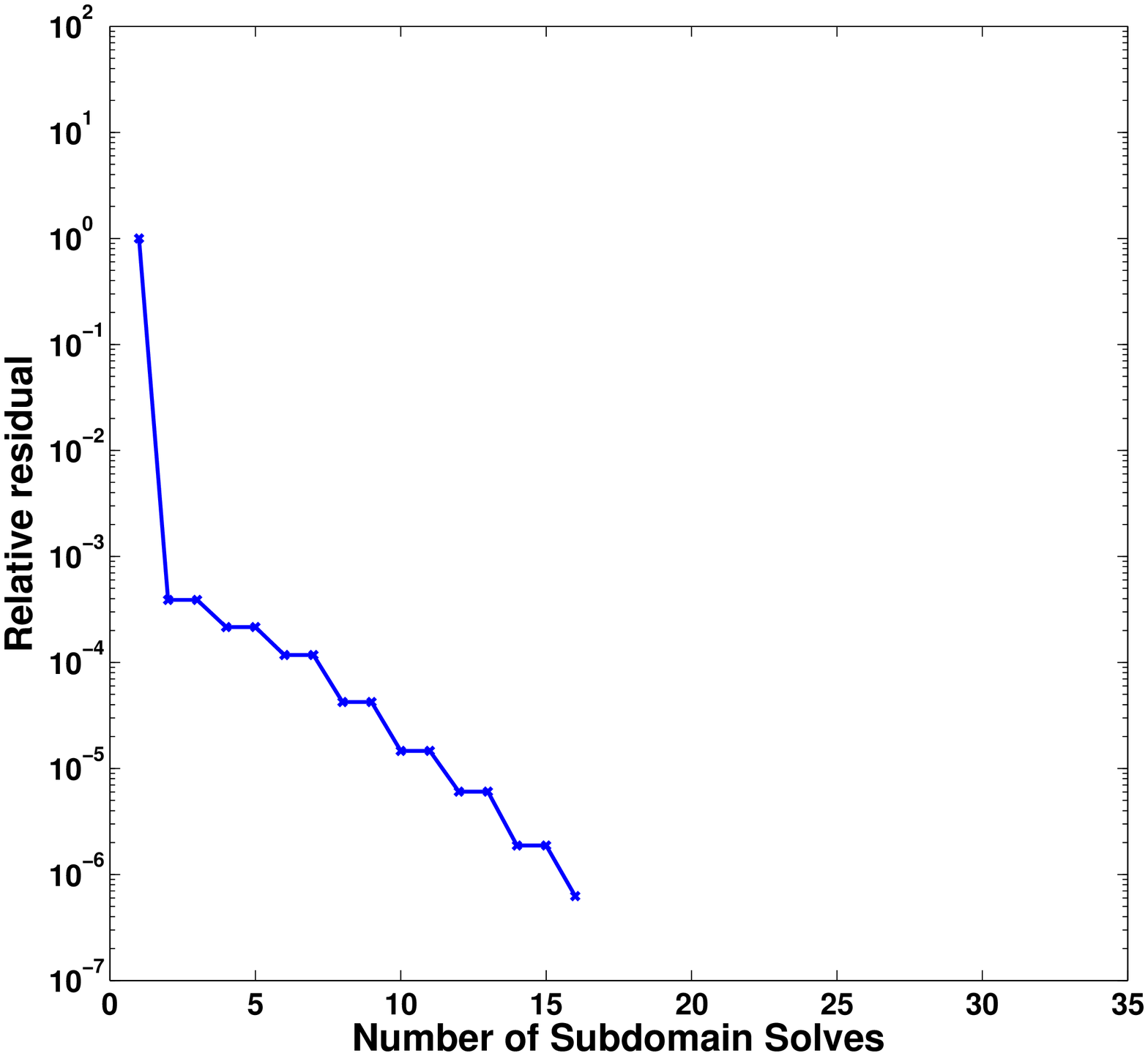} 
\end{minipage} 
\caption{The relative residuals in logarithmic scales using GMRES for Method~1 (on the left) and Method~2
         (with Opt.~2) (on the right).} 	
\label{Fig:test2relres} \vspace{-0.8cm}
\end{figure} 
%
%
% ----------------------------------
%
%    SECTION 7: Conclusion
%	
% ----------------------------------
%
%
\section{Conclusion}
We have given mixed formulations \All{for} two different interface problems for the diffusion equation, one
using the time-dependent Steklov-Poincar\'e operator and the other using OSWR with Robin transmission conditions
on the space-time interfaces between subdomains. The subdomain problem with Robin boundary conditions is
proved to be well-posed. A convergence proof of the OSWR algorithm in mixed form
is also given. Nonconforming time grids are considered and a suitable projection in time is employed
to exchange information \All{between subdomains} on the space-time interface. Numerical results for 2D problems
using mixed finite elements (with the lowest order Raviart-Thomas spaces on rectangles) \JR{for discretization in space} and the lowest order
discontinuous Galerkin method \JR{for discretization in time} are presented. We have analyzed numerically the performance of the two methods
for \JR{two} test cases, \JR{one} academic with two subdomains \JR{and one more} realistic with several subdomains.
We have observed that both methods handle \JR{well} the heterogeneity and nonconforming time grids, \JR{both} efficiently
\JR{preserving} the solution's accuracy in time. The two methods are also well-adapted for the simulation
of \JR{diffusive contaminant transport in and around a repository} with a special geometry and long time computations. In particular, for Method~2
we have shown that an adapted optimization technique to compute the optimized parameters is necessary \JR{if Jacobi iteration is used}.
We have pointed out the \JR{possible advantage for efficiency} of using time windows for problems with long time interval.
\JR{Work underway addresses} the coupling between advection and diffusion using operator splitting \MK{as~well~as} nonmatching grids in space. 
%
%
%% BIBLIOGRAPHY %%%%%%%%%%%%%%%%%%%%%%%%%%%%%%%%%%%%%%%%%%%%%

% At first stage: pdf file for Review. Use BiBTeX. It will generate a .bbl file thanks to the two following lines:
\bibliographystyle{siam}
\bibliography{DiffArticle} % Update this line according to your .bib filename

\end{document}